\documentclass[a4paper]{amsart}

\usepackage[utf8]{inputenc}			%for utf8 characters in pdfLatex
\usepackage[T1]{fontenc}			%for accents encoding
\usepackage[british]{babel} 		%british english

\usepackage{
amsmath,			%\DeclareMathOperator, etc
amsthm,				%\newtheorem, etc
amssymb,			%amsfonts, \mathbb, \cap, \cup, etc
latexsym,			%\unlhd, \sqsubset, \leadsto, \Join, etc
xfrac,				%lateral fractions \sfrac
mathtools,			%\coloneqq, \mathclap, \prescript, etc
bbm,				%several blackboard fonts
enumitem,			%enumeration options (alternative enumerate)
caption,			%for caption management \caption, \captionsetup
xcolor,				%for \color and \definecolor
tikz,				%for drawing graphs and diagrams
microtype,			%microtyping (eliminates most badbox errors)
hyphenat,			%for \hyp, allows hyphenating words with -
hyperref,			%hyperlinks 
cleveref,			%for \cref
}
\usepackage[indexonlyfirst,automake]{glossaries-extra}    %for \gls 
 
\usepackage[mathscr]{eucal}

\newtheorem{theo}{Theorem}[section]
\newtheorem{lem}[theo]{Lemma}
\newtheorem{coro}[theo]{Corollary}
\newtheorem{que}[theo]{Question}

\theoremstyle{definition}
\newtheorem{defi}[theo]{Definition}%[section]
\newtheorem{ex}[theo]{Example}
\newtheorem{rmk}[theo]{Remark}
\newtheorem*{nota}{Notation}

\crefname{theo}{Theorem}{Theorems}
\crefname{lem}{Lemma}{Lemmas}
\crefname{coro}{Corollary}{Corollaries}
\crefname{que}{Question}{Questions}
\crefname{defi}{Definition}{Definitions}
\crefname{ex}{Example}{Examples}
\crefname{rmk}{Remark}{Remarks}
\crefname{enumi}{}{}
\crefname{section}{Section}{Sections}

\makeglossaries
\newabbreviation{lpc}{lpc}{locally positively closed}
\newabbreviation{lpco}{LPCo}{locally positively compact}
\newabbreviation{lp type}{lp{\hyp}type}{local positive type}
\newabbreviation{lp formula}{lp{\hyp}formula}{local positive formula}
\newabbreviation{LJCP}{LJCP}{local joint continuation property}
\newabbreviation{lp logic topology}{lp{\hyp}logic topology}{local positive logic topology}

\renewcommand*\qedsymbol{\setbox0=\hbox{\quad\footnotesize{\normalfont Q.E.D.}}\kern\wd0 \strut \hfill \kern-\wd0 \box0}

\newcommand*\N{\mathbb{N}}
\newcommand*\Z{\mathbb{Z}}
\newcommand*\Q{\mathbb{Q}}
\newcommand*\R{\mathbb{R}}
\newcommand*\C{\mathbb{C}}

\newcommand*\tbullet{\text{\raisebox{0.25pt}{\scalebox{0.6}{$\bullet$}}}}
\newcommand*\sth{:}

\newcommand*\map{\colon}

\catcode`@=11 \def\rfrac#1#2{\def\rf@num{#1}\def\rf@den{#2}\mathpalette\rfr@c\relax} \def\rfr@c#1#2{\raise.5ex\hbox{$\mathsurround=0pt #1\rf@num$}\kern-0.1em/\kern-0.1em\lower.25ex\hbox{$\mathsurround=0pt #1\rf@den$}}\catcode`@=12

\newcommand*\Ima{{\mathrm{Im}}}
\newcommand*\tp{{\mathrm{tp}}}
\newcommand*\inc{{\iota}}
\newcommand*\id{{\mathrm{id}}}

\newcommand*\For{{\mathrm{For}}}
\newcommand*\LFor{{\mathrm{LFor}}}
\newcommand*\qftp{{\mathrm{qftp}}}
\newcommand*\ltp{{\mathrm{ltp}}}
\newcommand*\premodels{\mathrel{\mathord{\models}_\star}}
\newcommand*\premodelspc{\mathrel{\mathord{\models}^{\mathrm{pc}}_\star}}
\newcommand*\Aut{{\mathrm{Aut}}}
\newcommand*\dd{{\mathrm{d}}}
\newcommand*\lang{{\mathtt{L}}}
\newcommand*\theory{{\mathtt{T}}}
\newcommand*\Th{{\mathrm{Th}}}
\newcommand*\Stone{\mathrm{S}}
\newcommand*\loc{{\mathrm{loc}}}
\newcommand*\pc{{\mathrm{pc}}}

\newcommand*\cf{{\mathrm{cf}}}
\newcommand*\Com{{\mathrm{Com}}}
\newcommand*\retractor{\mathfrak{C}}
\newcommand*\proj{{\mathrm{proj}}}
\newcommand*\Dom{\mathrm{Dom}}
\newcommand*\Dtt{\mathtt{D}}
\newcommand*\Sorts{\mathtt{S}}
\newcommand*\BoundConst{\mathtt{B}}
\newcommand*\Bound{\mathrm{B}}
\newcommand*\ev{\mathrm{ev}}
\newcommand*\toprel{\mathrel{\top}}
\newcommand*\Mod{\mathcal{M}}
\newcommand*\inftesimal{\mathfrak{g}}
\newcommand*\Gscr{\mathscr{G}}
\newcommand*\cl{{\mathrm{cl}}}

\title{Retractors in local positive logic}

\author{Arturo Rodr\'{\i}guez Fanlo} 
\thanks{Rodr\'{\i}guez Fanlo was supported by the Israel Academy of Sciences and Humanities \& Council for Higher Education Excellence Fellowship Program for International Postdoctoral Researchers and partially supported by STRANO PID2021-122752NB-I0.}

\address{Autonomous University of Madrid, Ciudad Universitaria de Cantoblanco, 28049, Madrid, España.}
\email[A.~Rodr\'{\i}guez Fanlo]{arturo.rodriguez@uam.es}

\author{Ori Segel}

\address{Einstein Institute of Mathematics, Hebrew University of Jerusalem, 91904, Jerusalem, Israel.}
\email[O.~Segel]{Ori.Segel@mail.huji.ac.il}


\begin{document}
\begin{abstract}
We study type spaces and saturation for local positive logic.
\end{abstract}
\maketitle
\section*{Introduction}
This paper continues the work initiated by the authors on the development of local positive logic \cite{rodriguez2024completeness} and motivated by \cite[Section 2]{hrushovski2022lascar}. In this paper we study type spaces and saturated models (here called \emph{retractors}) for this new logic.

Types for local positive logic were not really discussed in \cite{hrushovski2022lascar}. Here, we cover the study of types for local positive logic in \cref{s:types}. The right notion of types for local positive logic seems to be \emph{local positive types} (\cref{d:local positive type}). For this notion we find in \cref{l:boundedly satisfiable} the right version of finite satisfiability. It is worth noting that for pointed variables (for instance, in the one{\hyp}sort case previously studied by Hrushovski) local positive types can be replaced by positive types (\cref{l:partial positive types}). We also study the topology of spaces of local positive types and adapt the definitions of Hausdorffness and semi{\hyp}Hausdorffness for local positive logic. We conclude giving characterisations of some completeness properties in terms of local positive types (\cref{l:completeness via compatibility,l:local joint continuation property via compatibility}).

In his paper, Hrushovski mostly focuses on the study of \emph{retractor spaces}, being \cite[Proposition 2.1]{hrushovski2022lascar} his main result on local positive logic. We continue his study, going deeper, in \cref{s:retractors}. We start the section explaining the different equivalent definitions of retractors (\cref{t:retractors}) and also the first existence and uniqueness results (\cref{t:relative retractor}). In \cref{s:local positive logic topologies}, we prove the Universality Lemma (\cref{t:universality lemma}), which is a stronger version of \cite[Lemma 2.4]{hrushovski2022lascar} and generalises \cite[Lemma 2.24]{segel2022positive}. In \cref{s:local positive compactness}, we recall from \cite{hrushovski2022lascar} the equivalence of the existence of retractors and local positive compactness (called \emph{primitive positive compactness}, ppC, in \cite{hrushovski2022lascar}). After that, we discus some consequences of this equivalence under different completeness properties. In \cref{s:local retractors and non-local universal models}, we discus the relationship between local retractors and non{\hyp}local homomorphism universal models. 

The study of the group of automorphisms of the core space was another fundamental topic in \cite{hrushovski2022definability,hrushovski2022lascar}. We generalise this study to retractors in arbitrary local positive logics in \cref{s:automorphisms of retractor}. In particular, we generalise several results from \cite{hrushovski2022definability} that were not studied in \cite{hrushovski2022lascar}.

This is the second of a series of papers which aims to extend several results of \cite{hrushovski2022lascar}. In subsequent papers we will study definability patterns for local positive logic, extending \cite{hrushovski2022definability,segel2022positive}. Finally, we will apply our results to study hyperdefinable approximate subgroups and rough approximate subgroups. 

\section{Preliminaries in local positive logic}
We start by recalling some basic definitions about local positive logic. We will follow the terminology and notation of \cite{rodriguez2024completeness}. For preliminaries in positive logic we refer to \cite{yaacov2007fondements,poizat2008positive,yaacov2003positive}.

\subsection{Local languages}
A \emph{local (first{\hyp}order) language} $\lang$ is a first{\hyp}order language\footnote{As it is common practice in positive logic, we consider that every first{\hyp}order language contains a binary relation of equality $=$ for each sort and two $0${\hyp}ary relations $\bot$ and $\top$, where $\top$ is always true and $\bot$ is always false.} with the following additional data:
\begin{enumerate}[label={\rm{(\roman*)}}, wide]
\item It is a relational language with constants, i.e. it has no function symbols except for constant symbols.
\item It has a distinguished partial ordered commutative monoid $\Dtt^s{\coloneqq} (\Dtt^s,\dd^0,\ast_{s},\prec_s)$ of binary relation symbols for each single sort $s$. The binary relation symbols in $\Dtt^s$ are the \emph{locality relation symbols} on the sort $s$. The identity $\dd^0$ of $\Dtt^s$ is the equality symbol on sort $s$ and it is also the least element for $\prec_s$. 
\item There is a function $\BoundConst$ assigning to any pair $c_1\, c_2$ of constant symbols of the same sort a locality relation symbol $\BoundConst_{c_1,c_2}$ on that sort. We call $\BoundConst$ the \emph{bound function} and $\BoundConst_{c_1,c_2}$ the \emph{bound} for $c_1$ and $c_2$.
\end{enumerate}

We call $(\Dtt,\BoundConst)$ the \emph{local signature} of $\lang$. A \emph{sort} is a tuple of single sorts; its \emph{arity} is its length. We write $\Sorts$ for the set of single sorts. A \emph{variable} $x$ is a set of single variables; its \emph{arity} is its size. The \emph{sort} of a variable $x$ is the sort $s=(s_{x_i})_{x_i\in x}$ where $s_{x_i}$ is the sort of $x_i$. 

We say that a local language $\lang$ is \emph{pointed} if it has at least one constant symbol for every sort. We say that a set of parameters (new constants) $A$ is \emph{pointed} for $\lang$ if the expansion $\lang(A)$ is pointed. We say that a variable $x$ of sort $s$ is \emph{pointed} in $\lang$ if there is at least one constant symbol for every sort that does not appear in $s$. We say that a sort $s$ is \emph{pointed} if there is at least one constant symbol for every sort that does not appear in $s$. 

A formula is \emph{positive} if it is built from atomic formulas by conjunctions, disjunctions and existential quantifiers; let $\For_+$ denote the set of positive formulas. A formula is \emph{negative} if it is the negation of a positive formula; let $\For_-$ denote the set of negative formulas. Let $\varphi$ be a formula, $x$ a single variable, $\dd$ a locality relation on the sort of $x$ and $t$ a term of the same sort such that $x$ does not appear in $t$ (i.e. a constant or a variable different to $x$). The \emph{local existential quantification of $\varphi$ on $x$ in $\dd(t)$} is the formula $\exists x \in  \dd(t)\mathrel{} \varphi\coloneqq \exists x\mathrel{} \dd(x,t)\wedge \varphi$. A formula is \emph{local} if it is built from atomic formulas using negations, conjunctions and local existential quantification; let $\LFor$ denote the set of local formulas. 
A formula is a \emph{\gls{lp formula}} if it is built from atomic formulas using conjunctions, disjunctions and local existential quantification; let $\LFor_+$ denote the set of \glspl{lp formula}. A formula is \emph{local negative} if it is the negation of a \gls{lp formula}.%; let $\LFor_-$ denote the set of local negative formulas. 

\subsection{Local structures}
A \emph{local $\lang${\hyp}structure} is an $\lang${\hyp}structure $M$ satisfying the following locality axioms:
\begin{enumerate}[label={\rm{(A{\arabic*})}}, ref={\rm{A{\arabic*}}}, wide]
\item \label{itm:axiom 1} Every locality relation symbol is interpreted as a symmetric binary relation.
\item \label{itm:axiom 2} Any two locality relation symbols $\dd_1$ and $\dd_2$ on the same sort with $\dd_1\preceq \dd_2$ satisfy $\dd^M_1\subseteq \dd^M_2$.
\item \label{itm:axiom 3} Any two locality relation symbols $\dd_1$ and $\dd_2$ on the same sort satisfy $\dd^M_1\circ \dd^M_2\subseteq (\dd_1\ast \dd_2)^M$.
\item \label{itm:axiom 4} Any two constant symbols $c_1$ and $c_2$ of the same sort satisfy $\BoundConst_{c_1,c_2}(c_1,c_2)$.
\item \label{itm:axiom 5} Any two single elements $a$ and $b$ of the same sort satisfy $\dd(a,b)$ for some locality relation $\dd$ on that sort. 
\end{enumerate}

For any element $a$ and locality relation $\dd$ on its sort, the \emph{$\dd${\hyp}ball at $a$} is the set $\dd(a)\coloneqq \{b\sth M\models \dd(a,b)\}$.

Expansions and reducts of local structures are defined in the usual way. 

\begin{rmk} The first four locality axioms could be easily expressed by first{\hyp}order universal axiom schemes. We write $\theory_\loc(\lang)$ for the universal first{\hyp}order theory given by axioms \cref{itm:axiom 1,itm:axiom 2,itm:axiom 3,itm:axiom 4} --- we omit $\lang$ if it is clear from the context.
\end{rmk}

We naturally restrict satisfaction to local structures. We write $M\models_\lang \Gamma$ to say that $M$ is an $\lang${\hyp}local model of $\Gamma$. We write $\Mod(\Gamma/\lang)$ for the class of $\lang${\hyp}local models of $\Gamma$. We say that $\Gamma$ is \emph{$\lang${\hyp}locally satisfiable} if it has an $\lang${\hyp}local model. We say that $\Gamma$ is \emph{finitely $\lang${\hyp}locally satisfiable} if every finite subset of $\Gamma$ is $\lang${\hyp}locally satisfiable. We typically omit $\lang$ if it is clear from the context.

\begin{nota} We write $\premodels$ for satisfaction for arbitrary structures and $\Mod_\star(\Gamma)$ for the class of arbitrary models of $\Gamma$.
\end{nota}

A \emph{negative $\lang${\hyp}local theory} $\theory$ is a set of negative sentences closed under local implication. We typically omit $\lang$ if it is clear from the context.  The \emph{associated positive counterpart} $\theory_+$ is the set of positive sentences whose negations are not in $\theory$. In other words, a positive formula $\varphi$ is in $\theory_+$ if and only if there is a local model $M\models_\lang \theory$ satisfying $\varphi$; we say in that case that $M$ is a \emph{witness of $\varphi\in\theory_+$}. We write $\theory_\pm\coloneqq \theory\wedge\theory_+$\footnote{We write $\Gamma\wedge\Gamma'\coloneqq\Gamma\cup \Gamma'$ for sets of formulas.}. 

For a local structure $M$, we write $\Th_-(M)$ for its negative theory, i.e. the set of negative sentences satisfied by $M$. For a subset $A$, we write $\Th_-(M/A)$ for the theory of the expansion of $M$ given by adding parameters for $A$. Note that $\Th_-(M)$ must be a local theory when $M$ is local.

\subsection{Positive closedness}
A homomorphism $f\map A\to B$ between two structures is a \emph{positive embedding} if $B\premodels \varphi(f(a))$ implies $A\premodels\varphi(a)$ for any positive formula $\varphi(x)$ and $a\in A^x$. A \emph{positive substructure} is a substructure $A\leq B$ such that the inclusion $\inc\map A\to B$ is a positive embedding; we denote it by $A\preceq^+B$\footnote{We use $\preceq^+$ imitating the notation for substructures and elementary substructures: we write $A\leq B$ to denote that $A$ is a substructure of $B$ and $A\preceq B$ to denote that $A$ is an elementary substructure of $B$.}. 

We say that a structure $A$ is \emph{positively closed} for a class of structures $\mathcal{K}$ if every homomorphism $f\map A\to B$ with $B\in\mathcal{K}$ is a positive embedding. We say that $M$ is a \emph{\gls{lpc} model} of $\Gamma$, written $M\models^\pc_\lang\Gamma$, if $M$ is a local model of $\Gamma$ positively closed for the class $\Mod(\Gamma/\lang)$ of local models of $\Gamma$; we write $\Mod^\pc(\Gamma/\lang)$ for the class of \gls{lpc} models of $\Gamma$. We say that $M$ is \gls{lpc} if it is a \gls{lpc} model of $\Th_-(M)$. We typically omit $\lang$ if it is clear from the context. Given two subsets $\Gamma_1$ and $\Gamma_2$ of first{\hyp}order sentences, we write $\Gamma_1\models^\pc_\lang\Gamma_2$ if $N\models_\lang \Gamma_2$ for any $N\models^\pc_\lang\Gamma_1$. 

\begin{nota} We say that $M$ is a \emph{(non{\hyp}local) positively closed model of $\Gamma$}, written $M\premodelspc\Gamma$, if $M$ is a model of $\Gamma$ and is positively closed for the class $\Mod_\star(\Gamma)$. We say that $\Gamma_1\premodelspc\Gamma_2$ if $N\premodels\Gamma_2$ for any $N\premodelspc\Gamma_1$ 
\end{nota}

\subsection{Completeness}
Let $\theory$ be a locally satisfiable negative local theory:
\begin{enumerate}[wide]
\item[\textrm{(I)}] $\theory$ is \emph{irreducible} if, for any two positive formulas $\varphi(x)$ and $\psi(y)$ with $x$ and $y$ disjoint finite variables of the same sort, we have
\[\exists x\mathrel{}\varphi(x)\in\theory_+\text{ and }\exists y\mathrel{}\psi(y)\in\theory_+\Rightarrow \exists x y\mathrel{} (\varphi(x)\wedge\psi(y))\in\theory_+.\]
%\item[\textrm{(UI)}] Let $\dd=(\dd_s)_{s\in\Sorts}$ be a choice of a locality relation symbols for each single sort. and write $\dd_s=\dd_{s_1}\times\cdots\times\dd_{s_n}$ for $s=(s_1,\ldots,s_n)$. 
%$\theory$ is \emph{$\dd${\hyp}irreducible} if, for any two positive formulas $\varphi(x)$ and $\psi(y)$ with $x$ and $y$ disjoint finite variables of the same sort $s$, 
%\[\exists x\mathrel{}\varphi(x)\in\theory_+\text{ and }\exists y\mathrel{}\psi(y)\in\theory_+\Rightarrow\exists x y\mathrel{} \dd_s(x,y)\wedge\varphi(x)\wedge\psi(y)\in\theory_+.\] 
%We say that $\theory$ is \emph{uniformly irreducible} if it is $\dd${\hyp}irreducible for some $\dd$.
\item[\textrm{(LJCP)}] $\theory$ has the \emph{\gls{LJCP}} if for any two local models $A,B\models_\lang \theory$ there are homomorphisms $f\map A\to M$ and $g\map B\to M$ to a common local model $M\models_\lang \theory$.
\item[\textrm{(wC)}] $\theory$ is \emph{weakly complete} if $\theory=\Th_-(M)$ for some $M\models_\lang\theory$. Equivalently, $\theory=\Th_-(M)$ for some $M\models^\pc_\lang\theory$.
\item[\textrm{(C)}] $\theory$ is \emph{complete} if $\theory=\Th_-(M)$ for every $M\models^\pc_\lang\theory$.
\end{enumerate}

\subsection{Denials}
Let $\varphi$ be a positive formula. A \emph{denial} of $\varphi$ is a positive formula $\psi$ such that $\theory\models_\lang \neg \exists x\mathrel{} (\varphi\wedge\psi)$. We write $\theory\models_\lang\varphi\perp\psi$ or $\varphi\perp \psi$ if $\theory$ is clear from the context. An \emph{approximation} to $\varphi$ is a positive formula $\psi$ such that every denial of $\psi$ is a denial of $\varphi$; we write $\theory\models_\lang \varphi\leq \psi$ or $\varphi\leq \psi$ if $\theory$ is clear from the context. We say that two positive formulas $\psi$ and $\varphi$ are \emph{complementary} if $\varphi$ approximates every denial of $\psi$; we write $\theory\models_\lang\varphi\toprel\psi$ or $\varphi\toprel \psi$ if $\theory$ is clear from the context. A \emph{full system of denials} of $\varphi$ in $\theory$ is a family $\Psi$ of denials of $\varphi$ such that, for any $M\models^\pc_\lang\theory$ and any $a\in M^x$, either $M\models_\lang\varphi(a)$ or $M\models_\lang\psi(a)$ for some $\psi\in \Psi$. A \emph{full system of approximations} of $\varphi$ in $\theory$ is a family $\Psi$ of approximations of $\varphi$ such that, for any $M\models^\pc_\lang\theory$ and any $a\in M^x$, if $M\models_\lang \psi(a)$ for all $\psi\in \Psi$, then $M\models_\lang\varphi(a)$.

\section{Local positive types} \label{s:types}
From now on, unless otherwise stated, we fix a local language $\lang$ and a locally satisfiable negative local theory $\theory$.

\begin{defi} \label{d:local positive type}
Let $x$ be a variable. A \emph{partial \gls{lp type} on $x$ of $\theory$} is a subset $\Gamma\subseteq \LFor^x_+(\lang)$ such that $\Gamma$ is \emph{locally satisfiable} in $\theory$, i.e. there is a local model $M\models_\lang \theory$ and an element $a\in M^x$ such that $M\models_\lang \Gamma(a)$. A \gls{lp type} of $\theory$ is a maximal partial \gls{lp type} of $\theory$. As usual, we omit $x$ if it is clear from the context.

Let $M$ be a local $\lang${\hyp}structure, $A$ a subset and $a\in M^x$. The \emph{partial \gls{lp type}} of $a$ in $M$ over $A$ is the set 
\[\ltp^M_+(a/A)\coloneqq\{\varphi\in\LFor^x_+(\lang(A))\sth M\models_\lang \varphi(a)\}.\]

We omit $M$ if it is clear from the context, and we omit $A$ if $A=\emptyset$. 
\end{defi}

\begin{rmk} It is a well known fact in positive logic that the partial positive type of an element is not maximal in general. Similarly, in the local logic context, $\ltp^M_+(a)$ is not maximal. Indeed, it suffices to find a homomorphism $h\map M\to N$ such that $N\models_\lang\varphi(h(a))$ with $\varphi\in \LFor_+(\lang)$ while $M\not\models_\lang\varphi(a)$. Then, $\ltp^M_+(a)\subsetneq \ltp^N_+(h(a))$. \end{rmk}

\begin{lem}\label{l:partial local positive types} Let $\Gamma\subseteq\LFor^x_+(\lang)$ be a subset of \glspl{lp formula}. Then, $\Gamma$ is a partial \gls{lp type} of $\theory$ if and only if there are $M\models^\pc_\lang\theory$ and $a\in M^x$ such that $M\models_\lang\Gamma(a)$. 
\end{lem}
\begin{proof} Suppose $\Gamma$ is a partial \gls{lp type}. By definition, $\Gamma$ is locally satisfiable. Then, there is $N\models_\lang\theory$ and $b\in N^x$ such that $N\models_\lang \Gamma(b)$. By \cite[Theorem 2.9]{rodriguez2024completeness}, there is a homomorphism $f\map N\to M$ to $M\models^\pc_\lang\theory$. Thus, as homomorphisms preserve satisfaction of positive formulas, we conclude that $M\models_\lang\Gamma(a)$ for $a=f(b)$, so $\Gamma$ is realised in a \gls{lpc} model of $\theory$.
\end{proof}
\begin{rmk} \label{r:lowenheim skolem realising types} Using downwards L\"{o}wenheim{\hyp}Skolem Theorem, we can add that $|M|\leq |\lang|+|x|$ in \cref{l:partial local positive types}.
\end{rmk}
\begin{lem} \label{l:positively closed and local positive types} Let $M\models_\lang\theory$. Then, $M$ is a \gls{lpc} model of $\theory$ if and only if $\ltp^M_+(a)$ is a \gls{lp type} of $\theory$ for every pointed tuple $a$ in $M$.
\end{lem}
\begin{proof} Suppose that $\ltp^M_+(c)$ is a \gls{lp type} of $\theory$ for every pointed $c$ in $M$. Let $h\map M\to N$ be a homomorphism to a local model $N\models_\lang\theory$. Suppose $N\models_\lang\varphi(h(a))$ for some $\varphi\in\For^x_+(\lang)$ and $a\in M^x$. We want to show that $M\models_\lang\varphi(a)$. 

Without loss of generality, $\varphi(x)=\exists y\mathrel{} \phi(x,y)$ where $\phi(x,y)$ is quantifier free and positive. Then, $N\models_\lang\phi(a,c)$ for some $c\in N^y$. Take $b\in M^y$ arbitrary and extend $ab$ to a pointed $a'$ arbitrarily. As $N$ is a local model, we have that $N\models_\lang \dd(c,h(b))$ for some locality relation symbol $\dd$. Consider $\varphi'(x,y)=\exists z \in \dd(y)\mathrel{} \phi(x,z)\in\LFor^{xy}_+(\lang)$ where $z$ is of the same sort as $y$. Then, we have that $N\models_\lang\varphi'(h(a,b))$, so $\varphi'(x,y)\in\ltp^{N}_+(h(a'))$. 

Since homomorphisms preserve satisfaction of positive formulas, $\ltp^M_+(a')\subseteq \ltp^{N}_+(h(a'))$. As $\ltp^M_+(a')$ is maximal among partial \glspl{lp type} of $\theory$, it follows that $\ltp^M_+(ab)=\ltp^{N}_+(h(ab))$, so we get $M\models_\lang\varphi'(a,b)$. In particular, $M\models_\lang\phi(a,c')$ for some $c'$ in $M$, so $M\models_\lang\varphi(a)$. Since $a$, $h$, $N$ and $\varphi$ are arbitrary, we conclude that $M$ is a \gls{lpc} model of $\theory$.

On the other hand, suppose $M$ is a \gls{lpc} model and $a\in M^x$ is pointed. Take $\varphi\in\LFor^x_+(\lang)$ such that $M\not\models_\lang\varphi(a)$. By \cite[Lemma 2.12]{rodriguez2024completeness}, we have $M\models_\lang\psi(a)$ for some $\psi\in\LFor^x_+(\lang)$ such that $\theory\models_\lang\psi\perp\varphi$. Thus, there is $\psi\in \ltp^M_+(a)$ such that $\{\varphi,\psi\}$ is not locally satisfiable in $\theory$. In other words, $\ltp^M_+(a)\wedge\varphi$ is not locally satisfiable in $\theory$. As $\varphi$ is arbitrary, we conclude $\ltp^M_+(a)$ is a \gls{lp type} of $\theory$.
\end{proof}
\begin{ex} The pointedness hypothesis is crucial in the previous result: $\ltp^N_+(a)$ does not need to be maximal for $a\in N^x$ with $N\models^\pc_\lang\theory$ when $a$ is not pointed. For instance, consider the language $\lang$ with two sorts $\Sorts=\{1,2\}$, locality relations $\Dtt^1=\{\dd_n\sth n\in\N\}$ and $\Dtt^2=\{=,\toprel\}$, unary relations $\{P_n\}_{n\in\N}$ of sort $1$, and binary relations $\{Q_n\}_{n\in\N}$ of sort $(1,2)$. Let $Z$ be the local $\lang${\hyp}structure with $1${\hyp}sort $\Z$, $2${\hyp}sort a singleton and interpretations $Z\models \dd_n(x,y)\Leftrightarrow |x-y|\leq n$, $Z\models P_n(x)\Leftrightarrow x\geq n$ and $Z\models Q_n(x,y)\Leftrightarrow x\leq-n$. By \cite[Lemma 2.6]{rodriguez2024completeness}, $Z$ is a \gls{lpc} model. %Indeed, balls are finite and $\theta_n(x)\coloneqq\exists y z \mathrel{}(P_n(x)\wedge \dd_n(x,y)\wedge Q_0(y,z))$ and $\theta_{-n}(x)\coloneqq\exists y z\mathrel{}(P_0(y)\wedge\dd_n(x,y)\wedge Q_n(x,z))$ for $n\in\N$ have the property that $\theta_n(k)$ if and only if $k=n$.
Now, $\ltp_+(-1)\subsetneq \ltp_+(0)$. Indeed, a \gls{lp formula} $\varphi(x)$ with $x$ variable of sort $1$ cannot contain variables of sort $2$. Hence, $\varphi$ is a \gls{lp formula} in the reduct $\lang'$ of $\lang$ given by removing $Q_n$ for $n\in\N$. As the map $f\map Z\to Z$ given by $k\mapsto k+1$ for $k\in\Z$ is obviously an $\lang'${\hyp}homomorphism we get that $\ltp_+(-1)\subseteq \ltp_+(0)$. Now, $P_0(x)\in \ltp_+(0)$ while $P_0(x)\notin \ltp_+(-1)$.
\end{ex}
\begin{coro} \label{c:local positive types} Let $x$ be a pointed variable and $\Gamma\subseteq\LFor^x_+(\lang)$ a subset of \glspl{lp formula}. Then, $\Gamma$ is a \gls{lp type} of $\theory$ if and only if $\Gamma=\ltp^M_+(a)$ for some $M\models^\pc_\lang\theory$ and $a\in M^x$.
\end{coro}

As an immediate consequence of \cite[Theorem 2.9]{rodriguez2024completeness} and \cref{c:local positive types}, we can show that every partial \gls{lp type} on a pointed variable extends to a \gls{lp type}. Naturally, the same is still true for arbitrary variables, but we require an alternative argument relying on the notion of \emph{bounded satisfiability}. As \cref{l:boundedly satisfiable} below shows, this notion is the right replacement of finite satisfiability in local logic. In particular, boundedness is the necessary condition to apply compactness in local logic. 

\begin{defi}\label{d:bound} Let $x=\{x_i\}_{i\in N}$ be a variable. Let $I$ be the subset of indexes $i\in N$ such that there is a constant symbol on the sort of $x_i$ and $J$ the set of pairs of indexes $(i,j)\in N\times N$ such that $x_i$ and $x_j$ are on the same sort. A \emph{bound} of $x$ in $\theory$ is a partial atomic type of the form \[\Bound(x)\coloneqq \bigwedge_{i\in I}\dd_i(x_i,c_i) \wedge\bigwedge_{(i,j)\in J} \dd_{i,j}(x_i,x_j)\] where $c_i$ is a constant on the sort of $x_i$ for $i\in I$, $\dd_i$ is a locality relation on the sort of $x_i$ for $i\in I$ and $\dd_{i,j}$ is a locality relation on the common sort of $x_i$ and $x_j$ for $(i,j)\in J$.
\end{defi}

As a relevant example, we point out the following special case. We say that a variable $x$ is \emph{duplicate free} if $x$ has no duplicated sorts (counting constants), i.e. it contains no single variable with the same sort as a constant or another single variable in $x$. For instance, a variable is duplicate free and pointed if and only if it is \emph{minimal pointed} (pointed without proper pointed subvariables). When $x$ is duplicate free, the only bound of $x$ is $\toprel$.

\begin{defi}\label{d:bounded satisfiable} A subset of formulas $\Gamma(x)\subseteq\For^x(\lang)$ is \emph{boundedly satisfiable} in $\theory$ if there is a bound $\Bound(x)$ such that $\Gamma(x)\wedge \Bound(x)$ is finitely locally satisfiable in $\theory$. In that case, we say that $\Bound$ is a \emph{feasible bound} for $\Gamma$. 
\end{defi}

\begin{rmk} When $\Gamma(x)\models_\lang \Bound(x)$ for some bound $\Bound(x)$, we have that $\Gamma(x)$ is boundedly satisfiable if and only if it is finitely locally satisfiable. In particular, when $x$ is duplicate free, as bounds on $x$ are trivial, we get that $\Gamma(x)$ is boundedly satisfiable if and only if it is finitely locally satisfiable.
\end{rmk}

\begin{lem} \label{l:boundedly satisfiable} 
Let $\Gamma$ be a subset of $\Pi_1${\hyp}local formulas of $\lang$. Then, $\Gamma$ is locally satisfiable in $\theory$ if and only if it is boundedly satisfiable in $\theory$. 
\end{lem}
\begin{proof} Obviously, if $\Gamma$ is locally satisfiable, then it is boundedly satisfiable as any $a$ realising $\Gamma$ gives us a feasible bound for $\Gamma$. On the other hand, suppose that $\Gamma(x)$ is boundedly satisfiable. Pick a feasible bound $\Bound(x)$ for $\Gamma(x)$. Let $\lang(\underline{x})$ be the expansion of $\lang$ given by adding new constant symbols for each variable in $x$: we take $\BoundConst_{\underline{x}_i,c_i}$ and $\BoundConst_{\,\underline{x}_i,\underline{x}_j}$ according to $\Bound$ --- for any other constant $c'$ in $\lang(\underline{x})$, we define $\BoundConst_{\underline{x}_i,c'}$ according to the structure of $\Dtt$. Now, consider $\widetilde{\theory}=\theory\wedge \Gamma(\underline{x})$. By hypothesis, $\widetilde{\theory}$ is finitely locally satisfiable. Thus, by \cite[Theorem 1.9]{rodriguez2024completeness}, there is a local $\lang(\underline{x})${\hyp}structure $M_{\underline{x}}$ satisfying $\widetilde{\theory}$. Hence, the respective $\lang${\hyp}reduct $M$ is a local model of $\theory$ and the interpretation of $\underline{x}$ realises $\Gamma$.
\end{proof}
\begin{coro} \label{c:there are types} Let $\Gamma\subseteq\LFor^x_+(\lang)$ be a partial \gls{lp type} of $\theory$. Then, there is a \gls{lp type} extending $\Gamma$.
\end{coro}
\begin{proof} As $\Gamma$ is boundedly satisfiable, pick any feasible bound $\Bound(x)$. Take $\Gamma'(x)=\Gamma(x)\wedge \Bound(x)$. Now, consider $\mathcal{K}=\{\Sigma\subseteq \LFor^x_+(\lang)\sth \Gamma'\subseteq \Sigma$ finitely locally satisfiable in $\theory\}$, which is partially ordered by $\subseteq$. For any chain $\Omega\subseteq \mathcal{K}$ we have that $\bigcup \Omega$ is an upper bound for $\Omega$ in $\mathcal{K}$. Indeed, for any $\Lambda\subseteq \bigcup \Omega$ finite, there is $\Sigma\in \Omega$ such that $\Lambda\subseteq \Sigma$, so $\bigcup\Omega$ is finitely locally satisfiable, concluding $\bigcup\Omega\in \mathcal{K}$. By Zorn{'s} Lemma, there is a maximal element $p\in \mathcal{K}$. We know that $\Gamma\subseteq p$ and $p$ is boundedly satisfiable. On the other hand, if $p\subseteq \Sigma$ with $\Sigma$ boundedly satisfiable, then $\Sigma\in \mathcal{K}$ and, by maximality of $p$, $\Sigma=p$. Hence, $p$ is a \gls{lp type} extending $\Gamma$.
\end{proof}

In general, we consider only partial \glspl{lp type}, as local compactness \cite[Theorem 1.9]{rodriguez2024completeness} applies to this kind of formulas. However, in the case of pointed variables, we can handle positive types by means of the following result.
\begin{lem} \label{l:partial positive types} Let $M$ be a local $\lang${\hyp}structure. Let $x$ be a pointed variable and $a,b\in M^x$. Then, $\ltp^M_+(a)\subseteq \ltp^M_+(b)$ if and only if $\tp^M_+(a)\subseteq \tp^M_+(b)$. 
\end{lem}
\begin{proof} Obviously, $\tp_+(a)\subseteq \tp_+(b)$ implies $\ltp_+(a)\subseteq\ltp_+(b)$. Conversely, assume $\ltp_+(a)\subseteq \ltp_+(b)$. Take $\varphi(x)\in\For^x_+(\lang)$. Without loss of generality, $\varphi(x)=\exists y\mathrel{} \phi(x,y)$ with $\phi$ quantifier{\hyp}free positive. Suppose $M\models_\lang\varphi(a)$. Then, there is $c$ such that $M\models_\lang\phi(a,c)$. As $a$ is pointed, we can find a term $t(a)$ on the sort of $c$. As $M$ is a local structure, there is a locality relation $\dd$ with $\dd(t(a),c)$. Thus, $M\models_\lang\varphi'(a)$ where $\varphi'(x)=\exists y\in \dd(t(x))\mathrel{} \phi(x,y)\in\LFor^x_+(\lang)$. As $\ltp_+(a)\subseteq \ltp_+(b)$, we conclude that $M\models_\lang\varphi'(b)$ too. Since $\varphi$ is arbitrary, we conclude that $\tp_+(a)\subseteq\tp_+(b)$.
\end{proof}

\subsection{Local positive types over parameters}
Let $M$ be a local $\lang${\hyp}structure and $A$ a subset. A \emph{(partial) \gls{lp type} of $M$ over $A$} is a (partial) \gls{lp type} of $\Th_-(M/A)$. We say that $M$ has the \emph{\gls{LJCP} over $A$} if $\Th_-(M/A)$ has the \gls{LJCP}. 

\begin{lem} \label{l:local positive types over parameters} Let $M$ be a local $\lang${\hyp}structure and $A$ a subset. Let $\Gamma\subseteq\LFor^x_+(\lang(A))$ be a subset of \glspl{lp formula} with parameters in $A$. Assume $M$ has the \gls{LJCP} over $A$. Then, $\Gamma$ is a partial \gls{lp type} of $M$ over $A$ if and only if there is an $\lang(A)${\hyp}homomorphism $f\map M_A\to N_A$ to $N_A\models^\pc_\lang\Th_-(M/A)$ realising $\Gamma$. Furthermore, suppose $x$ is pointed over $\lang(A)$. Then, $\Gamma(x)$ is a \gls{lp type} of $M$ over $A$ if and only if there is an $\lang(A)${\hyp}homomorphism $f\map M_A\to N_A$ to some $N_A\models^\pc\Th_-(M/A)$ such that $\Gamma=\ltp^{N_A}_+(a/A)$ for some $a\in N^x$. 
\end{lem}
\begin{proof} If there is an $\lang(A)${\hyp}homomorphism to a \gls{lpc} model of $\Th_-(M/A)$ realising $\Gamma$, then $\Gamma$ is obviously a partial \gls{lp type} over $A$. On the other hand, suppose $\Gamma$ is a partial \gls{lp type} over $A$. By \cref{l:partial local positive types}, there are $N'_A\models^\pc_\lang\Th_-(M/A)$ and $b\in {N'}^x_A$ such that $N'_A\models_\lang\Gamma(b)$. By assumption, $\Th_-(M/A)$ satisfies the \gls{LJCP}, so there are $\lang(A)${\hyp}homomorphisms $h_0\map N'_A\to N''_A$ and $h_1\map M_A\to N''_A$ to a common local model $N''_A\models_\lang\Th_-(M/A)$. By \cite[Theorem 2.9]{rodriguez2024completeness}, there is an $\lang(A)${\hyp}homomorphism $f'\map  N''_A\to N_A$ to $N_A\models^\pc_\lang\Th_-(M/A)$, so $f=f'\circ h_1\map M_A\to N_A$ is an $\lang(A)${\hyp}homomorphism with $N_A\models^\pc_\lang \Th_-(M/A)$. Since homomorphisms preserve satisfaction of positive formulas, we have that $N_A\models_\lang\Gamma(a)$ with $a=f'\circ h_0(b)$. 

Finally, if $\Gamma$ is a \gls{lp type} of $M$ over $A$, then $\Gamma=\ltp^N_+(a/A)$ by maximality. On the other hand, by \cref{l:positively closed and local positive types}, $\ltp^N_+(a/A)$ is a \gls{lp type} of $\Th_-(M/A)$ for any $N_A\models^\pc_\lang\Th_-(M/A)$ and $a\in N^x_A$ when $x$ is pointed.
\end{proof}
\begin{rmk} \label{r:lowenheim skolem realising types with parameters} Applying downwards L\"{o}wenheim{\hyp}Skolem we can take $|N|\leq |\lang|+|M|+|x|$ in \cref{l:local positive types over parameters}.
\end{rmk}
\begin{lem} \label{l:local positive types over parameters and denials} Let $M$ be a local $\lang${\hyp}structure and $A$ a subset. Let $p(x)$ be a partial \gls{lp type} of $M$ over $A$ on a variable $x$ pointed over $\lang(A)$. Then, $p$ is a \gls{lp type} of $M$ over $A$ if and only if, for every \gls{lp formula} $\varphi(x)$ over $A$ with $\varphi(x)\notin p$, there is a local primitive positive formula $\psi(x)$ over $A$ with $M\models_\lang\neg\exists x\mathrel{}(\psi(x)\wedge\varphi(x))$ such that $\psi(x)\in p$.
\end{lem}
\begin{proof} The right{\hyp}to{\hyp}left implication is obvious. On the other hand, suppose $p$ is a \gls{lp type} over $A$ and $\varphi(x)\notin p$. By \cref{c:local positive types}, there is $N_A\models^\pc_\lang\Th_-(M/A)$ and $c\in N^x$ such that $p(x)=\ltp^{N_A}_+(c/A)$. As $\varphi(x)\notin p$, we get that $N_A\not\models_\lang \varphi(c)$. By \cite[Lemma 2.12]{rodriguez2024completeness}, there is a local primitive positive formula $\psi(x)$ over $A$ with $M\models_\lang \psi(x)\perp\varphi(x)$ such that $N_A\models_\lang\psi(c)$, so $\psi(x)\in p$.
\end{proof}

%Let $\lang$ be a local language, $M$ a local $\lang${\hyp}structure and $A$ a subset. Let $\Gamma\subseteq\For^x(\lang(A))$ be a subset of formulas with parameters in $A$. We say that $\Gamma(x)$ is \emph{boundedly satisfiable} in $M$ over $A$ if there is a bound $\Bound(x)$ over $A$ such that $\Gamma\wedge\Bound(x)$ is finitely satisfiable in $M$.
%\begin{lem} \label{l:corollary boundedly satisfiable} Let $\lang$ be a local language, $M$ a local $\lang${\hyp}structure and $A$ a subset. Let $\Gamma(x)$ be set of $\Pi_1${\hyp}local formulas with parameters in $A$. Then, $\Gamma(x)$ is a partial $\Pi_1${\hyp}local type of $M$ on $x$ over $A$ if and only if it is boundedly finitely locally satisfiable in $M$ over $A$. 
%\begin{proof} Trivial by \cref{l:boundedly satisfiable}, noting that finitely satisfiable in $\Th_-(M/A)$ is the same as finitely satisfiable in $M$.
%\end{proof}
%\end{lem}

It should be noted that, contrary to what happens in usual positive logic, \glspl{lp type} might not be preserved under adding parameters. Explicitly, a partial \gls{lp type} of $M$ over $A$ is not necessarily a partial \gls{lp type} of $M$ over $B$ for $A\subseteq B$. The key issue is that a bound over $A$ is not generally a bound over $B$. For instance, consider \cite[Example 3.8(1)]{rodriguez2024completeness}. Recall that $\lang\coloneqq\{P_n,Q_n,\dd_n\}_{n\in\N}$ and $\theory\coloneqq\Th_-(\Z)$ where $\Z\models P_n(x)\Leftrightarrow x>n$, $\Z\models Q_n(x)\Leftrightarrow x<-n$ and $\Z\models \dd_n(x,y)\Leftrightarrow |x-y|\leq n$. Then, $\bigwedge_{n\in \N}P_n(x)$ is a partial \gls{lp type} over $\emptyset$ but it is not a partial \gls{lp type} over $0$. Such a counterexample cannot be given when $A$ is pointed:
\begin{lem} \label{l:types over pointed parameters} Let $M$ be a local $\lang${\hyp}structure and $A$ a pointed subset. Let $A\subseteq B$. Then, every partial \gls{lp type} of $M$ over $A$ is a partial \gls{lp type} of $M$ over $B$.
\end{lem}
\begin{proof} Let $\Gamma(x)$ be a partial \gls{lp type} of $M$ over $A$. By \cref{l:boundedly satisfiable}, $\Gamma(x)$ is boundedly satisfiable over $A$. Pick a feasible bound $\Bound(x)$ of $x$ over $A$ for $\Gamma$. Since $A$ is pointed, $\Bound(x)$ is also a bound over $B$. Hence, $\Gamma(x)$ is boundedly satisfiable over $B$ with feasible bound $\Bound$, so a partial \gls{lp type} over $B$ by \cref{l:boundedly satisfiable}.  
\end{proof}
It is natural to wonder if we can relax the pointedness hypothesis in the previous lemma. The following result shows that, in \gls{lpc} models, the \gls{LJCP} suffices. Furthermore, in some sense, the \gls{LJCP} is equivalent to preservation of types. 

\begin{lem} \label{l:types over local joint continuation property parameters} Let $M$ be a local $\lang${\hyp}structure, $A\subseteq B$ be subsets of $M$ and $M\models^\pc_\lang\Th_-(M/A)$. Suppose that $M$ has the \gls{LJCP} over $A$. Then, every partial \gls{lp type} of $M$ over $A$ is a partial \gls{lp type} of $M$ over $B$. Conversely, if $M$ has the \gls{LJCP} over $B$ and every partial \gls{lp type} on a minimal pointed variable of $M$ over $A$ is a partial \gls{lp type} of $M$ over $B$, then $M$ has the \gls{LJCP} over $A$.
\end{lem}
\begin{proof} By adding parameters to the language, we may assume that $A$ is empty. 

Suppose first that $M$ has the \gls{LJCP} over $\emptyset$. Let $\Gamma(x)$ be a partial \gls{lp type} of $M$ over $\emptyset$. By \cref{l:local positive types over parameters}, there is an $\lang${\hyp}homomorphism $f\map M\to N$ to $N\models^\pc_\lang\Th_-(M)$ realising $\Gamma(x)$. Now, $f$ is a positive embedding as $M$ is a \gls{lpc} structure. On the other hand, consider the $\lang(B)${\hyp}expansion $N_B$ of $N$ given by $b^{N_B}=f(b)$ for $b\in B$. As $f\map M\to N$ is a positive embedding, $N_B\models_\lang\Th_-(M/B)$. Consequently, $\Gamma(x)$ is locally satisfiable in $\Th_-(M/B)$, i.e. it is a partial \gls{lp type} of $M$ over $B$. 

Now, assume that $M$ has the \gls{LJCP} over $B$ and every partial \gls{lp type} on a minimal pointed variable of $M$ over $\emptyset$ is a partial \gls{lp type} of $M$ over $B$. Let $N\models^\pc_\lang\Th_-(M)$ and pick $c$ minimal pointed in $N$. Then, $p(x)\coloneqq\ltp^N_+(c)$ is a \gls{lp type} on a minimal pointed variable of $M$ over $\emptyset$. By hypothesis, $p(x)$ is a partial \gls{lp type} of $M$ over $B$. By \cref{l:local positive types over parameters}, there is an $\lang(B)${\hyp}homomorphism $f\map M_B\to M'_B$ to $M'_B\models^\pc_\lang\Th_-(M/B)$ realising $p(x)$. Pick a realisation $c'$ of $p(x)$. By \cite[Remark 2.4]{rodriguez2024completeness}, $M_B\models^\pc_\lang\Th_-(M/B)$, so $f$ is a positive embedding. By maximality, as $p(x)$ is a \gls{lp type} over $\emptyset$, we conclude $\ltp^{M'}_+(c')=p(x)$. Let $M'_c$ be the $\lang(c)${\hyp}expansion of $M'$ given by $c^{M'_c}=c'$. Then, $M'_c\models_\lang\Th_-(N/c)$. Indeed, $\varphi(c)\in \Th_-(N/c)$ if and only if $\varphi(x)\notin \ltp^N_+(c/A)=p(x)=\ltp^{M'}_+(c'/A)$, if and only if $M'_c\models_\lang\varphi(c)$. Since $c$ is pointed, by \cite[Theorem 3.5]{rodriguez2024completeness}, there are homomorphisms $g\map M'_c\to N'_c$ and $h\map N_c\to N'_c$ to $N'_c\models^\pc_\lang\Th_-(N/c)$. Then, $N'\models_\lang\Th_-(N)$ where $N\models_\lang\Th_-(M)$, so $N'\models_\lang\Th_-(M)$. Finally, $g\circ f\map M\to N'$ and $h\map N\to N'$ are homomorphisms. 

As $N$ is arbitrary and $M$ is \gls{lpc}, we conclude that $\Th_-(M)$ has the \gls{LJCP}. Indeed, given $N_1\models_\lang \Th_-(M)$ and $N_2\models_\lang \Th_-(M)$, there are then homomorphisms $f_1\map N_1\to N'_1$, $f_2\map M\to N'_1$, $g_1\map M\to N'_2$ and $g_2\map N_2\to N'_2$ with $N'_1,N'_2\models_\lang\Th_-(M)$. As $M\models^\pc_\lang\theory$, we have that $f_2$ and $g_1$ are positive embeddings, so $N'_1,N'_2\models_\lang\Th_-(M/M)$. As $M$ is pointed, by \cite[Theorem 3.5]{rodriguez2024completeness}, there are $\lang(M)${\hyp}homomorphisms $h_1\map N'_1\to N'$ and $h_2\map N'_2\to N'$ to a common local model $N'\models_\lang \Th_-(M/M)$. In particular, $h_1\circ f_1\map N_1\to N'$ and $h_2\circ g_2\map N_2\to N'$ are homomorphisms with $N'\models_\lang\Th_-(M)$. 
\end{proof}

The following example shows that the hypothesis $M\models^\pc_\lang\Th_-(M/A)$ is required in the previous lemma. In other words, for general local structures, assuming only the \gls{LJCP} is not enough for preservation of types:
\begin{ex}\label{e:types over LJP are not preserved under adding paramters} Consider the language $\lang$ with two sorts $\Sorts=\{1,2\}$, locality relations $\Dtt^s=\{\dd^s_q\sth q\in \Q_{\geq 0}\}$ for each $s\in \Sorts$ (and the usual ordered monoid structure induced by the positive rationals), and one binary relation $P$ of sort $(1,2)$. Let $R$ be the local $\lang${\hyp}structure whose sorts are two copies of the positive reals, $\R_{\geq 0}^1,\R_{\geq 0}^2$, with interpretations $R\models\dd^s_q(x,y)\Leftrightarrow |x-y|\leq q$ and $R\models P(x,y)\Leftrightarrow x\neq 0\wedge y>\sfrac{1}{x}$. Consider the sets of parameters $A=\{0^1\}$ and $B=\{0^1,0^2\}$ and the partial \gls{lp type} $p(xy)=\bigwedge_{q>0} \dd^1_q(x,0^1)\wedge P(x,y)$, where $0^s$ is the zero of sort $s$. We claim that $R$ has the \gls{LJCP} over $A$ and $p$ is a partial \gls{lp type} over $A$, but it is not a partial \gls{lp type} over $B$. 

Now, let $\widetilde{R}$ be an $\aleph_1${\hyp}saturated elementary extension of $R$. Since $p(xy)$ is finitely satisfiable, we find $\varepsilon^1$ of sort $1$ and $\infty^2$ of sort $2$ in $\widetilde{R}$ realising $p$. Consider the substructure $I$ of $\widetilde{R}$ with universe $I^1=\{\varepsilon^1,0^1\}$ and $I^2=\{\infty^2\}$. We get that $I\models_\lang \Th_-(R/A)$. Now, for any $M\models_\lang \Th_-(R/A)$, consider $f\map M\to I$ given by $f\map 0^1\mapsto 0^1$, $f\map a\mapsto \varepsilon^1$ for all $a\in M^1$ with $a\neq 0^1$, and $f\map a\mapsto \infty^2$ for all $a\in M^2$. Trivially, $f$ is a homomorphism. Therefore, $\Th_-(R/A)$ has the \gls{LJCP}. This also shows that $p(xy)$ is locally satisfiable over $A$. 

Now, $R\models_\lang \neg \exists x\mathrel{} (\dd_{\sfrac{1}{n}}(x,0^1)\wedge P(x,y)\wedge \dd_n(y,0^2))$ for all $n\in\N$. Therefore, $\Th_-(R/B)\models_\lang \left(\dd_{\sfrac{1}{n}}(x,0^1)\wedge P(x,y)\right)\perp \dd_n(y,0^2)$. Hence, $p(xy)$ is not boundedly satisfiable over $B$, so it is not a partial \gls{lp type} over $B$.    
\end{ex}

If $A$ is pointed and $A\subseteq B$, then $B$ is pointed too, i.e. pointedness over parameters is upwards preserved. It is natural to wonder whether the \gls{LJCP} over parameters is upwards preserved too. The following example shows that this is not the case.
\begin{ex} \label{e:ljp no upward preserved} Consider the language $\lang$ with two sorts $\Sorts=\{1,2\}$, locality relations $\Dtt^1=\{\dd^1_n\sth n\in \N\}$ (with the usual ordered monoid structure induced by the natural numbers) in sort $1$ and $\Dtt^2=\{=,\toprel\}$ (with the trivial ordered monoid structure) in sort $2$, one binary relation $P$ of sort $(1,2)$ and binary relations $\{Q_n\}_{n\in\N}$ of sort $(2,1)$. Let $N$ be the local $\lang${\hyp}structure, whose universe is $\N$ in sort $1$ and $\{-\infty,\infty\}$ in sort $2$, with interpretations $N\models\dd^1_n(x,y)\Leftrightarrow |x-y|\leq n$, $N\models P(x,y)\Leftrightarrow x=0\text{ and }y=-\infty$, and $N\models Q_n(y,x)\Leftrightarrow y=\infty\text{ or } x\geq n$. Let us show that $N\models^\pc_\lang\Th_-(N)$ and that $N$ has the \gls{LJCP} over $\emptyset$ but not over $\{\infty\}$. 

Take $M\models_\lang\Th_-(N)$ arbitrary. Consider the map $f\map M\to N$ given by 
\[\begin{array}{ll} 
f(a)=\min(\{k\sth M\models_\lang \exists xy\mathrel{}(\dd_k(a,x)\wedge P(x,y))\}\cup\{0\})&\text{ in sort }M_1, \vspace{5pt}\\ 
f(b)=\begin{cases} -\infty & \text{if }M\models_\lang\exists x\mathrel{}P(x,b),\\ \phantom{-}\infty & \text{otherwise},\end{cases}&\text{ in sort }M_2.\end{array}\]
We check that $f$ is a homomorphism. Obviously, $\toprel$ and $=$ are preserved by $f$. Take $a,a'\in M_1$ and $b\in M_2$ arbitrary. Suppose $M\models_\lang P(a,b)$. Then, by definition, $f(b)=-\infty$ and $f(a)=0$, so $N\models_\lang P(f(a),f(b))$. Suppose $M\models_\lang\dd_k(a,a')$ and $f(a)=n$ and $f(a')=m$. If there is no $a''\in M_1$ with $f(a'')=0$, then $f(a)=f(a')=0$ and $N\models_\lang\dd_k(f(a),f(a'))$ trivially. Otherwise, by definition of $f$, there is $a''\in M_1$ such that $M\models_\lang\exists x\mathrel{} P(a'',x)$ and $M\models_\lang\dd_n(a,a'')$. Then, by \cref{itm:axiom 3}, $M\models_\lang\dd_{n+k}(a',a'')$, so $m\leq n+k$ by definition of $f$. Symmetrically, $n\leq m+k$, so $|n-m|\leq k$, concluding $N\models_\lang \dd_k(f(a),f(a'))$. Finally, suppose $M\models_\lang Q_n(b,a)$. If $M\not\models_\lang \exists x\mathrel{}P(x,b)$, then $f(b)=\infty$ and $N\models_\lang Q_n(f(b),f(a))$. Otherwise, $f(b)=-\infty$ and $M\models_\lang \exists x\mathrel{}(Q_n(b,a)\wedge P(x,b))$. Say $f(a)=m$, so $M\models_\lang \exists xx'y \mathrel{}(P(x,b)\wedge Q_n(b,a)\wedge \dd_m(x',a)\wedge P(x',y))$. However, $N\models_\lang\neg\exists uvxx'y\mathrel{} (P(x,v)\wedge Q_n(v,u)\wedge\dd_m(x',u)\wedge P(x',y))$ for $n>m$. As $M\models_\lang\Th_-(N)$, we conclude $m\geq n$. In other words, $f(a)\geq n$, so $N\models_\lang Q_n(f(b),f(a))$.

In sum, $f$ is a homomorphism. As $M$ is arbitrary, we conclude that $\Th_-(N)$ has the \gls{LJCP}. Furthermore, this shows that $N$ is \gls{lpc} --- alternatively we can show that $N$ is \gls{lpc} using \cite[Lemma 2.6]{rodriguez2024completeness}.

Now, we show that $N$ does not have the \gls{LJCP} over $-\infty$. For the sake of contradiction, suppose $N$ has the \gls{LJCP} over $-\infty$. Consider $p(x)\coloneqq\bigwedge_{n\in\N} Q_n(-\infty,x)$. Obviously, $p(x)$ is finitely satisfiable in $N$ and $x$ is a minimal pointed variable. Therefore, $p(x)$ is a partial \gls{lp type} of $N$ over $-\infty$ by \cref{l:boundedly satisfiable}. Hence, by \cref{l:local positive types over parameters}, assuming that $N$ has the \gls{LJCP} over $-\infty$, there is $f\map N\to M$ to $M\models_\lang\Th_-(N/-\infty)$ such that $M$ realises $p(x)$; take $a\in M$ realising $p(x)$. Since $M$ is a local structure, there is $k$ such that $M\models_\lang \dd_k(f(0),a)$, concluding that $M\models_\lang \dd_k(f(0),a)\wedge P(f(0),f(-\infty))\wedge Q_{k+1}(f(-\infty),a)$. However, 
\[
N\models_\lang \neg\exists xx'y \mathrel{} (\dd_k(x,x')\wedge P(x,y)\wedge Q_{k+1}(y,x')),
\] contradicting that $M\models_\lang\Th_-(N/-\infty)$.        
\end{ex}
\begin{rmk} \cref{e:ljp no upward preserved} also shows that the \gls{lp type} of the parameters $A$ is not sufficient to determine whether the structure has the \gls{LJCP} over them. Indeed, in this example, $\infty$ and $-\infty$ have the same \gls{lp type} and $N$ has the \gls{LJCP} over $\infty$ but not over $-\infty$.   
\end{rmk}

\subsection{Spaces of types}
The \emph{space of \glspl{lp type}} of $\theory$ on $x$ over $\lang$, which we denote by $\Stone^x(\theory/\lang)$, is the set of \glspl{lp type} of $\theory$ on $x$ over $\lang$; the \emph{space of \glspl{lp type}} $\Stone(\theory/\lang)$ is the respective sorted set. We omit $\lang$ if it is clear from the context.

Let $\Gamma\subseteq \LFor^x_+(\lang)$ be a subset of \glspl{lp formula}. The \emph{interpretation} of $\Gamma$ in $\Stone^{x}(\theory/\lang)$ is the set
\[\langle \Gamma(x)\rangle_{\Stone(\theory/\lang)}\coloneqq \{p\in \Stone^x(\theory/\lang)\sth \Gamma\subseteq p\}\]
As usual, we omit $\lang$ and $x$ if they are clear from the context. 

The \emph{\gls{lp logic topology}} of $\Stone^x(\theory)$ is the topology defined by taking $\langle\Gamma\rangle$ as closed for every $\Gamma\subseteq \LFor^x_+(\lang)$. 

\begin{lem} The \gls{lp logic topology} of $\Stone^x(\theory)$ is a well{\hyp}defined topology.
\end{lem}
\begin{proof} By \cref{c:there are types}, there is a \gls{lp type} extending $x=x$, so $\Stone^x(\theory)\neq\emptyset$. Obviously, $\langle x=x\rangle=\Stone^{x}(\theory)$ and $\langle \bot\rangle=\emptyset$, so they are closed.  

Let $\{\Gamma_i\}_{i\in I}$ be a family of \glspl{lp formula}. Then, $\bigcap \langle \Gamma_i\rangle=\langle \bigwedge\Gamma_i\rangle$, where $\bigwedge \Gamma_i\coloneqq \bigcup \Gamma_i=\{\varphi\sth \varphi\in \Gamma_i\text{ for some }i\}$, so arbitrary intersection of closed subsets are closed.

Let $\Gamma_1,\Gamma_2\subseteq \LFor^x_+(\lang)$ and write $\Gamma_1\vee\Gamma_2\coloneqq \{\varphi_1\vee\varphi_2\sth \varphi_1\in\Gamma_1,\ \varphi_2\in\Gamma_2\}$. Then, $\langle \Gamma_1\vee\Gamma_2\rangle=\langle \Gamma_1\rangle\cup\langle\Gamma_2\rangle$. Indeed, trivially $\langle\Gamma_1\rangle \cup\langle\Gamma_2\rangle\subseteq \langle\Gamma_1\vee\Gamma_2\rangle$. On the other hand, suppose $p\notin \langle\Gamma_1\rangle\cup \langle\Gamma_2\rangle$. Then, there are $\varphi_1\in \Gamma_1$ and $\varphi_2\in\Gamma_2$ such that $\varphi_1\notin p$ and $\varphi_2\notin p$. By \cref{l:partial local positive types}, take $M\models^\pc_\lang\theory$ an $a\in M^x$ such that $\ltp^M_+(a)=p$. As $M\not\models_\lang\varphi_1(a)$ and $M\not\models_\lang\varphi_2(a)$, we conclude that $\varphi_1\vee\varphi_2\notin p$, so $p\notin \langle \Gamma_1\vee\Gamma_2\rangle$.
\end{proof}

\begin{rmk} The \gls{lp logic topology} is a $\mathrm{T}_1$ topology (i.e. Fr\'{e}chet). Indeed, for every \gls{lp type} $p$, we have that $\langle p\rangle$ is closed and, by maximality, only contains $p$.
\end{rmk}
As we already noted in \cite{rodriguez2024completeness}, compactness is not true in general for local logic. In particular, $\Stone(\theory)$ is not necessarily compact. However, compactness is true under boundedness.
\begin{lem} \label{l:bounds are compact} Every bound of $x$ is a compact subset of $\Stone^x(\theory)$.
\end{lem}
\begin{proof} Let $\Bound(x)$ be a bound of $x$ in $\theory$. We claim that $\langle \Bound(x)\rangle$ is a compact subset of $\Stone^x(\theory)$. Let $\{\Gamma_i(x)\}_{i\in I}$ be any family of partial \glspl{lp type} of $\theory$ and suppose that $\bigcap _{i\in I_0}\langle \Gamma_i(x)\rangle\cap \langle \Bound(x)\rangle$ is not empty for every $I_0\subseteq I$ finite. Consider $\Gamma(x)=\bigwedge \Gamma_i(x)\wedge \Bound(x)$. Then, $\Gamma$ is finitely locally satisfiable in $\theory$ and $\Gamma\models_\lang \Bound$. Therefore, by \cref{l:boundedly satisfiable}, $\Gamma$ is locally satisfiable in $\theory$. By \cref{c:there are types}, we conclude that there is $p\in\langle \Gamma\rangle=\bigcap \langle \Gamma_i\rangle\cap \langle \Bound\rangle$. As $\{\Gamma_i\}_{i\in I}$ is arbitrary, we conclude that $\langle \Bound\rangle$ is compact. 
\end{proof}
\begin{defi} We say that $\theory$ is \emph{Hausdorff} if $\Stone^x(\theory)$ is Hausdorff with the \gls{lp logic topology} for any variable $x$. 
\end{defi}
\begin{lem} \label{l:Hausdorff} $\Stone^x(\theory)$ is Hausdorff if and only if, for any two different \glspl{lp type} $p,q\in\Stone^x(\theory)$, there are \glspl{lp formula} $\varphi(x)$ and $\psi(x)$ with $\varphi\notin p$, $\psi\notin q$ and $\varphi\toprel\psi$.  
\end{lem}
\begin{proof} Suppose $\Stone^x(\theory)$ is Hausdorff and let $p,q\in\Stone^x(\theory)$ be two different \glspl{lp type}. Thus, there are $\Gamma_1,\Gamma_2\subseteq \LFor^x_+(\lang)$ such that $\langle \Gamma_1\rangle\cup\langle\Gamma_2\rangle=\Stone^x(\theory)$ and $p\notin \langle\Gamma_1\rangle$ and $q\notin\langle\Gamma_2\rangle$. There are $\varphi\in \Gamma_1$ and $\psi\in \Gamma_2$ such that $\varphi\notin p$ and $\psi\notin q$. As $\langle \Gamma_1\rangle\cup\langle\Gamma_2\rangle =\Stone^x(\theory)$, we have that $\langle \varphi\vee\psi\rangle=\Stone^x(\theory)$. For any $M\models^\pc_\lang\theory$ and any $a\in M^x$, we have that $\varphi\vee\psi\in \ltp_+(a)$. In other words, $\theory\models^\pc_\lang\forall x\mathrel{} (\varphi\vee\psi)$, i.e. $\varphi\toprel\psi$ by \cite[Lemma 2.15(2)]{rodriguez2024completeness}.

On the other hand, suppose that for any two different \glspl{lp type} $p(x)$ and $q(x)$ there are $\varphi,\psi\in\LFor^x_+(\lang)$ such that $\varphi\notin p$, $\psi\notin q$ and $\varphi\toprel\psi$. Then, by \cite[Lemma 2.15(2)]{rodriguez2024completeness} and \cref{l:partial local positive types}, $\langle\varphi\rangle\cup\langle\psi\rangle=\langle\varphi\vee\psi\rangle=\Stone^x(\theory)$, so $\langle\varphi\rangle^c$ and $\langle\psi\rangle^c$ are open disjoint subsets of $\Stone^x(\theory)$ separating $p$ and $q$.
\end{proof}
\begin{coro} \label{c:hausdorff finite variables} $\theory$ is Hausdorff if and only if $\Stone^x(\theory)$ is Hausdorff for every finite variable $x$.
\end{coro}
\begin{proof} Let $x$ be an arbitrary variable. Suppose $p,q\in\Stone^x(\theory)$ are different. Then, $p(x)\wedge q(x)$ is not locally satisfiable. As $p$ and $q$ are maximal, they contain a bound on $x$, so $p(x)\wedge q(x)$ contains a bound on $x$. Consequently, by \cref{l:boundedly satisfiable}, $p(x)\wedge q(x)$ is not locally finitely satisfiable. Hence, there is a finite subvariable $y\subseteq x$ such that $p_{\mid y}\wedge q_{\mid y}$ is not locally finitely satisfiable, where $p_{\mid y}$ and $q_{\mid y}$ are corresponding restrictions. Hence, $\langle p_{\mid y}\rangle$ and $\langle q_{\mid y}\rangle$ are disjoint. Since $p$ and $q$ are maximal, they contain bounds on $x$, so $p_{\mid y}$ and $q_{\mid y}$ contain bounds on $y$. By \cref{l:bounds are compact}, we conclude that $\langle p_{\mid y}\rangle$ and $\langle q_{\mid y}\rangle$ are compact disjoint. Say $\langle p_{\mid y}\rangle=\{p_i(y)\}_{i\in I}$ and $\langle q_{\mid y}\rangle=\{q_j(y)\}_{j\in J}$. By \cref{l:Hausdorff}, for each $i\in I$ and $j\in J$, there are $\varphi_{ij},\psi_{ij}\in \LFor^y_+(\lang)$ with $\varphi_{ij}\toprel\psi_{ij}$ such that $\varphi_{ij}\notin p_i$ and $\psi_{ij}\notin q_j$. Thus, $\{\langle \psi_{ij}\rangle^c\sth j\in J\}$ is an open covering of $\langle q_{\mid y}\rangle$. Hence, by compactness, there is a finite $J_i\subseteq J$ such that $\langle q_{\mid y}\rangle\subseteq \langle \bigwedge_{j\in J_i} \psi_{ij}\rangle^c$. Then, $p_i\in \langle \bigvee_{j\in J_i} \varphi_{ij}\rangle^c$ where $\langle \bigvee_{j\in J_i} \varphi_{ij}\rangle^c$ and $\langle \bigwedge_{j\in J_i} \psi_{ij}\rangle^c$ are disjoint. Hence, $\{\langle\bigvee_{j\in J_i} \varphi_{ij}\rangle^c\sth i\in I\}$ is an open covering of $\langle p_{\mid y}\rangle$. By compactness, we get that there is a finite $I'\subseteq I$ such that $\langle p_{\mid y}\rangle\subseteq \langle \bigwedge_{i\in I'}\bigvee_{j\in J_i}\varphi_{ij}\rangle^c$ and $\langle q_{\mid y}\rangle \subseteq \langle \bigvee_{i\in I'}\bigwedge_{j\in J_i}\psi_{ij}\rangle^c$ with $\langle \bigwedge_{i\in I'}\bigvee_{j\in J_i}\varphi_{ij}\rangle^c$ and $\langle \bigvee_{i\in I'}\bigwedge_{j\in J_i}\psi_{ij}\rangle^c$ disjoint. Then, $\varphi(y)\coloneqq \bigwedge_{i\in I'}\bigvee_{j\in J_i}\psi_{ij}$ and $\psi(y)\coloneqq \bigvee_{i\in I'}\bigwedge_{j\in J_i}\psi_{ij}$ satisfy that $\varphi\toprel\psi$ and $\varphi\notin p$ and $\psi\notin q$. As $p$ and $q$ are arbitrary, by \cref{l:Hausdorff}, we conclude that $\Stone^x(\theory)$ is Hausdorff. As $x$ is arbitrary, we conclude that $\theory$ is Hausdorff.
\end{proof}

Let $x$ and $y$ be two variables. Recall that a \emph{substitution} of $x$ by $y$ is a function $\zeta\map x\to y$ such that $x_i$ and $\zeta(x_i)$ have the same sort for any $x_i\in x$. Given a formula $\varphi$, we write $\zeta_*(\varphi)$ for the formula recursively defined by replacing at the same time every free instance of $x_i$ in $\varphi$ by a free instance of $\zeta(x_i)$. 

In positive logic, a key notion generalizing Hausdorffness is semi{\hyp}Hausdorffness. The definition of semi{\hyp}Hausdorffness depends on the fact that pull{\hyp}backs by substitutions of variables are well{\hyp}defined continuous maps between the spaces of types. Unfortunately, in local positive logic, for non pointed variables, the pull{\hyp}back of a \gls{lp type} by a substitution of variables is not necessarily maximal, i.e. it may be only a partial \gls{lp type}. Therefore, for non pointed variables, pull{\hyp}backs by substitutions of variables are not necessarily well{\hyp}defined maps between spaces of types. In particular, the definition of semi{\hyp}Hausdorffness only makes sense for pointed variables.    
\begin{lem} Let $x$ and $y$ be pointed variables and $\zeta$ a substitution of $x$ by $y$. Then:
\begin{enumerate}[label={\rm{(\arabic*)}}, ref={\rm{\arabic*}}, wide]
\item $\zeta^*\coloneqq \Ima^{-1}\zeta_*\map \Stone^y(\theory)\to\Stone^x(\theory)$ is a well{\hyp}defined function. Furthermore, for any $M\models_\lang \theory$ and $a\in M^y$, we have $\zeta^*\ltp_+(a)=\ltp_+(\zeta^*a)$.
\item $\Ima^{-1}\zeta^*(\langle \Gamma\rangle)=\langle \Ima\, \zeta_*(\Gamma)\rangle$ for any $\Gamma\subseteq\LFor^x_+(\lang)$.
\end{enumerate}
\end{lem}
\begin{proof} 
\begin{enumerate}[label={\rm{(\arabic*)}}, wide]
\item[\hspace{-1.2em}\setcounter{enumi}{1}\theenumi] First, recall that $\premodels \zeta_*\varphi(a)$ if and only if $\,\premodels \varphi(\zeta^*a)$. Thus, $\zeta_*\varphi\in \ltp_+(a)$ if and only if $\varphi\in \ltp_+(\zeta^*a)$ for any $\varphi\in\LFor^x_+(\lang)$, $M\models_\lang\theory$ and $a\in M^y$. In other words, $\zeta^*\ltp_+(a)=\ltp_+(\zeta^*a)$. In particular, by \cref{c:local positive types}, $\zeta^*\map \Stone^y(\theory)\to \Stone^x(\theory)$ is a well{\hyp}defined function.
\item  $\Ima^{-1}\zeta^*(\langle\Gamma\rangle)=\{q\in\Stone^y(\theory)\sth \zeta^*(q)\in \langle\Gamma\rangle\}=\{q\in\Stone^y(\theory)\sth \Gamma \subseteq \Ima^{-1}\zeta_*(q)\}=\{q\in\Stone^y(\theory)\sth \Ima\,\zeta_*(\Gamma)\subseteq q\}=\langle\Ima\, \zeta_*(\Gamma)\rangle.$
\qedhere
\end{enumerate}
\end{proof}

\begin{coro} Let $x$ and $y$ be pointed variables and $\zeta$ a substitution of $x$ by $y$. Then, $\zeta^*\map \Stone^y(\theory)\to\Stone^x(\theory)$ is a continuous map for the logic topologies.  
\end{coro}

%\begin{rmk} In particular, bijective substitutions of pointed variables define homeomorphisms between the spaces of types.\end{rmk}

Let $x_1,\ldots,x_n$ be pointed variables and write $\overline{x}=(x_1,\ldots,x_n)$. Let $y$ be a \emph{disjoint copy} of $\overline{x}$, i.e. $y$ is a variable disjoint to $x_1,\ldots,x_n$ together with a partition $y=y_1\sqcup\cdots\sqcup y_n$ and bijective substitutions $\iota_i\map x_i\to y_i$. Then, 
\[\begin{array}{lccc} \iota^*:& \Stone^y(\theory)&\to &\Stone^{\bar{x}}(\theory)\coloneqq \prod_i\Stone^{x_i}(\theory)\\ & p&\mapsto & (\iota^*_1(p),\ldots,\iota^*_n(p))\end{array}\] 
is a continuous map. 

Let $\Gamma\subseteq \LFor^y_+(\lang)$ be a subset of \glspl{lp formula}. The \emph{interpretation} of $\Gamma$ in $\Stone^{\bar{x}}(\theory)=\prod\Stone^{x_i}(\theory)$ is the subset $\Gamma(\Stone^{\overline{x}}(\theory))\coloneqq\Ima\, \iota^*(\langle\Gamma\rangle)$; we write $\theory\models_\lang\Gamma(p_1,\ldots,p_n)$ when $(p_1,\ldots,p_n)\in \Gamma(\Stone^{\bar{x}}(\theory))$. In other words, $\theory\models_\lang \Gamma(p_1,\ldots,p_n)$ if and only if there are $a_1,\ldots,a_n$ in $M\models_\lang\theory$ with $\ltp^M_+(a_i)=p_i$ for $1\leq i\leq n$ such that $M\models_\lang \Gamma(a_1,\ldots,a_n)$. 

We say that a subset $V$ of $\Stone^{\bar{x}}(\theory)$ is \emph{locally positively $\bigwedge${\hyp}definable} if there is $\Gamma\subseteq \LFor^y_+(\lang)$ such that $\Ima^{-1}\iota^*(V)=\langle\Gamma\rangle$; we say then that $\Gamma$ \emph{defines} $V$. In other words, for any $M\models^\pc_\lang\theory$ and $a_1,\ldots,a_n$ in $M$, we have $(\ltp^M_+(a_i))_{1\leq i\leq n}\in V$ if and only if $M\models \Gamma(a_1,\ldots,a_n)$. 
\begin{rmk} Thus, $\Gamma$ defines $V$ if and only if $\Gamma(\Stone^{\overline{x}}(\theory))=V$ and, for any $a_1,\ldots,a_n$ and $a'_1,\ldots,a'_n$ in \gls{lpc} models with $\ltp_+(a_i)=\ltp_+(a'_i)$, we have $\Gamma(a_1,\ldots,a_n)$ if and only if $\Gamma(a'_1,\ldots,a'_n)$.
\end{rmk}
\begin{defi} We say that $\theory$ is \emph{semi{\hyp}Hausdorff} if, for any pointed variable $x$, the diagonal $\Delta^x(\theory)\coloneqq \{(p,q)\in \Stone^x(\theory)\sth p=q\}$ is locally positively $\bigwedge${\hyp}definable. In other words, $\theory$ is semi{\hyp}Hausdorff if and only if, for any pointed variable $x$, there is a partial \gls{lp type} $\Delta(x,y)\subseteq \LFor^{xy}_+(\lang)$, where $y$ is disjoint copy of $x$, such that, for any $M\models^\pc_\lang\theory$ and any $a,b\in M^x$, $M\models_\lang\Delta(a,b)$ if and only if $\ltp^M_+(a)=\ltp^M_+(b)$. 
\end{defi}
 
\begin{lem} If $\theory$ is Hausdorff, then $\theory$ is semi{\hyp}Hausdorff.
\end{lem}
\begin{proof} Pick $x,y$ disjoint pointed variables of the same sort. Take 
\[\Delta(x,y)=\{\varphi(x)\vee\psi(y)\sth \varphi,\psi\in\LFor^x_+(\lang),\ \varphi\toprel\psi\}.\] 
We claim that $\Delta(x,y)$ defines the diagonal of $\Stone^x(\theory)$. In other words, we claim that for any $M\models^\pc_\lang\theory$ and $a,b\in M^s$, $M\models_\lang\Delta(a,b)$ if and only if $\ltp_+(a)=\ltp_+(b)$. 

Say $\ltp_+(a)=\ltp_+(b)$. Take $\varphi,\psi\in\LFor^x_+(\lang)$ arbitrary with $\varphi\toprel\psi$. In particular, we have that $M\models_\lang \varphi(a)\vee\psi(a)$ and $M\models_\lang \varphi(b)\vee\psi(b)$. Without loss of generality, say $M\models_\lang\psi(a)$. As $\ltp_+(a)=\ltp_+(b)$, we conclude that $M\models_\lang\psi(b)$, so $M\models_\lang \varphi(a)\vee\psi(b)$. As $\varphi,\psi$ are arbitrary, we conclude that $M\models_\lang\Delta(a,b)$.

On the other hand, suppose $\ltp_+(a)\neq\ltp_+(b)$. Then, by Hausdorffness, there are $\varphi,\psi\in\LFor^x_+(\lang)$ such that $\varphi\toprel\psi$ and $\varphi\notin\ltp_+(a)$ and $\psi\notin\ltp_+(b)$. Hence, $M\not\models_\lang \varphi(a)\vee\psi(b)$ with $\varphi(x)\vee\psi(y)\in \Delta(x,y)$, concluding $M\not\models_\lang\Delta(a,b)$.
\end{proof}

\begin{rmk} In classic first{\hyp}order logic, reducts are also continuous functions between the spaces of types. It is a well known fact that, in positive logic, that is no longer the case. The issue is that the reduct of a positive type could be only a partial positive type. Of course, the same still applies in local positive logic. \end{rmk}

Let $M$ be a local $\lang${\hyp}structure, $x$ a variable and $A$ a set. The \emph{space of \glspl{lp type}} $\Stone^x(M/A)$ of $M$ on $x$ over $A$ is the set of \glspl{lp type} of $M$ on $x$ over $A$; the \emph{space of \glspl{lp type}} $\Stone(M/A)$ is the respective sorted set. We omit $M$ if it is clear from the context, but we never omit $A$.

Let $M$ be a local $\lang${\hyp}structure and $A$ a subset. Let $\Gamma\subseteq\LFor^x_+(\lang(A))$ be a subset of \glspl{lp formula}. We define
\[\langle \Gamma(x)\rangle_{\Stone(A)}\coloneqq\{p\in \Stone^x(A)\sth \Gamma\subseteq p\}.\]

The \emph{\gls{lp logic topology}} of $\Stone^x(A)$ is the topology defined by taking $\langle \Gamma\rangle$ as closed for every $\Gamma\subseteq\LFor^x_+(\lang(A))$.

\begin{rmk} Tautologically, $\Stone(A)=\Stone(\Th_-(M/A))$ as topological spaces.
\end{rmk}

Let $M$ be a local $\lang${\hyp}structure and $A$ a subset. We say that $M$ is \emph{Hausdorff over $A$} if $\Stone(A)$ is Hausdorff. In other words, $M$ is Hausdorff over $A$ if $\Th_-(M/A)$ is Hausdorff. We say that $M$ is \emph{semi{\hyp}Hausdorff over $A$} if $\Th_-(M/A)$ is semi{\hyp}Hausdorff. In other words, $M$ is semi{\hyp}Hausdorff over $A$ if the diagonal $\Delta^x_+(A)\coloneqq \{(p,q)\in\Stone^x(A)\sth p=q\}$ is locally positively $\bigwedge_A${\hyp}definable for any variable $x$ pointed over $A$. 

\begin{lem} \label{l:Hausdorffness and parameters} Let $M$ be \gls{lpc} and $A\subseteq B$. Assume $M$ has the \gls{LJCP} over $B$ and $x$ is a pointed variable over $B$. If $\Stone^x(M/A)$ is Hausdorff, then $\Stone^x(M/B)$ is Hausdorff. In particular, if $B$ is pointed with $A\subseteq B$ and $M$ Hausdorff over $A$, then $M$ is Hausdorff over $B$. 
\end{lem}
\begin{proof} Without loss of generality, assume $A=\emptyset$ --- note that $M_A\models^\pc_\lang\Th_-(M/A)$ by \cite[Remark 2.4]{rodriguez2024completeness}. Let $p(x)$ and $q(x)$ be two different \glspl{lp type} of $M$ over $B$. Let $z$ be a variable disjoint to $x$ with an enumeration $b\in M^z$ of $B$. Consider $\widehat{p}(xz)\coloneqq\{\phi(x,z)\in\LFor^{xz}_+(\lang)\sth \phi(x,b)\in p(x)\}$ and $\widehat{q}(xz)\coloneqq\{\phi(x,z)\in\LFor^{xz}_+(\lang)\sth\phi(x,b)\in q(x)\}$. We note that $\widehat{p}(xz)$ and $\widehat{q}(xz)$ are \glspl{lp type} in $M$. Indeed, by \cref{l:local positive types over parameters}, find a homomorphism $f\map M\to N$ to $N_B\models^\pc_\lang \Th_-(M/B)$ such that $\ltp_+(a/B)=p(x)$ for some $a\in N^x$. Then, $\ltp_+(a\, b^{N_B})=\widehat{p}(xz)$. By \cite[Lemma 2.5]{rodriguez2024completeness}, $N\models^\pc_\lang\Th_-(M)$. By \cref{l:positively closed and local positive types}, we conclude that $\widehat{p}$ is a \gls{lp type}. Now, $\widehat{p}(xz)\neq \widehat{q}(xz)$, so there are $\psi(x,z)$ and $\varphi(x,z)$ with $M\models_\lang\psi\toprel\varphi$ such that $\psi\notin \widehat{p}$ and $\varphi\notin \widehat{q}$. Consequently, $M\models_\lang\psi(x,b)\toprel\varphi(x,b)$ and $\psi(x,b)\notin p$ and $\varphi(x,b)\notin q$. As $p$ and $q$ are arbitrary, we conclude that $\Stone^x(M/B)$ is Hausdorff by \cref{l:Hausdorff}. 
\end{proof}

\begin{lem} \label{l:semi-Hausdorffness and parameters} Let $M$ be a \gls{lpc} and $A\subseteq B$ subsets. Suppose $M$ has the \gls{LJCP} over $B$. If $M$ is semi{\hyp}Hausdorff over $A$, then $M$ is semi{\hyp}Hausdorff over $B$. 
\end{lem}
\begin{proof} Without loss of generality, by adding parameters, assume $A=\emptyset$ --- note that $M_A\models^\pc_\lang\Th_-(M/A)$ by \cite[Remark 2.4]{rodriguez2024completeness}. 

Take a pointed variable $x$ over $B$. Let $z$ be a variable disjoint to $x$ with an enumeration $b\in M^z$ of $B$. Since $M$ is semi{\hyp}Hausdorff over $\emptyset$, there is a partial \gls{lp type} $\Delta(xz,x'z')$ (with $x'z'$ disjoint copy of $xz$) defining equality of \glspl{lp type} in $\Stone^{xz}(M/\emptyset)$. In other words, for any $N\models^\pc_\lang\Th_-(M)$, $c\in N^{xz}$ and $c'\in N^{x'z'}$, we get that $N\models_\lang\Delta(c,c')$ if and only if $\ltp_+(c)=\ltp_+(c')$.  Consider the partial \gls{lp type} over $B$ given by $\Delta'(x,x')\coloneqq \Delta(xb,x'b)$. Take $N_B\models^\pc_\lang\Th_-(M/B)$ and $a\in N^x$ and $a'\in N^{x'}$. If $\ltp_+(a/b)=\ltp_+(a'/b)$, then $\ltp_+(ab)=\ltp_+(a'b)$, so $N\models_\lang\Delta(ab,a'b)$, i.e. $N_B\models_\lang \Delta'(a,a')$. Conversely, if $N_B\models_\lang \Delta'(a,a')$, then $N\models_\lang\Delta(ab,a'b)$. By the \gls{LJCP} over $B$ (and \cite[Theorem 2.9]{rodriguez2024completeness}), there are homomorphisms $g\map N_B \to N'_B$ and $f\map M_B\to N'_B$ to $N'_B\models^\pc_\lang\Th_-(M/B)$. As $N_B\models^\pc_\lang\Th_-(M/B)$, we get that $g$ is a positive embedding. By \cite[Lemma 2.5]{rodriguez2024completeness}, $N'\models^\pc_\lang\Th_-(M)$. By \cite[Corollary 2.13]{rodriguez2024completeness}, we conclude $N\models^\pc_\lang\Th_-(M)$. 
Thus, $\ltp_+(ab)=\ltp_+(a'b)$, concluding that $\ltp_+(a/B)=\ltp_+(a'/B)$.
\end{proof}

\begin{defi} Let $M\models_\lang\theory$ and $I$ be a first{\hyp}order structure. Let $a=(a_i)_{i\in I}$ be a sequence in $M^x$ for some variable $x$. We say that $(a_i)_{i\in I}$ is \emph{locally positively indiscernible} if $\ltp^M_+(a_{\bar{\imath}})=\ltp^M_+(a_{\bar{\jmath}})$ whenever $\qftp^I(\bar{\imath})=\qftp^I(\bar{\jmath})$.
\end{defi}
\begin{lem}\label{l:standard lemma} Let $M\models_\lang\theory$ and $I$ be a linearly ordered set\footnote{We can consider other Ramsey first{\hyp}order structures.}. Let $a=(a_i)_{i\in I}$ be a sequence in $M$. Assume that there is a locality relation $\dd$ such that $M\models \dd(a_i,a_j)$ for any $i,j\in I$. Then, there is a local model $N\models_\lang\theory$ and a locally positively indiscernible sequence $b=(b_i)_{i\in I}$ of $N$ such that, for any \gls{lp formula} $\varphi\in \LFor_+(\lang)$, we have $N\models_\lang \varphi(b_{\bar{\jmath}})$ whenever $M\models_\lang \varphi(a_{\bar{\imath}})$ for every $\bar{\imath}$ with $\qftp(\bar{\imath})=\qftp(\bar{\jmath})$.
\end{lem}
\begin{proof} Take an elementary extension $M\preceq \widetilde{N}$ into a (non{\hyp}local) $\kappa${\hyp}saturated structure with $\kappa>|\lang|+|I|$ --- note that $\widetilde{N}$ is not longer a local structure in general. By the Ehrenfeucht{\hyp}Mostowski{'s} Standard Lemma \cite[Lemma 5.1.3]{tent2012course}, we can find an indiscernible sequence $b=(\widetilde{b}_i)_{i\in I}$ in $\widetilde{N}$ such that, for any formula $\varphi$, $\widetilde{N}\premodels \varphi(b_{\bar{\jmath}})$ whenever $M\premodels \varphi(a_{\bar{\imath}})$ for every $\bar{\imath}$ with $\qftp(\bar{\imath})=\qftp(\bar{\jmath})$. As $M$ is a local model, $a_0$ must satisfy some bound $\Bound(x)$. Now, there is a bound $\Bound'(x)$ implied by $\Bound(y)\wedge\dd(x,y)$, so that $a_i$ satisfies $\Bound'(x)$ for all $i\in\N$. Hence, $b_i$ satisfies $\Bound'(x)$ for all $i\in\N$. Complete $b_0$ to a pointed tuple $o$ picking one element for each remaining single sort. Let $\Dtt(o)$ be the local component at $o$. By the choice of $o$, and the fact that $b_0$ satisfies a bound, we conclude that $\Dtt(o)\models_\lang\theory$ and satisfies every local formula with parameters in $\Dtt(o)$ satisfied by $\widetilde{N}$ by \cite[Lemma 1.8]{rodriguez2024completeness}. As $\widetilde{N}\premodels \dd(b_0,b_i)$ for each $i\in I$, we know that $b$ lies in $\Dtt(o)$. Therefore, $\Dtt(o)$ contains a locally positively indiscernible sequence $b$ such that, for any \gls{lp formula} $\varphi\in \LFor_+(\lang)$, we have $\Dtt(o)\models_\lang \varphi(b_{\bar{\jmath}})$ whenever $M\models_\lang \varphi(a_{\bar{\imath}})$ for every $\bar{\imath}$ with $\qftp(\bar{\imath})=\qftp(\bar{\jmath})$. 
\end{proof}

\subsection{Completeness via types}
We say that two subsets $\Gamma_1(x)$ and $\Gamma_2(y)$ of formulas in $\lang$ are \emph{(locally) compatible} in $\theory$ if there is a local model $M\models_\lang \theory$ realising both of them, i.e. there are $a\in M^x$ and $b\in M^y$ such that $M\models_\lang \Gamma_1(a)$ and $M\models_\lang\Gamma_2(b)$. 

By definition, $\theory$ is irreducible if and only if any two locally satisfiable quantifier free positive formulas are compatible. By \cite[Theorem 3.12]{rodriguez2024completeness}, weak completeness is equivalent to irreducibility in countable languages, which gives a characterisation of weak completeness in terms of compatibility. The following lemma provides a characterisation of completeness in terms of compatibility.

\begin{lem} \label{l:completeness via compatibility}  $\theory$ is complete if and only if any \gls{lp type} $p(x)$ of $\theory$, with $x$ finite, and any quantifier free positive formula $\varphi(y)$, with $\exists y\mathrel{} \varphi(y)\in\theory_+$, are compatible in $\theory$.
\end{lem}
\begin{proof} One direction is obvious. By \cref{l:partial local positive types}, $p$ is realised in some $M\models^\pc_\lang\theory$. Thus, if $\theory$ is complete, since $\Th_-(M)=\theory$, $M\models_\lang\theory_+$, so $M\models_\lang\exists y\mathrel{} \varphi(y)$, concluding that it realises $\varphi$.

Conversely assume $\theory$ is incomplete. Then, there is some $M\models^\pc_\lang\theory$ such that $M$ is not a model of $\theory_+$. Let $\exists y\mathrel{} \varphi(y)\in\theory_+$, with $\varphi$ quantifier free, such that $M\models_\lang\neg\exists y\mathrel{} \varphi(y)$. Let $s=(s(y_i))_{y_i\in y}$ be the sort of $y$. For each single sort $s(y_i)$ pick a single variable $x_{s(y_i)}$ disjoint to $y$ of sort $s(y_i)$ and let $x=\{x_{s(y_i)}\sth y_i\in y\}$. Let $z$ be a duplicate free variable disjoint to $x$ and $y$ such that $xz$ is pointed. Choose $a\in M^x$ and $b\in M^z$, and set $p\coloneqq\ltp_+^M(a)$ and $q\coloneqq\ltp_+^M(ab)$. 

For each $y_i\in y$, choose a locality relation $\dd_{y_i}$ of sort $s(y_i)$ and set $\dd(x,y)=\bigwedge\dd_{y_i}(x_{s(y_i)},y_i)$. Then, the set of formulas 
\[\Sigma_\dd(x,y,z)\coloneqq q(x,z)\wedge \varphi(y)\wedge \dd(x,y)\] 
is inconsistent with $\theory$.

Indeed, for the sake of contradiction, suppose otherwise and take $N\premodels\theory$ with $c\in N^x$, $d\in N^y$ and $e\in N^z$ such that $N\premodels\Sigma_\dd(c,d,e)$. Consider $N'\coloneqq\Dtt(ce)$ (and note $d\in N'$ as $N\premodels \dd(c,d)$). By \cite[Lemma 1.8]{rodriguez2024completeness}, since $\theory$ is a negative local theory and $\Sigma_\dd$ is a set of \glspl{lp formula}, $N'\models_\lang\Sigma_\dd(c,d,e)$ and $N'\models_\lang\theory$. By \cite[Theorem 2.9]{rodriguez2024completeness}, there is a homomorphism $h\map N'\to P$ to $P\models^\pc_\lang\theory$. Since homomorphisms preserve satisfaction of positive formulas, we get that $P\models_\lang\Sigma_\dd(h(c,d,e))$. In particular, $\ltp^P_+(h(ce))=q$ by maximality of $q$. Hence, as $\ltp^P_+(h(ce))=q$, we get that $P_{ce}\models_\lang \psi$ for any $\psi\in \LFor^0_-(\lang(ce))$ realised by $M_{ce}$, where $P_{ce}$ is the $\lang(ce)${\hyp}expansion of $P$ given by $h$. As $ce$ is pointed, this implies that $P_{ce}\models_\lang \Th_-(M/ce)$. Hence, as $P\models^\pc_\lang\theory$, we get that $P_{ce}\models^\pc_\lang\Th_-(M/ce)$ and, similarly, $M_{ce}\models^\pc_\lang\Th_-(M/ce)$. Therefore, by \cite[Theorem 3.5]{rodriguez2024completeness} applied to $\Th_-(M/ce)$, we get that $\Th_-(P_{ce})=\Th_-(M_{ce})$. Therefore, since $P\models_\lang\exists y\mathrel{} \varphi(y)$, we conclude that $M$ satisfies $\exists y\mathrel{} \varphi(y)$ as well, getting a contradiction. 

Thus, for any choice of locality relations $\dd$, there is some $\psi_\dd(x,z)\in q$ such that $\neg\exists x y z\mathrel{} (\psi_\dd(x,z)\wedge\varphi(y)\wedge\dd(x,y))\in\theory$ --- where the existential quantifier ranges only on the variables that actually appears in the formula. Now, note that $\exists z\mathrel{} \psi_\dd(x,z)\in p$. Thus, we conclude that $p$ contradicts $\exists y\mathrel{} (\dd(x,y)\wedge\varphi(y))$ for any choice of $\dd$. In other words, $p$ and $\exists y\mathrel{} \varphi(y)$ are incompatible.
\end{proof}

On the other hand, we have the following characterisation of the \gls{LJCP} in terms of compatibility.
\begin{lem} \label{l:local joint continuation property via compatibility} A local negative theory $\theory$ has the \gls{LJCP} if and only if any two \glspl{lp type} on minimal pointed variables are compatible.
\end{lem}
\begin{proof} One direction is obvious. Let $p(x)$ and $q(y)$ be \glspl{lp type} on pointed variables. By definition, $p$ and $q$ are realised in some $A,B\models_\lang\theory$. If $\theory$ has the \gls{LJCP}, there are homomorphisms $f\map A\to M$ and $g\map B\to M$ to $M\models_\lang \theory$. Then, as homomorphisms preserve satisfaction of positive formulas, $M$ realises $p$ and $q$.

On the other hand, suppose that any two partial \glspl{lp type} on minimal pointed variables are compatible. Let $A,B\models_\lang\theory$. Pick $a\in A^x$ and $b\in B^y$ minimal pointed tuples (with $x$ and $y$ disjoint). Then, by assumption, $p(x)\coloneqq\ltp(a)$ and $q(y)\coloneqq\ltp_+(b)$ are compatible, so $p(x)\wedge q(y)$ is boundedly satisfiable. Pick a feasible bound $\Bound(x,y)$. Let $M^\ast\premodels\theory$ be (non{\hyp}local) $\kappa${\hyp}saturated. Then, $p(x)\wedge q(y)\wedge \Bound(x,y)$ is realised in $M^\ast$ by some $a'$ and $b'$. By saturation, there are homomorphism $f\map A\to M^\ast$ and $g\map B\to M^\ast$ with $f(a)=a'$ and $g(b)=b'$. Consider the substructure $M\coloneqq\Dtt(a')$. By \cite[Lemma 1.8]{rodriguez2024completeness}, $M\models_\lang \theory$. As $M^\ast\premodels \Bound(a',b')$, we get that $b'$ is in $M$. Since $A$ and $B$ are local, we get that $f\map A\to M$ and $g\map B\to M$, concluding that $\theory$ has the \gls{LJCP}. 
\end{proof}
\begin{que} \label{q:joint property via compatibility} In light of \cref{l:completeness via compatibility}, does compatibility of \glspl{lp type} on finite variables imply the \gls{LJCP}? \end{que}

\section{Retractors} \label{s:retractors} Recall that, unless otherwise stated, we have fixed a local language $\lang$ and a locally satisfiable negative local theory $\theory$.

We start this section by defining saturation in local positive logic. Normally, saturation is defined in terms of realisation of types. In local logic, however, without the \gls{LJCP}, there are incompatible \glspl{lp type}. Consequently, in general, it is not possible to have a structure realising every \gls{lp type}. In order to overcome this issue, we define local positive saturation as a universal property for positive embeddings. 
\begin{defi} Let $M\models_\lang\theory$ and $\kappa$ a cardinal. We say that $M$ is a \emph{locally positively $\kappa${\hyp}saturated} for $\theory$ if, for any positive embeddings $f\map A\to M$ and $g\map A\to B$ to $B\models_\lang\theory$ with $|A|,|B|<\kappa$, there is a homomorphism $h\map B\to M$ such that $f=h\circ g$. %We say that $M$ is \emph{locally positively absolutely saturated} for $\theory$ if it is locally positively $\kappa${\hyp}saturated for $\theory$ for any cardinal $\kappa$. 
\end{defi}
From now on, for the sake of simplicity, we will abbreviate the terminology by writing ``saturated'' in place of ``locally positively saturated''.
\begin{rmk} When $M\models^\pc_\lang\theory$, the condition that $g\map A\to B$ is a positive embedding is superfluous. Indeed, since $f\map A\to M$ is a positive embedding, we automatically have $A\models^\pc_\lang\theory$ by \cite[Corollary 2.13]{rodriguez2024completeness}, so $g\map A\to B$ must be a positive embedding.
\end{rmk}
\begin{lem} \label{l:saturation and parameters} If $M$ is $\kappa${\hyp}saturated for $\theory$, then $M_A$ is $\kappa${\hyp}saturated for $\Th_-(M/A)$ for every subset $A$ with $|A|<\kappa$. Conversely, if $M$ is $\kappa${\hyp}saturated for $\Th_-(M)$, then $M$ is $\kappa${\hyp}saturated for $\theory$. 
\end{lem}
\begin{proof} Suppose $M$ is $\kappa${\hyp}saturated for $\theory$. Let $f\map B_A\to M_A$ and $g\map B_A\to C_A$ be positive embeddings where $C_A\models_\lang\Th_-(M/A)$ and $|B|,|C|<\kappa$. In particular, $C\models_\lang\theory$ so, as $M$ is $\kappa${\hyp}saturated for $\theory$, there is an $\lang${\hyp}homomorphism $h\map C\to M$ such that $f=h\circ g$. Now, $h(a^{C_A})=h(g(a^{B_A}))=f(a^{B_A})=a$ for any $a\in A$, so $h$ is in fact an $\lang(A)${\hyp}homomorphism. Hence, we conclude that $M_A$ is $\kappa${\hyp}saturated for $\Th_-(M/A)$.

Conversely, suppose $M$ is $\kappa${\hyp}saturated for $\Th_-(M)$. Let $f\map A\to M$ and $g\map A\to B$ be positive embeddings where $B\models_\lang\theory$ and $|A|,|B|<\kappa$. As $f$ and $g$ are positive embeddings, $\Th_-(M)=\Th_-(A)=\Th_-(B)$. Since $M$ is $\kappa${\hyp}saturated for $\Th_-(M)$, we conclude that there is a homomorphism $h\map B\to M$ with $f=hg$. Hence, $M$ is $\kappa${\hyp}saturated for $\theory$.
\end{proof}

Despite the convenience of defining local positive saturation as a universal property, a natural characterisation in terms of the realisation of certain types remains.   

\begin{lem} \label{l:saturation and types} A local model $M\models_\lang\theory$ is $\kappa${\hyp}saturated for $\theory$ if and only if it realises every \gls{lp type} on $x$ over $A$ for any variable $|x|<\kappa$ and any subset $A$ with $|A|<\kappa$ such that $M$ satisfies the \gls{LJCP} over $A$. 
\end{lem}
\begin{proof} Let $M$ be $\kappa${\hyp}saturate for $\theory$ and $p\in\Stone^x(A)$ with $|x|<\kappa$ and $A$ a subset with $|A|<\kappa$ such that $M$ satisfies the \gls{LJCP} over $A$. By downwards L\"{o}wenheim{\hyp}Skolem, there is $B_A\preceq M_A$ with $|B|\leq |\lang|+|A|<\kappa$. As $B_A\preceq M_A$, we have that $\Th_-(B_A)=\Th_-(M/A)$. Then, by \cref{l:local positive types over parameters}, there is an $\lang(A)${\hyp}homomorphism $g\map B_A\to N_A$ to $N_A\models^\pc_\lang\Th_-(M/A)$ such that $N_A\models_\lang p(b)$ for some $b$. Furthermore, we can take $|N|\leq |\lang|+|B|+|x|<\kappa$ by using downwards L\"{o}wenheim{\hyp}Skolem Theorem. As $M$ is $\kappa${\hyp}saturated for $\theory$, we find a homomorphism $h\map N\to M$ such that $\id=h\circ g\map B\to M$ is the natural inclusion. In particular, $h(a^N)=h(g(a))=a$, so $h$ is an $\lang(A)${\hyp}homomorphism too. As homomorphisms preserve satisfaction of positive formulas, we conclude that $h(b)$ realises $p$ in $M$. 

Conversely, suppose $M\models_\lang\theory$ realises every \gls{lp type} on $x$ over $A$ for any variable $|x|<\kappa$ and any subset $A$ with $|A|<\kappa$ such that $M$ has the \gls{LJCP} over $A$. Let $f\map A\to M$ and $g\map A\to B$ be positive embeddings with $B\models_\lang \theory$ and $|A|,|B|<\kappa$. As $f$ and $g$ are positive embeddings, we conclude that $\Th_-(B_A)=\Th_-(A/A)=\Th_-(M/A)$ where $B_A$ and $M_A$ are the natural $\lang(A)${\hyp}expansions given by $g$ and $f$ respectively. Let $b=(b_i)_{i\in I}$ be an enumeration of $B$ and let $p$ be the \gls{lp type} of $b$ in $B_A$. Then, $p$ is a \gls{lp type} of $\Th_-(M/A)$, i.e. a \gls{lp type} of $M$ over $g(A)$. Then, there is $c=(c_i)_{i\in I}$ in $M$ realising $p$. Consider the map $h\map B\to M$ given by $h\map b_i\mapsto c_i$. Obviously, $f=h\circ g$. Also, as $b$ and $c$ have the same \gls{lp type}, they have in particular the same quantifier free positive type, so $h$ is a homomorphism. 
\end{proof}
\begin{rmk} By \cite[Theorem 3.5]{rodriguez2024completeness}, every structure has the \gls{LJCP} over any pointed set of parameters. Thus, every $\kappa${\hyp}saturated model realises every \gls{lp type} over any pointed set of parameters of size less than $\kappa$. On the other hand, note that in \cref{l:saturation and types} the proof of the right-to-left implication actually shows that a local model is $\kappa$-saturated if it realises every \gls{lp type} over a pointed set of parameters of size less than $\kappa$. 
\end{rmk}

We now show that there are saturated models in local positive logic. The proof is essentially word for word the usual one for classic first{\hyp}order logic or positive logic. We present here a sketch of the proof. 

\begin{theo} \label{t:existence saturated} For any $N\models_\lang \theory$, there is a homomorphism $h\map N\to M$ to a $\kappa${\hyp}saturated $M\models^\pc_\lang\theory$. 
\end{theo}
\begin{proof} Without loss of generality assume $\kappa\geq |\lang|$. By \cite[Theorem 2.9]{rodriguez2024completeness}, find a homomorphism $f_{-1}\map N\to M_0$ with $M_0\models^\pc_\lang\theory$. Take a cardinal $\lambda$ such that $\kappa\leq \cf(\lambda)$ (for example, $\lambda=2^\kappa$). Applying \cite[Theorem 2.9 and Lemma 2.8]{rodriguez2024completeness} repeatedly, we define a direct sequence $(M_i,f_{i,j})_{i,j\in \lambda, i\geq j}$ of \gls{lpc} models of $\theory$ such that, for any $i,j\in \lambda$ with $j\leq i$ and any positive embedding $g\map A\to B$ from $A\preceq^+ M_j$ to $B\models_\lang\theory$ with $|A|,|B|<\kappa$, there is a homomorphism $h\map B\to M_i$ with ${f_{ji}}_{\mid A}=h\circ g$.

Take $M= \underrightarrow{\lim}_{i\in\lambda} M_i$, $f_i$ the corresponding coprojections and $f=f_0\circ f_{-1}\map N\to M$. By \cite[Lemma 2.8]{rodriguez2024completeness}, $M$ is a \gls{lpc} model of $\theory$. Let $f\map A\to M$ and $g\map A\to B$ be positive embeddings where $B\models_\lang \theory$ with $|A|,|B|<\kappa$. As $\kappa<\cf(\lambda)$, we get $\Ima\, f\subseteq \Ima\, f_i$ for some $i\in \lambda$. As $f$ and $f_i$ are positive embeddings, there is a positive embedding $f'\map A\to M_i$ with $f=f_i\circ f'$. By construction, there is $h\map B\to M_{i+1}$ such that $f'=h\circ g$, so $f=h'\circ g$ with $h'=f_i\circ h$. Hence, $M$ is $\kappa${\hyp}saturated.
\end{proof}
\begin{rmk} \label{r:lowenheim skolem saturation} Using L\"{o}wenheim{\hyp}Skolem Theorem we can add a bound on the size of $M$ in \cref{t:existence saturated}. Explicitly, if $|N|\leq 2^\kappa$ and $\kappa\geq|\lang|$,  we can add $|M|\leq 2^{\kappa}$. Indeed, taking $\lambda=2^\kappa$ and using downwards L\"{o}wenheim{\hyp}Skolem Theorem at each step, we can take $|M_i|\leq 2^{\kappa}$, $\mu_i\leq 2^{\kappa}$ and $|M^\eta_i|\leq 2^{\kappa}$ for each $i,\eta\in 2^{\kappa}$.
\end{rmk}

One useful consequence of saturation is homogeneity. In local positive logic, we have the following version of homogeneity for pointed tuples.

\begin{lem} \label{l:homogeneity for pointed} Let $M\models^\pc_\lang\theory$ be $\kappa${\hyp}saturated and $a$ and $a'$ pointed tuples in $M$ with $|a|=|a'|<\kappa$ such that $\ltp_+(a)=\ltp_+(a')$. Pick $b$ with $|b|<\kappa$. Then, there is $b'$ such that $\ltp_+(ab)=\ltp_+(a'b')$.
\end{lem}
\begin{proof} Say $a,a'\in M^y$. Consider $p(x)\coloneqq\{\varphi(x,a')\sth \varphi(x,a)\in \ltp_+(b/a)\}$. We claim that $p(x)$ is a \gls{lp type} over $a'$. As $\ltp_+(b/a)$ contains a bound over $a$, $p(x)$ also contains a bound over $a'$. Obviously, $p(x)$ is closed under finite conjunctions. For any $\varphi(x,a)\in \ltp_+(b/a)$, we have that $\exists x\mathrel{}\varphi(x,y)\in\tp_+(a)$. By \cref{l:partial positive types}, as $\ltp_+(a)=\ltp_+(a')$ and $a$ and $a'$ are pointed, we conclude that $\tp_+(a)=\tp_+(a')$, so $\exists x\mathrel{}\varphi(x,y)\in\tp_+(a')$.  Consequently, $p(x)$ is finitely satisfiable. Hence, by \cref{l:boundedly satisfiable}, $p(x)$ is a partial \gls{lp type} over $a'$. Finally, if $\varphi(x,a')\notin p(x)$, then $\varphi(x,a)\notin \ltp_+(b/a)$, so $M\not\models_\lang \varphi(b,a)$. By \cite[Lemma 2.12]{rodriguez2024completeness}, there is a \gls{lp formula} $\psi(x,y)$ with $M\models_\lang \varphi(x,y)\perp\psi(x,y)$ such that $M\models_\lang\psi(b,a)$. Thus, $\psi(x,a)\in\ltp_+(b/a)$, so $\psi(x,a')\in p(x)$ with $M\models_\lang \psi(x,a')\perp\varphi(x,a')$. As $\varphi(x,a')$ is arbitrary, we conclude that $p(x)$ is a \gls{lp type} over $a'$ by \cref{l:local positive types over parameters and denials}.

Since $M$ is $\kappa${\hyp}saturated, by \ref{l:saturation and types}, there is $b'$ realising $p(x)$, so $\ltp_+(b'/a')=p(x)$. Hence, for any \gls{lp formula} $\varphi(x,y)$, we have that $M\models_\lang\varphi(b',a')$ if and only if $\varphi(x,a')\in p(x)$, if and only if $\varphi(x,a)\in\ltp_+(b/a)$, if and only if $M\models_\lang\varphi(b,a)$. In conclusion, $\ltp_+(ab)=\ltp_+(a'b')$. 
\end{proof}

In addition to saturation, it is often desirable to ensure the existence of enough automorphisms; that is strong homogeneity. In model theory, it is customary to work with saturated and strongly homogeneous (for a sufficiently big cardinal) structures; these structures are often referred to as \emph{monster} models. By adopting the methodology for constructing monster models in classic first{\hyp}order logic, we can construct corresponding monster models in local positive logic. We provide here a sketch of the proof.

\begin{defi} Let $M\models_\lang\theory$ and $\kappa$ a cardinal. 
%We say that $M$ is a \emph{locally positively $\kappa${\hyp}homogeneous} if, for any set of parameters $A$ with $|A|<\kappa$ such that $M$ has the \gls{LJCP} over $A$ and any tuples $a$ and $b$ with $|a|=|b|<\kappa$ such that $\ltp_+(a/A)=\ltp_+(b/A)$, for any $a'$ finite there is $b'$ with $\ltp_+(aa'/A)=\ltp_+(bb'/A)$. 
We say that $M$ is \emph{locally positively strongly $\kappa${\hyp}homogeneous} if, for any pointed tuples $a$ and $b$ with $|a|=|b|<\kappa$ such that $\ltp_+(a)=\ltp_+(b)$, there is an automorphism $\sigma\in\Aut(M)$ with $\sigma(a)=b$. 
\end{defi}

From now on, for the sake of simplicity, we will abbreviate the terminology by writing ``strongly homogeneous'' in place of ``locally positively strongly homogeneous''. 

\begin{theo} \label{t:existence strongly homogeneous} For any $N\models_\lang \theory$, there is a homomorphism $h\map N\to M$ to a $\kappa${\hyp}saturated and strongly $\kappa${\hyp}homogeneous $M\models^\pc_\lang\theory$. 
\end{theo}
\begin{proof}
Let $\lambda\geq |\lang|\oplus|N|$ be a strong limit cardinal with $\cf(\lambda)\geq \kappa$. %--- for instance, take $\lim_{i<2^\kappa}\lambda_i$ with $\lambda_i=\beth_{\lim_{j<i}\lambda_j}$. 
Using \cref{t:existence saturated,r:lowenheim skolem saturation}, we define by recursion a direct system $(M_\eta,f_{\eta\nu})_{\nu,\eta\in\lambda, \eta>\nu}$ of \gls{lpc} models of $\theory$ with $|M_\eta|\leq 2^{(|\eta|\oplus|\lang|)^+}$ for each $\eta$ such that $M_\eta\models^\pc_\lang \theory$ is $(|\eta|\oplus|\lang|)^+${\hyp}saturated and there is a homomorphism $h_0\map N\to M_0$. %Indeed, by \cref{t:existence saturated,r:lowenheim skolem saturation}, we can find $M_0$. Similarly, by \cref{t:existence saturated,r:lowenheim skolem saturation}, given $M_\eta$ we can find $M_{\eta+1}$. Finally, for a limit ordinal $\eta$, assuming $(M_i,f_{ij})_{j,i\in\eta, i>j}$ is already constructed, we take first the direct limit structure $\tilde{M}_\eta$. Now, $|\tilde{M}_\eta|<2^{(|\eta|\oplus|\lang|)^+}$ and, by \cite[Lemma 2.8]{rodriguez2024completeness}, $\tilde{M}_\eta\models^\pc_\lang\theory$. Then, we apply \cref{t:existence saturated,r:lowenheim skolem saturation} again and find $f\map \tilde{M}_\eta\to M_\eta$. 
Let $M$ be the direct limit of $(M_\eta,f_{\eta\nu})_{\nu,\eta\in\lambda \eta>\nu}$. Obviously, $|M|\leq \lambda$ and, by \cite[Lemma 2.8]{rodriguez2024completeness}, $M\models^\pc_\lang\theory$. Without loss of generality, we identify $M_\eta$ with its image in $M$ via the corresponding coprojection. 
By \cref{l:saturation and types}, as $\kappa<\cf(\lambda)$, $M$ is $\kappa${\hyp}saturated. %We prove first that $M$ is $\kappa${\hyp}saturated. Take $A\subseteq M$ arbitrary with $|A|<\kappa$ such that $M$ has the \gls{LJCP} over $A$. Since $\kappa<\cf(\lambda)$, we have that $A\subseteq M_\eta$ for some $\eta>\kappa$. As $\Th_-(M_\eta/A)=\Th_-(M/A)$, we get that $M_\eta$ has the \gls{LJCP} over $A$. Therefore, as $M_\eta$ is $\kappa${\hyp}saturated, for any \gls{lp type} $p(x)$ over $A$ with $|x|<\kappa$, there is a realisation of $p(x)$ in $M_\eta$, so in $M$. By \cref{l:saturation and types}, we conclude that $M$ is $\kappa${\hyp}saturated.

We now prove that $M$ is strongly $\kappa${\hyp}homogeneous. Take pointed tuples $a$ and $b$ with $|a|=|b|<\kappa$ such that $\ltp_+(a)=\ltp_+(b)$. Since $\kappa<\cf(\lambda)$, there is $\eta>\kappa$ such that $M_\eta$ contains $a$ and $b$. Let $(a_i)_{i<\lambda}$ and $(b_i)_{i<\lambda}$ be enumerations of $M$ with $a=(a_i)_{i<\alpha}$ for some $\alpha<\eta$ and $\{a_i,b_i\}_{i<\xi}$ contained in $M_\xi$ for each $\xi>\eta$. By back{\hyp}and{\hyp}forth, using \cref{l:homogeneity for pointed}, we recursively define a chain of partial maps $(f_\xi)_{\xi\in\lambda}$ preserving satisfaction of \glspl{lp formula} such that $f_\xi\map a_i\mapsto b_i$ for $i<\alpha$ and $a_i\in \Dom f_\xi\subseteq M_{\eta+\xi}$ and $b_i\in \Ima f_\xi\subseteq M_{\eta+\xi}$ for all $i<\xi$. %We set $f_\xi\map a_i\mapsto b_i$ for each $i<\alpha$ and $\xi\leq \alpha$. For a limit ordinal $\xi$, we simply take the union $f_\xi=\bigcup_{i<\xi} f_i$. Now, suppose $f_\xi$ is defined with $\xi\geq\alpha$. Let $d_\xi$ enumerate the domain of $f_\xi$. Note that $d_\xi$ and $f(d_\xi)$ are contained in $M_{\eta+\xi}\subseteq M_{\eta+\xi+1}$. Also, $a_\xi$ and $b_\xi$ are contained in $M_{\eta+\xi+1}$. Now, $M_{\eta+\xi+1}$ is $|\eta+\xi+1|^+${\hyp}saturated (so $|\xi|^+${\hyp}saturated). Hence, by \cref{l:homogeneity for pointed}, there is $a'_\xi$ in $M_{\eta+\xi+1}$ such that $\ltp_+(d_\xi a_\xi)=\ltp_+(f_\xi(d_\xi) a'_\xi)$. By \cref{l:homogeneity for pointed} again, there is $b'_\xi$ in $M_{\eta+\xi+1}$ such that $\ltp_+(d_\xi a_\xi b'_\xi)=\ltp_+(f_\xi(d_\xi) a'_\xi b_\xi)$. Take $f_{\xi+1}$ extending $f_\xi$ by mapping $a_\xi\mapsto a'_\xi$ and $b'_\xi\mapsto b_\xi$. 
Take $f=\bigcup_{\xi<\lambda} f_\xi$. Then, $f\map M\to M$ is surjective with $f(a)=b$ and preserves satisfaction of \glspl{lp formula}. In particular, $f$ is an homomorphism. As $M$ is \gls{lpc}, $f$ is a positive embedding, so it is also injective. Therefore, $f$ is an automorphism.    
\end{proof}

\begin{defi} \label{d:retractor}
Let $f\map A\to B$ be a homomorphism. A \emph{retraction} of $f$ is a homomorphism $g\map B\to A$ such that $g\circ f=\id$. A \emph{(local) retractor of $\theory$} is a local model $\retractor\models_\lang\theory$ such that every homomorphism $f\map \retractor\to N$ to $N\models_\lang\theory$ has a retraction.
\end{defi}

\begin{theo} \label{t:retractors} Let $\retractor\models_\lang \theory$. Then, $\retractor$ is a retractor of $\theory$ if and only if $\retractor$ is a $\kappa${\hyp}saturated \gls{lpc} model of $\theory$ for every cardinal $\kappa$.
\end{theo}
\begin{proof} Suppose that $\retractor\models^\pc_\lang\theory$ is $\kappa${\hyp}saturated for every cardinal $\kappa$. Let $f\map \retractor\to A$ be a homomorphism with $A\models_\lang\theory$. Since it is \gls{lpc}, $f$ is a positive embedding. By saturation with $\id\map \retractor\to \retractor$, there is a homomorphism $g\map A\to \retractor$ such that $\id=g\circ f$, so $g$ is a retraction of $f$.

On the other hand, suppose $\retractor$ is a retractor of $\theory$. Take a homomorphism $h\map \retractor\to M$ to $M\models_\lang\theory$. Since $\retractor$ is a retractor, there is $g\map M\to \retractor$ such that $\id=g\circ h$. Take $\varphi\in\For^x_+(\lang)$ and $a\in \retractor^x$. Suppose $M\models_\lang\varphi(h(a))$. Then, $\retractor\models_\lang\varphi(gh(a))$, so $\retractor\models_\lang\varphi(a)$ as $gh(a)=a$. Hence, we conclude that $\retractor$ is \gls{lpc}.

Let $A$ be a subset such that $\retractor$ satisfies the \gls{LJCP} over $A$ and let $p(x)$ be a \gls{lp type} of $\retractor$ over $A$. By \cref{l:local positive types over parameters}, there is an $\lang(A)${\hyp}homomorphism $f\map \retractor_A\to N_A$ with $N_A\models^\pc_\lang \Th_-(\retractor/A)$ and $b\in N^x$ such that $N_A\models_\lang p(b)$. Since $\retractor$ is a retractor, there is a homomorphism $g\map N\to \retractor$ such that $\id=g\circ f$. As $f$ is an $\lang(A)${\hyp}homomorphism, we conclude that $g$ is an $\lang(A)${\hyp}homomorphism. Thus, $\retractor_A \models_\lang p(g(b))$, concluding that $\retractor_\lang$ realises $p(x)$. Since $p(x)$ and $A$ are arbitrary, we conclude that $\retractor$ is $\kappa${\hyp}saturated for any cardinal $\kappa$ by \cref{l:saturation and types}
\end{proof}

The following corollary is \cite[Proposition 2.1(3)]{hrushovski2022lascar}. %We give the proof for the sake of completeness.
\begin{coro} \label{c:endomorphism of retractor} Let $\retractor$ be a retractor of $\theory$. Then, every endomorphism of $\retractor$ is an automorphism. 
\end{coro}
%\begin{proof} Let $f\map \retractor\to \retractor$ be a homomorphism. Since $\retractor$ is a retractor, there is $g\map \retractor\to\retractor$ such that $\id=g\circ f$. In particular, $g$ is surjective. By \cref{t:retractors}, $\retractor\models^\pc_\lang\theory$, so $g$ is a positive embedding. In particular, $g$ is injective. Hence, $g$ is an automorphism and, \emph{a posteriori}, $f=g^{-1}$ is an automorphism too.\end{proof}

The following corollary is \cite[Proposition 2.1(4)]{hrushovski2022lascar}. We provide an alternative proof of it.

\begin{coro} \label{c:homogeneity of retractor} Let $\retractor$ be a retractor of $\theory$. Let $A$ be a subset of $\retractor$ and $x$ a variable pointed over $\lang(A)$. Take $a,b\in \retractor^x$. Then, $\ltp^{\retractor}_+(a/A)=\ltp^{\retractor}_+(b/A)$ if and only if there is an automorphism $\sigma\in \Aut(\retractor/A)$ such that $\sigma(a)=b$. In particular, $\ltp^{\retractor}_+(a/A)=\ltp^{\retractor}_+(b/A)$ if and only if $\tp(a/A)=\tp(b/A)$.
\end{coro}
\begin{proof} The backward direction is obvious. %if $\sigma(a)=b$ for some $\sigma\in\Aut(\retractor/A)$, then $\ltp^{\retractor}_+(a/A)=\ltp^{\retractor}_+(b/A)$.
On the other hand, consider the family $\mathcal{K}$ of partial homomorphisms of $\retractor$ fixing $A$ and mapping $a$ to $b$. Assuming $\ltp^{\retractor}_+(a/A)=\ltp^{\retractor}_+(b/A)$, we have that $\mathcal{K}\neq\emptyset$. Obviously, $\mathcal{K}$ is partially ordered by $\subseteq$. Also, for any chain $\Omega\subseteq \mathcal{K}$, we have that $\bigcup \Omega\in\mathcal{K}$. Therefore, by Zorn{'s} Lemma, there is a maximal element $\sigma\in \mathcal{K}$. Take an enumeration $d$ of $\Dom\, \sigma$. Let $c\in\retractor$ be an arbitrary single element. As $d$ contains $a$ and $A$, we have that $d$ is pointed. By saturation of $\retractor$, we conclude that there is an element $c'$ such that $\ltp^{\retractor}_+(cd)=\ltp^{\retractor}_+(c'\sigma(d))$. Therefore, $\sigma\cup \{(c,c')\}\in\mathcal{K}$, concluding by maximality that $c\in \Dom\,\sigma$. As $c$ is arbitrary, we conclude that $\sigma\map \retractor\to\retractor$ is an endomorphism. By \cref{c:endomorphism of retractor}, we conclude that $\sigma\in\Aut(\retractor/A)$.
\end{proof}

\begin{defi} Let $M\models_\lang \theory$ and $\kappa$ a cardinal. We say that $M$ is \emph{locally $\kappa${\hyp}homomorphism{\hyp}universal for $\theory$} if for every $N\models_\lang \theory$ with $|N|<\kappa$ there is a homomorphism $f\map N\to M$. We say that it is \emph{locally absolutely homomorphism{\hyp}universal for $\theory$} if it is locally $\kappa${\hyp}homomorphism{\hyp}universal for every cardinal $\kappa$.\end{defi}
From now on, for the sake of simplicity, we will abbreviate the terminology by writing ``universal'' in place of ``locally homomorphism{\hyp}universal''. 
 
%The following lemma is \cite[Proposition 2.1(6)]{hrushovski2022lascar}. %We give the proof for the sake of completeness.
\begin{lem} \label{l:universality of retractor} Assume $\theory$ has the \gls{LJCP} and  $\kappa>|\lang|$. Then, every $\kappa${\hyp}saturated model of $\theory$ is $\kappa${\hyp}universal for $\theory$. In particular, when $\theory$ satisfies the \gls{LJCP}, there is at most one retractor of $\theory$ up to automorphism and it is absolutely universal for $\theory$. 
\end{lem}
\begin{proof} See \cite[Proposition 2.1(6)]{hrushovski2022lascar}.\end{proof}

Without the \gls{LJCP}, retractors are not necessarily unique. However, we can still classify them in a natural way, getting also a bound on the number of retractors. In \cite{hrushovski2022lascar}, Hrushovski suggested a way of doing so; he suggested considering all the possible types and when they are compatible. We provide here a rephrasing of that idea that seems to work better in the many sorted case.  

\begin{defi}\label{d:compatible} We say that two local models $A$ and $B$ of $\theory$ are \emph{(locally) compatible in $\theory$} if there are homomorphisms $f\map A\to M$ and $g\map B\to M$ to a common local model $M\models_\lang \theory$; write $A\frown_{\theory} B$. For a local model $A\models_\lang\theory$, we write $\Com_{\theory}(A)\coloneqq\{B\models_\lang\theory\sth A\frown_{\theory}B\}$ and $\Com^\pc_{\theory}(A)\coloneqq\{B\in\Com_{\theory}(A)\sth B\models^\pc_\lang\theory\}$. As usual, we omit $\theory$ if it is clear from the context.
\end{defi}
\begin{lem}\label{l:compatible} Let $M\models^\pc_\lang\theory$. Then, for any $A_1,A_2\in\Com(M)$, $A_1\frown A_2$. In particular, $\frown$ is an equivalence relation in the class $\Mod^\pc(\theory)$ and $\Com^\pc(M)$ is the equivalence class of $M$ for every $M\models^\pc_\lang\theory$.
\end{lem}
\begin{proof} As $A_1\frown M$ and $A_2\frown M$. There are then homomorphisms $f_1\map A_1\to B_1$, $f_2\map M\to B_1$, $g_1\map M\to B_2$ and $g_2\map A_2\to B_2$ with $B_1,B_2\models_\lang\theory$. As $M\models^\pc_\lang\theory$, we have that $f_2$ and $g_1$ are positive embeddings, so $B_1,B_2\models_\lang\Th_-(M/M)$. As $M$ is pointed, by \cite[Theorem 3.5]{rodriguez2024completeness}, there are $\lang(M)${\hyp}homomorphisms $h_1\map B_1\to N$ and $h_2\map B_2\to N$ to a common local model $N\models_\lang \Th_-(M/M)$. In particular, $h_1\circ f_1\map A_1\to N$ and $h_2\circ g_2\map A_2\to N$ are homomorphisms with $N\models_\lang\theory$. Therefore, $A_1\frown A_2$.
\end{proof}

\begin{theo}\label{t:relative retractor} Let $\mathcal{E}$ be an equivalence class of compatibility in $\Mod^\pc(\theory)$. Then:
\begin{enumerate}[label={\rm{(\arabic*)}}, ref={\rm{\arabic*}}, wide]
\item \label{itm:relative retractor:universal} $\retractor\in\mathcal{E}$ is a retractor of $\theory$ if and only if it is universal for $\mathcal{E}=\Com^\pc(\retractor)$.
\item \label{itm:relative retractor:uniqueness} There is at most one retractor in $\mathcal{E}$ up to isomorphism.
\item \label{itm:relative retractor:existence} There is a retractor in $\mathcal{E}$ if and only if $\mathcal{E}$ is bounded, i.e. there is a cardinal $\kappa$ such that every element of $\mathcal{E}$ has cardinality at most $\kappa$.  
\end{enumerate}
\end{theo}
\begin{proof}
\begin{enumerate}[wide] 
\item[\hspace{-1.2em}{\rm{(\ref*{itm:relative retractor:universal})}}] Pick $A\in\mathcal{E}$. As $A\frown\retractor$, there are homomorphisms $f\map A\to B$ and $g\map \retractor\to B$ to $B\models_\lang\theory$. Now, there is a retraction $h\map B\to\retractor$ of $g$. In particular, $h\circ f\map A\to \retractor$ is a homomorphism. Since $A$ is arbitrary, we conclude that $\retractor$ is universal for $\mathcal{E}$.

Now, suppose $\retractor$ is universal for $\mathcal{E}$. By \cref{t:existence saturated}, there is a homomorphism $h\map \retractor\to \retractor'$ where $\retractor'$ is a $\lambda${\hyp}saturated \gls{lpc} model of $\theory$ for any $\lambda>|\retractor|$. As $\retractor'\in\mathcal{E}$, we conclude that there is a positive embedding $g\map \retractor'\to\retractor$. In particular, we have that $|\retractor'|\leq |\retractor|<\lambda$. By $\lambda${\hyp}saturation of $\retractor'$, there is a homomorphism $f\map \retractor\to\retractor'$ such that $\id=f\circ g$, in particular, $f$ is surjective. As $\retractor$ is \gls{lpc} for $\theory$, $f$ is a positive embedding. Hence, $f$ is a surjective embedding, that is an isomorphism. Hence, $\retractor$ is $\lambda${\hyp}saturated for $\theory$. As $\lambda$ is arbitrary, we conclude that it is a retractor of $\theory$.
\item[{\rm{(\ref*{itm:relative retractor:uniqueness})}}] Let $\retractor_1$ and $\retractor_2$ be compatible retractors of $\theory$. By {\rm{(\ref*{itm:relative retractor:universal})}}, there is a homomorphism $f\map \retractor_1\to \retractor_2$. Now, considering $\id\map \retractor_1\to \retractor_1$, by saturation of $\retractor_1$, there is a homomorphism $g\map \retractor_2\to \retractor_1$ such that $\id=g\circ f$. Again, considering $\id\map \retractor_2\to \retractor_2$, by saturation of $\retractor_2$, there is a homomorphism $f'\map \retractor_1\to \retractor_2$ such that $\id=f'\circ g$. Hence, $g$ is an isomorphism, $\retractor_1\cong\retractor_2$ and, \emph{a posteriori}, $f=f'=g^{-1}$.
\item[{\rm{(\ref*{itm:relative retractor:existence})}}] By {\rm{(\ref*{itm:relative retractor:universal})}}, if $\retractor$ is a retractor in $\mathcal{E}$, then $\retractor$ is universal for $\mathcal{E}$, so every element of $\mathcal{E}$ has cardinality at most $|\retractor|$. On the other hand, suppose $\mathcal{E}$ is bounded. Since it is bounded, we can take a set $\mathcal{E}_0$ of representatives of $\mathcal{E}$ such that every element of $\mathcal{E}$ is isomorphic to exactly one element of $\mathcal{E}_0$. Now, $\mathcal{E}_0$ is partially ordered by $\lesssim$, where $A\lesssim B$ if there is an embedding $f\map A\to B$. Furthermore, note that $\lesssim$ is a directed order since any $A_1,A_2\in\mathcal{E}_0$ are compatible. Now, take a chain $\{A_i\}_{i\in I}$ in $\mathcal{E}_0$. By \cite[Lemma 2.6]{rodriguez2024completeness}, $A\coloneqq\underrightarrow{\lim}A_i$ is a \gls{lpc} model of $\theory$ and $A_i\lesssim A$ for every $i\in I$. Therefore, by Zorn{'s} Lemma, there are maximal elements in $\mathcal{E}_0$. As $\lesssim$ is directed, we conclude that, in fact, $\mathcal{E}_0$ has a maximum $\retractor$. Thus, $\retractor$ is universal for $\Com^\pc(\retractor)=\mathcal{E}_0$ and, by (\ref*{itm:relative retractor:universal}), it is a retractor. \qedhere
\end{enumerate}
\end{proof}

\begin{rmk} \label{lowenheim skolem number of retractors} By downwards L\"owenheim{\hyp}Skolem Theorem, using \cite[Corollary 2.13]{rodriguez2024completeness}, every equivalence class of compatibility in $\Mod^\pc(\theory)$ contains at least one element $A$ with $|A|\leq |\lang|$. Therefore, there are at most $2^{|\lang|}$ possible equivalence classes. In particular, $\theory$ has at most $2^{|\lang|}$ different retractors. 
\end{rmk}
\begin{ex} \label{e:sharp bound number retractors}
The bound $2^{|\lang|}$ on the number of retractors is sharp:

Let $I$ be a set and $\lang$ the one{\hyp}sorted local language with unary predicates $\{P_i,Q_i\}_{i\in I}$ and locality relations $\Dtt=\{\dd_n\}_{n\in\N}$ (with the usual ordered monoid structure induced by the natural numbers). Consider the $\lang${\hyp}structure $M$ with universe $2^I$ and interpretations $M\models \dd_n(\eta,\zeta)\Leftrightarrow |\{i\sth \eta(i)\neq\zeta(i)\}|\leq n$ for $n\in\N$, $M\models P_i(\eta)\Leftrightarrow \eta(i)=1$ and $M\models Q_i(\eta)\Leftrightarrow \eta(i)=0$ for each $i\in I$. Let $\theory=\Th_-(M)$ be the negative theory of $M$. Since $M\models\theory_\loc$, $\theory$ is locally satisfiable by \cite[Theorem 1.8]{rodriguez2024completeness}. For $\eta\in M$, write $M_\eta\coloneqq\Dtt(\eta)$.

Now, we look at $M$ as no{\hyp}local structure, i.e. as $\lang_\star${\hyp}structure. Note that the positive logic topology in $M$ is coarser than the Tychonoff topology, so $M$ is compact with the positive logic topology. Thus, by \cite[Lemma 2.24]{segel2022positive}, we get that $M$ is universal for $\Mod_\star(\theory)$. In particular, by \cref{l:non-local retractors}, $\theory$ has retractors.

Let $\eta\in M$ and suppose $M\models_\lang \varphi$ with $\varphi$ a positive formula in $\lang$. For $I_0\subseteq I$ finite, let $\lang_{I_0}$ be the reduct of $\lang$ consisting of $\{\dd_k\}_{k\in\N}\cup \{P_i,Q_i\sth i\in I_0\}$ and take $I_0$ such that $\varphi$ is a formula in $\lang_{I_0}$. Consider the map $f\map M\to M_\eta$ given by $f(\xi)(i)=\xi(i)$ for $i\in I_0$ and $f(\xi)(i)=\eta(i)$ for $i\notin I_0$. Obviously, $P_i(\xi)$ implies $P_i(f(\xi))$ and $Q_i(\xi)$ implies $Q_i(f(\xi))$ for any $i\in I_0$. On the other hand, $|\{i\sth f(\xi)(i)\neq f(\zeta)(i)\}|\leq |\{i\sth \xi(i)\neq\zeta(i)\}|$, so $\dd_k(\xi,\zeta)$ implies $\dd_k(f(\xi),f(\zeta))$ for any $k$. Thus, $f$ is an $\lang_{I_0}${\hyp}homomorphism. As homomorphisms preserve satisfaction, we conclude that $M_\eta\models_\lang\varphi$. As $\varphi$ is arbitrary, we conclude that $\Th_-(M_\eta)=\theory$ for any $\eta\in M$. Hence, $\theory$ is a local theory. 

Note that $M_\eta\not\frown M_\zeta$ for $\eta,\zeta\in M$ with $\{i\sth \zeta(i)\neq \eta(i)\}$ infinite. In order to reach a contradiction, suppose that there are homomorphisms $f\map M_\zeta\to N$ and $g\map  M_\zeta\to N$ with $N\models_\lang \theory$. Then, $\bigwedge_{i\in I}R_i(f(\eta),g(\zeta))$ where $R_i(x,y)=(P_i(x)\wedge Q_i(y))\vee (Q_i(x)\wedge P_i(y))$ and $I=\{i\sth \zeta(i)\neq\eta(i)\}$. On the other hand, for any finite subset $J\subseteq I$, $\theory\models_\lang \neg\exists x y\mathrel{} \left(\bigwedge_{i\in J}R_i(x,y)\wedge \dd_{|J|-1}(x,y)\right)$. Thus, we get $N\not\models_\lang \dd_n(f(\eta),g(\zeta))$ for every $n\in\N$, contradicting that $N$ is a local structure.
\end{ex}

The following easy lemma provides a useful variation of \cref{t:relative retractor}(\cref{itm:relative retractor:universal}) which is also a converse to \cref{c:endomorphism of retractor}.

\begin{lem} \label{l:retractor iff endomorphism are automorphisms} Let $\retractor\models_\lang\theory$. Suppose that $\retractor$ is universal for $\Com(\retractor)$ and every endomorphism of $\retractor$ is an automorphism. Then, $\retractor$ is a retractor of $\theory$.
\end{lem}
\begin{proof} Let $f\map \retractor\to M$ be a homomorphism to $M\models_\lang \theory$. By universality of $\retractor$ for $\Com(\retractor)$, there is a homomorphism $g\map M\to \retractor$. Thus, $h\coloneqq g\circ f$ is an endomorphism, so an automorphism. Therefore, $h^{-1}\circ g\map M\to\retractor$ is a retraction of $f$, concluding that $\retractor$ is a retractor. 
\end{proof}

\subsection{Local positive logic topologies} \label{s:local positive logic topologies}
Recall that, unless otherwise stated, we have fixed a local language $\lang$ and a locally satisfiable negative local theory $\theory$. 

Let $M$ be a local $\lang${\hyp}structure, $x$ a variable and $A$ a subset. We say that $D\subseteq M^x$ is \emph{locally positively $A${\hyp}definable} if there is a \gls{lp formula} $\underline{D}(x)$ with parameters in $A$ defining it, i.e. $D=\{a\in M^x\sth M\models_\lang\underline{D}(a)\}$. We say that $D$ is \emph{locally positively $\bigwedge_A${\hyp}definable} if there is $\underline{D}(x)\subseteq \LFor^x_+(\lang(A))$ defining it, i.e. $D=\{a\in M^x\sth M\models_\lang\underline{D}(a)\}$. The \emph{local positive $A${\hyp}logic topology} of $M^x$ is the topology on $M^x$ given by taking as closed sets the locally positively $\bigwedge_A${\hyp}definable subsets. %The \emph{local positive $A${\hyp}logic topology} of $M$ is the sorted family of local positive $A${\hyp}logic topologies. 
The local positive $M${\hyp}logic topology is simply called the \emph{\gls{lp logic topology}}.

\begin{lem} \label{l:product logic topologies} Let $M$ be a local $\lang${\hyp}structure, $x,y$ two disjoint variables and $A$ a subset. Then, the local positive $A${\hyp}logic topology in $M^x\times M^y$ is finer than the product topology from the local positive $A${\hyp}logic topologies of $M^x$ and $M^y$.
\end{lem}
\begin{proof} It suffices to show that the Cartesian projections are continuous. By symmetry, we just show that $\proj_x\map M^x\times M^y\to M^x$ is continuous. Let $V\subseteq M^x$ be locally positively $\bigwedge_A${\hyp}definable. Then, $\proj^{-1}_x(V)$ is trivially locally positively $\bigwedge_A${\hyp}definable, being defined by the same partial type $\underline{V}$ as $V$ but with an extra dummy variable $y$.
\end{proof}

The main aim of this subsection is to prove the following two fundamental results, one being the converse of the other. The first was already proven in \cite[Lemma 2.3]{hrushovski2022lascar}. %We give here the proof for the sake of completeness.
\begin{theo} \label{t:compactness of retractors} Let $\retractor$ be a retractor of $\theory$. Then, every ball of $\retractor$ is compact in the \gls{lp logic topology}.
\end{theo}
%\begin{proof} Consider a ball $\dd(a)$ in $\retractor$. Let $\{V_i\}_{i\in I}$ be a family of locally positively $\bigwedge_{\retractor}${\hyp}definable subsets of $\retractor$ such that $\bigcap_{i\in I_0} V_i\cap \dd(a)\neq \emptyset$ for any finite $I_0\subseteq I$. Consider $\underline{V}(x)=\bigwedge_{i\in I} \underline{V}_i(x)\wedge \dd(x,a)$. Note that $\dd(x,a)$ is a bound of $x$ over $a$. By hypothesis, every finite subset of $\underline{V}$ is satisfiable in $\retractor$, so $\underline{V}$ is boundedly satisfiable. By \cref{l:boundedly satisfiable}, we conclude that $\underline{V}$ is locally satisfiable in $\Th_-(\retractor/\retractor)$. By saturation of $\retractor$ (\cref{t:retractors}), we conclude that there is $b\in V=\bigcap V_i\cap \dd(a)$. As $\{V_i\}_{i\in I}$ is arbitrary, we conclude that $\dd(a)$ is compact.\end{proof}
The next theorem is a stronger version of \cite[Lemma 2.4]{hrushovski2022lascar}, closer to the original version for positive logic.
\begin{theo} \label{t:compactness and positively closed implies retractor} Let $M\models^\pc_\lang\theory$. Suppose that every ball in a single sort of $M$ is compact in the \gls{lp logic topology}. Then, $M$ is a retractor of $\theory$.
\end{theo}
\begin{proof} Let $f\map A\to M$ be a positive embedding and $g\map A\to B$ a homomorphism with $A,B\models_\lang\theory$. Consider the family $\mathcal{K}$ of partial homomorphisms $h\map D\subseteq B\to M$ with $f=h\circ g$ and preserving satisfaction of \glspl{lp formula} (i.e. if $B\models_\lang \varphi(d)$ with $d\in D$ and $\varphi$ local positive, then $M\models_\lang\varphi(h(d))$). By \cite[Corollary 2.13]{rodriguez2024completeness}, as $f$ is a positive embedding and $M\models^\pc_\lang\theory$, we have that $A\models^\pc_\lang\theory$. Therefore, $g\map A\to B$ is a positive embedding, so $h_0=f\circ g^{-1}\map \Ima\, g\to M$ is a partial homomorphism from $B$ to $M$ with $f=h_0\circ g$. If $B\models_\lang\varphi(g(a))$ with $\varphi$ \gls{lp formula}, then $A\models_\lang\varphi(a)$ as $g$ is a positive embedding, so $M\models_\lang \varphi(f(a))$, so $h$ preserves satisfaction of \glspl{lp formula}. In particular, $\mathcal{K}\neq \emptyset$. Obviously, $\mathcal{K}$ is partially ordered by $\subseteq$ and, for any chain $\Omega\subseteq \mathcal{K}$, $\bigcup \Omega\in\mathcal{K}$. By Zorn{'s} Lemma, there is a maximal element $h\in\mathcal{K}$. Write $D\coloneqq \Dom\ h$ and take an arbitrary single element $b\in B$. For some $a\in A$ and locality relation $\dd$, we have $b\in \dd(g(a))$. Consider $\Gamma(x,D)=\ltp^{B}_+(b/D)$. Then, for any $\Delta(x,D)\subseteq \Gamma(x,D)$ finite, we have that $B\models_\lang \exists x\in\dd(g(a))\mathrel{} \bigwedge \Delta(x,D)$. Since $h$ preserves satisfaction of \glspl{lp formula}, we have that $M\models_\lang \exists x\in\dd(a)\mathrel{} \bigwedge \Delta(x,h(D))$. Therefore, $\Delta(M,h(D))\cap \dd(a)\neq \emptyset$ for any $\Delta\subseteq \Gamma$ finite. By compactness of $\dd(a)$ in the \gls{lp logic topology}, we conclude that there is $c\in \dd(a)$ such that $M\models_\lang\Gamma(c,h(D))$. Thus, $h\cup \{(b,c)\}\in \mathcal{K}$, concluding by maximality that $b\in D$. As $b$ is arbitrary, we conclude that $h\map B\to M$ is a homomorphism with $f=h\circ g$. 
\end{proof}
Without local positive closedness, $M$ is not necessarily a retractor. Indeed, the \gls{lp logic topology} of a reduct is always coarser than the \gls{lp logic topology} on the original language, so every reduct of a retractor still satisfies that balls are compact on the \gls{lp logic topology}. However, we know that the reduct of a \gls{lpc} structure might be no \gls{lpc}, so the reduct of a retractor is not necessarily a retractor. %For instance, consider the language with only one sort and two unary relations $P$ and $Q$ and the negative theory axiomatised by $P\perp Q$. The retractor of this theory is the structure with only two elements; one element satisfying $P$ and one element satisfying $N$. However, in the reduct given by dropping out $Q$, the retractor is the structure with only one element satisfying $P$.

However, it is worth noting that, in the proof of \cref{t:compactness and positively closed implies retractor}, the fact that $M$ is \gls{lpc} for $\theory$ is only used to conclude that $A$ is a \gls{lpc} model of $\theory$. Therefore, dropping out this hypothesis, we can get the following useful variation.
\begin{theo}[Universality Lemma] \label{t:universality lemma} Let $M\models_\lang\theory$. Suppose that every ball in a single sort in $M$ is compact in the \gls{lp logic topology}. Then, $M_A$ is universal for $\Th_-(M/A)$ for any substructure $A\leq M$. In particular, if $A\models^\pc_\lang\theory$, $M$ is universal for $\Com^\pc(A)$ and so there is a retractor of $\theory$ at $A$. 
\end{theo}
\begin{proof} Let $B_A\models_\lang\Th_-(M/A)$. Write $f\map A\to M$ for the inclusion and let $g\map A\to B$ be the map given by $a\mapsto a^{B_A}$. Consider the family $\mathcal{K}$ of partial homomorphisms $h\map D\subseteq B\to M$ preserving satisfaction of \glspl{lp formula} such that $f=h\circ g$. 

First, note that $\mathcal{K}\neq \emptyset$. If $a_1\neq a_2$, then $\neg a_1=a_2\in \Th_-(M/A)$, so $a^{B_A}_1\neq a^{B_A}_2$. Hence, $g$ is injective. On the other hand, since $B_A\models_\lang \Th_-(M/A)$, $g$ preserves satisfaction of local negative formulas, so $g^{-1}\map \Ima\, g\to A$ preserves satisfaction of \glspl{lp formula}. Therefore, $h_0=f\circ g^{-1}\map \Ima\, g\to M$ is a partial homomorphism from $B$ to $M$ with $f=h_0\circ g$ preserving satisfaction of \glspl{lp formula}, concluding that $h_0\in \mathcal{K}$. 

Obviously, $\mathcal{K}$ is partially ordered by $\subseteq$ and, for any chain $\Omega\subseteq \mathcal{K}$, $\bigcup \Omega\in\mathcal{K}$. By Zorn{'s} Lemma, there is a maximal element $h\in\mathcal{K}$. As in the proof of \cref{t:compactness and positively closed implies retractor}, we conclude that $\Dom\, h=B$, so, in other words, $h\map B_A\to M_A$ is an $\lang(A)${\hyp}homomorphism. %Indeed, write $D\coloneqq \Dom\, h$ and take an arbitrary single element $b\in B$. For some $a\in A$ and locality relation $\dd$, we have $b\in \dd(g(a))$. Consider $\Gamma(x,D)=\ltp^B_+(b/D)$. Then, for any $\Delta(x,D)\subseteq \Gamma(x,D)$ finite, we have $B\models_\lang \exists x\in\dd(g(a))\mathrel{} \bigwedge \Delta(x,D)$. Using that $h$ preserves \glspl{lp formula}, we have that $M\models_\lang \exists x\in\dd(a)\mathrel{} \bigwedge \Delta(x,h(D))$. Therefore, $\Delta(M,h(D))\cap \dd(a)\neq \emptyset$ for any $\Delta\subseteq \Gamma$ finite. By compactness of $\dd(a)$, we conclude that there is $c\in \dd(a)$ such that $M\models_\lang\Gamma(c,h(D))$. Thus, $h\cup \{(b,c)\}\in \mathcal{K}$, concluding by maximality that $b\in D$. As $b$ is arbitrary, we conclude that $h\map B\to M$ is a homomorphism with $f=h\circ g$.  

Suppose now that $A\models^\pc_\lang\theory$. Then, we have that $B_A\models_\lang \Th_-(M/A)$ if and only if $B\models_\lang\theory$ and there is a homomorphism from $A$ to $B$ --- where $B$ is the $\lang${\hyp}reduct of $B_A$ and $B_A$ is the $\lang(A)${\hyp}expansion of $B$ given by the homomorphism from $A$ to $B$. %Indeed, if $B_A\models_\lang\Th_-(M/A)$, obviously $B\models_\lang \theory$ and $g\map a\mapsto a^{B}$ defines a homomorphism from $A$ to $B$. On the other hand, as $A\models^\pc_\lang\theory$, if $B\models_lang \theory$ and there is a homomorphism $g\map A\to B$, then $B_A\models_\lang \Th_-(A/A)=\Th_-(M/A)$ via $a^{B_A}=g(a)$.
Now, let $B\models_\lang\theory$ with $A\frown B$. There is $N\models_\lang\theory$ and homomorphisms from $A$ to $N$ and from $B$ to $N$. It follows that $N_A\models_\lang \Th_-(M/A)$, so there is a homomorphism from $N$ to $M$ by universality of $M$ for $\Th_-(M/A)$. In particular, there is a homomorphism from $B$ to $M$. As $B\in\Com(A)$ is arbitrary, we conclude that $M$ is universal for $\Com(A)$. In particular, $|N|\leq |M|$ for every $N\in\Com^\pc(A)$, so there is a retractor of $\theory$ at $A$ by \cref{t:relative retractor}(\ref{itm:relative retractor:existence}).
\end{proof}

\begin{lem}\label{l:ori condition} Let $M$ be a local structure. Suppose that every single sort admits a topology such that balls are compact and every relation symbol (except equality) is closed in the corresponding product topology. Then, balls are compact in the \gls{lp logic topology}.
\end{lem}
\begin{proof} According to the statement, pick a topology $\mathcal{T}_s$ in $M^s$, for every single sort $s$, such that every ball in $M^s$ is compact in $\mathcal{T}_s$ and every relation $R$ (except equality) of sort $s_1\ldots s_n$ is closed in $M^{s_1}\times\cdots\times M^{s_n}$ with the corresponding product topology. It suffices to show that, for every ball, the subspace topology induced by $\mathcal{T}$ is finer than the \gls{lp logic topology}. 

Let $\varphi(x)=\exists y_1\ldots y_n\in\dd(x)\mathrel{}\phi(x,y,\bar{c})$ be a \gls{lp formula} with $x$ single variable, $c=(c_1,\ldots,c_m)$ parameters and $\phi$ quantifier free positive. Without loss of generality, we may assume that $\varphi$ contains no equality, as it can be always replaced by identifying the corresponding variables and constants. Take a ball $\dd_0(e)$ in $M^x$ where $x$ is a single variable. Consider $\psi(x,c,e)\coloneqq\exists y_1\ldots y_n\in \dd\ast \dd_0(e)\mathrel{}\psi'(x,y,c,e)$ with $\psi'(x,y,c,e)\coloneqq \phi(x,y,c)\wedge\bigwedge\dd(y_i,x)\wedge \dd_0(x,e)$. Note that $\psi(M,c,e)=\varphi(M,c)\cap\dd_0(e)$.

Let $a\notin \psi(M,c,e)$. Write $P\coloneqq \prod^m_{i=1}\{c_i\} \times\{e\}$. Then, $\{a\}\times\prod^n_{i=1} \dd\ast \dd_0(e)\times P$ is disjoint to $\psi'(M)$. Now, $\psi'(M)$ is closed in the product topology. Hence, for any $b=(b_1,\ldots,b_n)\in \prod^n_{i=1}\dd\ast\dd_0(e)$, there are $\mathcal{T}${\hyp}open sets $A_b\subseteq M^x$ and $B_b\subseteq M^y$ with $a\in A_b$ and $b\in B_b$ such that $A_b\times B_b\times P$ and $\psi'(M)$ are disjoint. Now, $\prod^n_{i=1}\dd\ast\dd_0(e)$ is compact, so there are finitely many $\mathcal{T}${\hyp}open sets $A_1,\ldots,A_k\subseteq M^x$ and $B_1,\ldots,B_k\subseteq M^y$ such that $A_j\times B_j\times P$ and $\psi'(M)$ disjoint for each $j\in\{1,\ldots,k\}$ with $a\in \bigcap^k_{j=1}A_j$ and $\prod^n_{i=1}\dd\ast\dd_0(e)\subseteq \bigcup^k_{j=1}B_j$. Consider the $\mathcal{T}${\hyp}open set $A\coloneqq\bigcap^k_{j=1} A_j$, neighbourhood of $a$. Pick $a'\in A$ and $b'\in\prod^n_{i=1}\dd\ast\dd_0(e)$. Now, $b'\in B_j$ for some $j\in\{1,\ldots,k\}$, so $(a',b',c,e)\in  A_j\times B_j\times P$, concluding that $(a',b',c,e)$ is not in $\psi'(M)$. Consequently, $A\times  \prod^n_{i=1}\dd\ast\dd_0(e)\times P$ is disjoint to $\psi'(M)$, so $A$ is disjoint to $\psi(M,c,e)$. As $a$ is arbitrary, $\varphi(M,c)\cap\dd_0(e)=\psi(M,c,e)$ is $\mathcal{T}${\hyp}closed. Since $\varphi(x,c)$ is arbitrary, the subspace topology in $\dd_0(e)$ induced by $\mathcal{T}$ is finer than the \gls{lp logic topology}. As $\dd_0(e)$ is compact in $\mathcal{T}$, we conclude that $\dd_0(e)$ is compact in the \gls{lp logic topology}. 
\end{proof}

\begin{defi} \label{d:entourage} Let $\dd_0$ and $\dd$ be locality relations. An \emph{entourage of $\dd_0$ in $\dd$} is a pair of \glspl{lp formula} $\psi(x,y)$ and $\varphi(x,y)$ such that $\psi(x,y)\leq \dd(x,y)$, $\varphi(x,y)\perp \dd_0(x,y)$ and $\psi(x,y)\toprel \varphi(x,y)$. We say that a pair of \glspl{lp formula} is an \emph{entourage} if it is an entourage for some locality relations $\dd_0$ and $\dd_0$. We say that $\theory$ has \emph{entourages} if it has entourages in all sorts. 
\end{defi}
\begin{rmk} If $(\psi,\varphi)$ is an entourage of $\dd_0$, then $\dd_0\leq\psi$ by definition of $\psi\toprel\varphi$.\end{rmk}  
\begin{lem} \label{l:basic entourages}  
\begin{enumerate}[label={\rm{(\arabic*)}}, ref={\rm{\arabic*}}, wide, topsep=0pt]
\item[\hspace{-1.2em}\setcounter{enumi}{1}\rm{(\theenumi)}] %\label{itm:basic entourages:finer} 
Let $(\psi,\varphi)$ and $(\psi',\varphi')$ be entourages of $\dd_0$ in $\dd$. Then, $(\psi\wedge\psi',\varphi\vee\varphi')$ is an entourage of $\dd_0$ in $\dd$.
\item \label{itm:basic entourages:coarser} Let $(\psi,\varphi)$ be an entourage of $\dd_0$ in $\dd$. Let $\psi'$ and $\varphi'$ be \glspl{lp formula} with $\psi\leq\psi'\leq \dd$ and $\varphi\leq \varphi'$ with $\varphi'\perp \dd_0$. Then, $(\psi',\varphi')$ is an entourage of $\dd_0$ in $\dd$.
\item \label{itm:basic entourages:approximation} Let $(\psi,\varphi)$ be an entourage of $\dd_0$ in $\dd$. If $\dd'_0\leq\dd_0$ and $\dd\leq\dd'$, then $(\psi,\varphi)$ is an entourage of $\dd'_0$ in $\dd'$. In particular, every entourage is an entourage of equality.
\item \label{itm:basic entourages:product} Let $(\psi(x,y),\varphi(x,y))$ be an entourage of $\dd_0(x,y)$ in $\dd(x,y)$ and let $(\psi'(x',y')$, $\varphi'(x',y'))$ be an entourage of $\dd'_0(x',y')$ in $\dd'(x',y')$. Then, $(\psi(x,y)\wedge \psi'(x',y')$, $\varphi(x,y)\vee\varphi'(x',y'))$ is an entourage of $\dd_0\times\dd'_0$ in $\dd\times\dd'$.
\end{enumerate}
\end{lem}
\begin{proof} Elementary by \cite[Lemma 2.15]{rodriguez2024completeness}. 
\end{proof}

\begin{lem} \label{l:entourages} Let $M\models^\pc_\lang\theory$. Let $\dd$ be a locality relation and $(\psi,\varphi)$ an entourage in $\dd$. Then, every $\dd${\hyp}ball is a closed neighbourhood of its centre in the \gls{lp logic topology} of $M$. Furthermore, $\psi(M,a)$ is a closed neighbourhood of $a$. 
\end{lem}
\begin{proof} Obviously, $\psi(M,a)$ and $\varphi(M,a)$ are closed subsets of $M^x$ in the \gls{lp logic topology}. As $\varphi(x,y)\perp x=y$, $a\notin \varphi(M,a)$. On the other hand, as $\psi\toprel \varphi$ and $M\models^\pc_\lang\theory$, we get that $\varphi(M,a)\cup \psi(M,a)=M^x$ by \cite[Lemma 2.15(2)]{rodriguez2024completeness}. Thus, $a\in \neg\varphi(M,a)\subseteq\psi(M,a)$, concluding that $\psi(M,a)$ is a closed neighbourhood of $a$. Finally, as $\psi\leq \dd$ and $M\models^\pc_\lang\theory$, by \cite[Lemma 2.15(2)]{rodriguez2024completeness},  we get $\psi(M,a)\subseteq \dd(a)$, so $\dd(a)$ is a closed neighbourhood of $a$ as well.
\end{proof}

\begin{coro} \label{c:weak local compactness retractor} Let $\retractor$ be a retractor of $\theory$. Suppose $\dd$ has an entourage. Then, for every $a$, the $\dd${\hyp}ball at $a$ is a compact neighbourhood of $a$ in the \gls{lp logic topology}. In particular, $\retractor$ is weakly locally closed compact in the \gls{lp logic topology} whenever $\theory$ has entourages.
\end{coro}
\begin{proof} By \cref{t:compactness of retractors,l:entourages}.   
\end{proof}

\begin{defi}\label{d:full system entourages} Let $\dd$ be a locality relation. A \emph{full system of entourages} is a family $\{(\psi_i,\varphi_i)\}_{i\in I}$ of entourages in $\dd$ such that $\{\psi_i\}_{i\in I}$ is a full system of approximations of $=$. We say that $\theory$ has \emph{full systems of entourages} if it has a full system of entourages for every sort. \end{defi}

\begin{coro} \label{c:local compactness retractor} Let $\retractor$ be a retractor of $\theory$. Let $\dd(x,y)$ be a locality relation and $\{(\psi_i,\varphi_i)\}_{i\in I}$ a full system of entourages in $\dd$. Then, for every $a\in\retractor^y$, the family $\{\bigwedge_{i\in I_0}\psi_i(\retractor,a)\sth I_0\subseteq I\text{ finite}\}$ is a local base of closed compact neighbourhoods of $a$ in the \gls{lp logic topology}. In particular, $\retractor$ is locally closed compact in the \gls{lp logic topology} whenever $\theory$ has full systems of entourages.
\end{coro}
\begin{proof} By \cref{l:basic entourages}, $(\bigwedge_{i\in I_0}\psi_i,\bigvee_{i\in I_0}\varphi_i)$ is an entourage of $\dd$ for any $I_0\subseteq I$ finite. Therefore, by \cref{c:weak local compactness retractor}, $\{\bigwedge_{i\in I_0}\psi_i(\retractor,a)\sth I_0\subseteq I\text{ finite}\}$ is a family of compact neighbourhoods of $a$ in the \gls{lp logic topology}. Let $\phi(x)$ be a \gls{lp formula} with $\retractor\not\models_\lang\phi(a)$. Since $\{\psi_i\}_{i\in I}$ is a full system of approximations, $\bigcap_{i\in I}\psi_i(\retractor,a)=\{a\}$. Thus, $\bigcap_{i\in I}\psi_i(\retractor,a)\cap \phi(\retractor)=\emptyset$. By compactness of $\dd(a)$ (\cref{t:compactness of retractors}), we conclude that there is $I_0\subseteq I$ finite with $\bigcap_{i\in I_0}\psi_i(\retractor,a)\cap \phi(\retractor)=\emptyset$, so $\bigwedge_{i\in I_0}\psi_i(\retractor,a)\subseteq \neg\phi(\retractor)$.  
\end{proof}

\subsection{Local positive compactness} \label{s:local positive compactness}
Recall that, unless otherwise stated, we have fixed a local language $\lang$ and a locally satisfiable negative local theory $\theory$. 

We say that $\theory$ is \emph{\gls{lpco}}\footnote{Since Hrushovski preferably works with (local) primitive positive formulas, this property is called ``(local) primitive positive compactness'', abbreviated {$\mathrm{(ppC)}$}, in \cite{hrushovski2022lascar}. In his paper, Hrushovski attributes the origin of this property to Walter Taylor.} if for any \gls{lp formula} $\varphi(x,y,z)$ with $\left(\exists z\mathrel{} \varphi(x,y,z)\right)\perp x=y$, any bound $\Bound(x,z)$ and any locality relation symbol $\dd(x,y)$, where $x$ and $y$ are different single variables on the same sort and $z$ is an arbitrary finite variable, there is a number $k\in \N_{>0}$ such that 
\[\theory\models_\lang \neg\exists x_1 \ldots x_k z\mathrel{}  \bigwedge_{i\neq j} \Bound(x_i,z)\wedge \dd(x_i,x_j)\wedge\varphi(x_i,x_j,z).\] 

\begin{lem} \label{l:local positive compactness} A theory is \gls{lpco} if and only if there are no infinite non{\hyp}constant locally positively indiscernible sequences over pointed sets of parameters in \gls{lpc} models. 
\end{lem}
\begin{proof} Suppose there is a \gls{lpc} model $M\models^\pc_\lang\theory$ with a non{\hyp}constant sequence $(b_i)_{i\in\omega}$ which is locally positively $A${\hyp}indiscernible over a pointed set of parameters $A$. By \cite[Lemma 2.13]{rodriguez2024completeness}, as $M$ is \gls{lpc} and $b_0\neq b_1$, there is a positive formula $\varphi_0(x,y)\in\For^{xy}_+(\lang)$ such that $\varphi_0(x,y)\perp x=y$ and $M\models_\lang\varphi_0(b_0,b_1)$. Say $\varphi_0=\exists w\mathrel{} \phi_0(x,y,w)$ with $\phi_0$ quantifier free positive. Take a enumeration $a\in M^z$ of $A$. There is a bound $\Bound(x,a)$ over $a$ realised by $b_0$. As $(b_i)_{i\in\omega}$ is locally positively $A${\hyp}indiscernible, we get that $M\models_\lang\Bound(b_i,a)$ for all $i\in \omega$. Also, there is a locality relation $\dd_1$ such that $M\models_\lang\varphi_1(b_0,b_1,a)$ with $\varphi_1(x,y,z)\coloneqq\exists w\in\dd_1(z_w)\mathrel{} \phi_0(x,y,w)$, where $z_w\subseteq z$ is a subtuple of the sort of $w$. Take $\varphi(x,y,z)=\varphi_1(x,y,z)\vee\varphi_1(y,x,z)$ and note that $\varphi(x,y,z)$ is a \gls{lp formula} with $\left(\exists z\mathrel{} \varphi(x,y,z)\right)\perp x=y$ and $M\models_\lang\varphi(b_0,b_1,a)$. As $(b_i)_{i\in \N}$ is locally positively $A${\hyp}indiscernible, we conclude that $M\models_\lang \varphi(b_i,b_j,a)$ for any $i\neq j$. Hence, $M\models_\lang \bigwedge_{i\neq j} \Bound(b_i,a)\wedge\dd(b_i,b_j)\wedge\varphi(b_i,b_j,a)$. Thus, we have $\exists x_1\ldots x_k z\mathrel{} (\bigwedge_{i\neq j}\Bound(x_i,z)\wedge\dd(x_i,x_j)\wedge\varphi(x_i,x_j,z))\in \theory_+$ for every $k\in \N$ as witnessed by $(b_i)_{i\in\omega}$ and $a$. In particular, $\theory$ is not \gls{lpco}.

Conversely, suppose $\theory$ is not \gls{lpco}. Then, there are a a locality relation $\dd$, a bound $\Bound(x,z)$ and a \gls{lp formula} $\varphi(x,y,z)$ such that $\left(\exists z\mathrel{}\varphi(x,y,z)\right)\perp x=y$ and $\bigwedge_{i\neq j} \Bound(x_i,z)\wedge\dd(x_i,x_j)\wedge\varphi(x_i,x_j,z)$ is finitely satisfiable. As it implies a bound, $\bigwedge_{i\neq j} \Bound(x_i,z)\wedge\dd(x_i,x_j)\wedge\varphi(x_i,x_j,z)$ is boundedly satisfiable. By \cref{l:boundedly satisfiable,l:partial local positive types}, there is $M\models^\pc_\lang\theory$ with $(b'_i)_{i\in\omega}$ and $a$ in $M$ such that $M\models_\lang \bigwedge_{i\neq j}\Bound(b'_i,a)\wedge\dd(b'_i,b'_j)\wedge\varphi(b'_i,b'_j,a)$. Take any pointed set $A$ containing $a$. By \cref{l:standard lemma}, there is $N\models^\pc_\lang\Th_-(M/A)$ with a locally positively $A${\hyp}indiscernible sequence $(b_i)_{i\in\omega}$ such that, for any $\psi\in\LFor_+(\lang(A))$, $N\models_\lang\psi(b_{\bar{\jmath}})$ whenever $M\models_\lang\psi(b'_{\bar{\imath}})$ for every $\bar{\imath}$ with $\qftp(\bar{\imath})=\qftp(\bar{\jmath})$. In particular, we get that $N\models_\lang \varphi(b_i,b_j,a)$ for all $i\neq j$. As $\left(\exists z\mathrel{}\varphi(x,y,z)\right)\perp x=y$ and $M\models_\lang\exists z\mathrel{}\varphi(b'_i,b'_j,a)$ for all $i\neq j$, we conclude that $(b_i)_{i\in\omega}$ is an infinite non{\hyp}constant locally positively $A${\hyp}indiscernible sequence with $A$ pointed. Since $A$ is pointed, we have the \gls{LJCP} over $A$ \cite[Theorem 3.5]{rodriguez2024completeness}, so we can assume that there is a positive $\lang(A)${\hyp}embedding $f\map M\to N$. By \cite[Lemma 2.5]{rodriguez2024completeness}, we get that $N\models^\pc_\lang\theory$. Hence, there is a non{\hyp}constant locally positively indiscernible sequence over a pointed set of parameters in some \gls{lpc} model of $\theory$.
\end{proof}

In \cite[Proposition 2.1(1)]{hrushovski2022lascar}, Hrushovski proved that every \gls{lpco} theory has retractors. Furthermore, in \cite{hrushovski2022lascar}, Hrushovski briefly argues (implicitly assuming the \gls{LJCP}) that local positive compactness, existence of retractors and boundedness of \gls{lpc} models are equivalent properties. Here, in this subsection, we provide the full proof of these equivalences under weak completeness and show how the strength of the equivalence varies under different completeness conditions. 

We start reproving that local positive compactness implies the existence of retractors. The following is exactly \cite[Proposition 2.1(1)]{hrushovski2022lascar}. %We provide the full proof for sake of completeness. 
\begin{theo}\label{t:local positive compactness} Suppose $\theory$ is \gls{lpco}. Then, every $|\lang|^+${\hyp}saturated \gls{lpc} model of $\theory$ is a retractor. In particular, $\theory$ has retractors.
\end{theo}
\begin{proof} See \cite[Proposition 2.1(1)]{hrushovski2022lascar}, noting that no irreducibility condition (especially \gls{LJCP}) is used for the proof of this item. 
\end{proof}

We now list all the conditions equivalent to existence of retractors. We start with a version that does not require irreducibility conditions:
\begin{theo} \label{t:local positive compactness and relative retractors} Let $M\models^\pc_\lang \theory$. Then, the following are equivalent:
\begin{enumerate}[label={\rm{(\arabic*)}}, ref={\rm{\arabic*}}, wide]
\item \label{itm:local positive compactness and relative retractors:compactness} There is a homomorphism $h\map M\to N$ with $N\models_\lang \theory$ whose balls in single sorts are compact in the \gls{lp logic topology}.
\item \label{itm:local positive compactness and relative retractors:existence retractor} $\theory$ has a retractor at $M$.
\item \label{itm:local positive compactness and relative retractors:bounded pc} $\Com^\pc(M)$ is bounded.
\item \label{itm:local positive compactness and relative retractors:local positive compactness} $\Th_-(M/M)$ is \gls{lpco} (in $\lang(M)$). 
\item \label{itm:local positive compactness and relative retractors:|L|+-saturated models are retractors} Every $|\lang(M)|^+${\hyp}saturated \gls{lpc} model $N$ of $\theory$ such that $M$ is a positive substructure of $N$ is a retractor of $\theory$.
\end{enumerate}
\end{theo}
\begin{proof} 
\begin{enumerate}[wide]
\item[\hspace{-1.2em}(\ref*{itm:local positive compactness and relative retractors:compactness})$\Leftrightarrow$(\ref*{itm:local positive compactness and relative retractors:existence retractor})] By \cref{t:universality lemma,t:compactness of retractors}.
\item[(\ref*{itm:local positive compactness and relative retractors:existence retractor})$\Leftrightarrow$(\ref*{itm:local positive compactness and relative retractors:bounded pc})] By \cref{t:relative retractor}(\ref{itm:relative retractor:existence}).
\item[(\ref*{itm:local positive compactness and relative retractors:existence retractor})$\Rightarrow$(\ref*{itm:local positive compactness and relative retractors:local positive compactness})] Let $x$ be a single variable and $z$ a finite variable. Take a bound $\Bound(x,z)$ of $x$ in $\Th_-(M/M)$, a \gls{lp formula} $\varphi(x,y,z)\in\LFor_+(\lang(M))$ with $\Th_-(M/M)\models_\lang \left(\exists z\mathrel{}\varphi(x,y,z)\right)\perp x=y$ and a locality relation $\dd(x,y)$. In order to reach a contradiction, suppose that 
\[M\models_\lang \exists x_1 \ldots x_k z\mathrel{} \big(\bigwedge_{i\neq j} \Bound(x_i,z)\wedge\dd(x_i,x_j)\wedge\varphi(x_i,x_j,z)\big)\] 
for every $k\in\N$. Consider the partial \gls{lp type} 
\[\Gamma(\bar{x},z)=\bigwedge_{i\neq j}\Bound(x_i,z)\wedge\dd(x_i,x_j)\wedge\varphi(x_i,x_j,z)\]
with $\bar{x}=\{x_i\}_{i\in \N}$ countable set of single variables. By supposition, $\Gamma(\bar{x},z)$ is finitely locally satisfiable in $\Th_-(M/M)$. As $\Gamma(\bar{x},z)$ implies a bound, by \cref{l:boundedly satisfiable}, we conclude that $\Gamma(\bar{x},z)$ is locally satisfiable in $\Th_-(M/M)$. Let $\retractor$ be a retractor of $\theory$ for $M$. As $M$ is \gls{lpc}, $\Th_-(\retractor/M)=\Th_-(M/M)$. By \cref{l:saturation and parameters}, $\retractor_M$ is a retractor of $\Th_-(M/M)$. By \cref{t:retractors}, $\retractor$ satisfies $\Gamma(\bar{x},z)$. Therefore, there are $(a_i)_{i\in \N}$ and $c$ in $\retractor$ with $a_i\in\dd(a_0)$ for each $i\in\N$ and $\retractor\models_\lang \varphi(a_i,a_j,c)$ for each $i\neq j$ --- in particular, $a_i\neq a_j$ for $i\neq j$. Applying Zorn{'s} Lemma, find $(a_i)_{i\in \lambda}$ maximal such that $a_i\in \dd(a_0)$ for every $i\in \lambda$ and $\retractor\models_\lang \varphi(a_i,a_j,c)$ for each $i\neq j$ --- in particular, $a_i\neq a_j$ for $i\neq j$, so $\lambda\leq |\retractor|$. Now, consider $U_i=\neg \varphi(\retractor,a_i,c)\cup \neg \varphi(a_i,\retractor,c)$. By maximality of $(a_i)_{i\in\lambda}$, it follows that $\dd(a_0)\subseteq \bigcup_{i\in \lambda} U_i$. By \cref{t:compactness of retractors}, $\dd(a_0)$ is compact in the \gls{lp logic topology}. As $\{U_i\}_{i\in\lambda}$ is an open covering, there is $I_0\subseteq \lambda$ finite such that $\dd(a_0)\subseteq \bigcup_{i\in I_0} U_i$. As $I_0$ is finite, there is $n\in\N$ with $n\notin I_0$. Then, by choice of $(a_i)_{i\in \lambda}$, we have that $a_n\in\dd(a_0)$ and $a_n\notin\bigcup_{i\in I_0} U_i$, getting a contradiction. 
\item[(\ref*{itm:local positive compactness and relative retractors:local positive compactness})$\Rightarrow$(\ref*{itm:local positive compactness and relative retractors:|L|+-saturated models are retractors})] By \cref{t:local positive compactness}.
\item[(\ref*{itm:local positive compactness and relative retractors:|L|+-saturated models are retractors})$\Rightarrow$(\ref*{itm:local positive compactness and relative retractors:existence retractor})] By \cref{t:existence saturated}.\qedhere
\end{enumerate}
\end{proof}
\begin{coro} \label{c:bound of the size of retractors} We have $|\retractor|\leq 2^{|\lang|}$ for every retractor $\retractor$ of $\theory$. 
\end{coro}
\begin{proof} By L\"{o}wenheim{\hyp}Skolem, take $M\preceq \retractor$ with $|M|\leq |\lang|$. Using \cref{t:existence saturated,r:lowenheim skolem saturation}, find an $|\lang|^+${\hyp}saturated $\retractor_0\lang_\lang\theory$ with $|\retractor_0|\leq 2^{|\lang|}$ and $M\leq \retractor_0$. Since $M\preceq \retractor$ and $\retractor\models^\pc_\lang\theory$, we have $M\models^\pc_\lang\theory$ by \cite[Corollary 2.13]{rodriguez2024completeness}. Thus, $M\preceq^+\retractor_0$. As $\retractor$ is a retractor of $\theory$ at $M$, by \cref{t:local positive compactness and relative retractors}, $\Th_-(M)$ is \gls{lpco}. As $M\preceq^+\retractor_0$, $\Th_-(\retractor_0)=\Th_-(M)$. By \ref{l:saturation and parameters}, $\retractor_0$ is $|\lang|^+${\hyp}saturated for $\Th_-(M)$. By \cref{t:local positive compactness}, we conclude that $\retractor_0$ is a retractor of $\Th_-(M)$. By \ref{l:saturation and parameters}, as $\retractor_0\models^\pc_\lang\theory$, we conclude that it is a retractor of $\theory$ too. Note that $\retractor\frown M\frown\retractor_0$ and $M\models^\pc_\lang\theory$. By \cref{l:compatible}, $\retractor\frown\retractor_0$. By \cref{t:relative retractor}(\ref{itm:relative retractor:uniqueness}), we conclude that $\retractor\cong\retractor_0$. In particular, $|\retractor|\leq 2^{|\lang|}$.
\end{proof}
\begin{coro} \label{c:exist retractors iff pc bounded} The following are equivalent:
\begin{enumerate}[label={\rm{(\arabic*)}}, wide]
\item $\theory$ has retractors.
\item There is some cardinal $\kappa$ such that $|M|<\kappa$ for every $M\models^\pc_\lang\theory$.
\item For every $M\models^\pc_\lang\theory$, $|M|\leq 2^{|\lang|}$.
\end{enumerate}
\end{coro}
We now state the equivalences under irreducibility assumptions.
\begin{coro} \label{c:exist retractors iff local positive compactness} Assume $\theory$ is weakly complete. Then, $\theory$ has retractors if and only if it is \gls{lpco}.
\end{coro}
\begin{proof} By \cref{t:local positive compactness}, if $\theory$ is \gls{lpco}, then $\theory$ has retractors. On the other hand, suppose that $\theory$ is weakly complete and has retractors. Let $\retractor$ be a retractor at $M$ with $M\models^\pc_\lang\theory_{\pm}$. As homomorphisms preserve satisfaction, we get $\Th_-(\retractor)=\theory$. By \cref{t:local positive compactness and relative retractors}, $\Th_-(\retractor/\retractor)$ is \gls{lpco} so, in particular, $\theory=\Th_-(\retractor)$ is \gls{lpco}. 
\end{proof}
\begin{coro} \label{c:exist retractors iff exists one retractor} Assume $\theory$ is complete. Then, $\theory$ has retractors if and only if it has at least one retractor.
\end{coro}
\begin{proof} Suppose $\theory$ is complete and has at least one retractor $\retractor$. By \cref{t:local positive compactness and relative retractors}, we conclude that $\Th_-(\retractor/\retractor)$ is \gls{lpco}. In particular, $\Th_-(\retractor)$ is \gls{lpco}. By completeness, $\theory=\Th_-(\retractor)$ is \gls{lpco}. Then, by \cref{t:local positive compactness}, $\theory$ has retractors.
\end{proof}

\begin{ex} \cref{c:exist retractors iff exists one retractor} may not hold if $\theory$ is merely weakly complete. Indeed, take the one{\hyp}sorted local language $\lang$ with locality relations $\Dtt=\{\dd_n\}_{n\in\N}$ (with the usual ordered monoid structure induced by the natural numbers), unary predicates $\{P_n,Q_n\}_{n\in\N}$ and a binary relation $\prec$. 

Consider the local structure in $\R$ given by $P_n(x)\Leftrightarrow x\geq n$, $Q_n(x)\Leftrightarrow x\leq -n$, $x\prec y\Leftrightarrow 0<x<y$ and $\dd_n(x,y)\Leftrightarrow |x-y|\leq n$. Let $\theory=\Th_-(\R)$. Then, $\theory$ is not \gls{lpco}. Indeed, we can take $\varphi(x,y)\coloneqq x\prec y\vee y\prec x$, that satisfies $\varphi(x,y)\perp x=y$, and the locality relation $\dd_1$. In that case, we find that $S=\{\sfrac{1}{i}\}_{i\in\N_{>0}}$ satisfies that, for any two distinct $i,j$, we have $\varphi(i,j)\wedge\dd_1(i,j)$. Thus, by \cref{c:exist retractors iff local positive compactness}, $\theory$ does not have retractors (that is, there is some models of $\theory$ incompatible with all retractors).

On the other hand, let $\widetilde{\R}$ be a (non{\hyp}local) $\omega${\hyp}saturated elementary extension of $\R$. Then, since $\bigwedge_n Q_n(x)$ is finitely satisfiable, there is some $a\in\widetilde{\R}$ realising it. Consider the substructure $A$ with universe $\{a\}$. Since $\widetilde{\R}\cong \R$, we have $\widetilde{\R}\premodels\theory$, so $A\models_\lang\theory$. Furthermore, $A$ is a retractor as, if $g\map A\to M$ is a homomorphism with $M\models_\lang\theory$, then the constant function $f\map M\to A$ mapping everything to $a$ is homomorphism too (so a retractor). Indeed, if there is a homomorphism $g\map A\to M$, we must have that $P_n(M)=\prec(M)=\emptyset$, since $M\models_\lang Q_{m+1}(h(a))$ while $\neg\exists x y\mathrel{} (Q_{m+1}(x)\wedge P_n(y)\wedge\dd_{m+1}(x,y))\in\theory$ and $\neg\exists x y z\mathrel{} (Q_{m+1}(x)\wedge y\prec z\wedge\dd_m(x,y))\in\theory$ for all $m\in\N$. On the other hand, for any $n\in\N$, we have $A\models_\lang Q_n(a)$ and $A\models_\lang \dd_n(a,a)$, thus $f$ is a homomorphism.
\end{ex}

\subsection{Local retractors and non{\hyp}local universal models} \label{s:local retractors and non-local universal models}
As covered in \cite[Subsection 2.5]{hrushovski2022definability}, negative theories have a non{\hyp}local (homomorphism) universal positively closed model if and only if there is a bound on the cardinality of their non{\hyp}\gls{lpc} models. Since every negative local theory is, in particular, a negative theory, it is natural to ask whether there is a connection between the existence of local retractors and the existence of non{\hyp}local universal positively closed models. One direction is straightforward:
\begin{lem} \label{l:non-local retractors} Assume $\theory$ is a negative local theory that has a non{\hyp}local universal model $\mathfrak{U}$. Then, $\theory$ has local retractors, and every local retractor embeds into $\mathfrak{U}$. 
\end{lem}
\begin{proof}
Let $M\models^\pc_\lang\theory$. Then, since $M\premodels\theory$, there is some homomorphism $h\map M\to\mathfrak{U}$. Since $h(M)\models_\lang\theory$, we conclude that $h$ is a positive embedding into $h(M)$ and, in particular, injective. We get that $|M|\leq|\mathfrak{U}|$ for all $M\models^\pc_\lang\theory$. Thus, by \cref{t:relative retractor}(\cref{itm:relative retractor:existence}), $\theory$ has retractors.
\end{proof}

Unfortunately, this is the only connection that holds --- a theory with local retractors does not need to have a non{\hyp}local universal model. 

\begin{ex} Consider the one{\hyp}sorted local language $\lang=\{0,1,+,\cdot\}$ with locality relations $\Dtt=\{\dd_n\}_{n\in\N}$ (with the usual ordered monoid structure induced by the natural numbers), where $0,1$ are constants and $+,\cdot$ are ternary relations.

Consider the structure $\R$ with the usual field interpretation ($+,\cdot$ are the graphs of addition and multiplication respectively) and metric interpretation of the locality relations. By \cref{t:universality lemma}, $\R$ is local universal for $\theory_{\R}=\Th_-(\R)$. Since every field endomorphism of $\R$ is the identity, it is in particular an automorphism thus  $\R$ is a retractor by \cref{l:retractor iff endomorphism are automorphisms}. Since $\R$ is local universal, it is the unique retractor (up to isomorphism).

However, as a negative theory, $\theory_{\R}$ is unbounded. Indeed, we have the positive formula $x<y$, which is expressed by $\exists z w\mathrel{} (z\cdot w=1\wedge y=x+z^2)$. Now, $x<y$ denials equality in $\theory_\R$. Since $\R$ has infinite ordered sequences, by compactness, there is a positively closed model of $\theory_\R$ of any cardinality. 

More generally, in fields we can use the formula $\exists z w\mathrel{} (z\cdot w=1\wedge y=x+z)$, which positively defines inequality in every field. Hence, by compactness, for any infinite field $K$, the negative theory $\Th_-(K)$ must have positively closed models of arbitrarily large cardinality. Thus, for instance, we can consider $\C$ or $\Q_p$. Indeed:

The $p${\hyp}adic numbers $\Q_p$ as $\lang${\hyp}structure where the locality predicates are interpreted according to the $p${\hyp}adic metric is again local universal in their own theory $\theory_p=\Th_-(\Q_p)$ by \cref{t:universality lemma}. Since every field endomorphism of $\Q_p$ is also the identity, $\Q_p$ is the unique retractor of $\theory_p$ by \cref{l:retractor iff endomorphism are automorphisms}. However, it has arbitrarily large (non{\hyp}local) positively closed models as inequality is positively definable. 

Similarly, we can consider $\C$ as $\lang${\hyp}structure and observe that, again, $\theory_\C=\Th_-(\C)$ has arbitrarily large (non{\hyp}local) positively closed models, and $\C$ is local universal
 for $\theory_\C$ by \cref{t:universality lemma}. If $h\map\C\to\C$ is an $\lang${\hyp}endomorphism, it fixes $\Q$ (since it is in particular a field endomorphism). For any $a,b\in\C$, we have that $\C\models \dd_m(na,nb)$ if and only if $\|a-b\|\leq \sfrac{m}{n}$. Consequently, $h$ is continuous. Since $\R$ is the closure of $\Q$, $h$ setwise fixes $\R$, so $h_{\mid\R}$ is an endomorphism of $\R$, so it pointwise fixes $\R$. Thus, it is either the identity or complex conjugation, and in any case an automorphism. We thus get here as well that $\C$ is the unique retractor for $\theory_\C$ by \cref{l:retractor iff endomorphism are automorphisms}. 
\end{ex}

If there is a non{\hyp}local universal positively closed model $\mathfrak{U}$, one might expect that the local retractors are precisely the maximal local substructures of $\mathfrak{U}$ (i.e. the local components containing the constants). This is the case for instance in \cite[Example 3.7(3)]{rodriguez2024completeness}. However, the following example shows that this is not true in general:

\begin{ex} \label{e:local retractors inside universal models} Consider the one{\hyp}sorted local language $\lang$ with locality predicates $\Dtt=\{\dd_n\}_{n\in\N}$ (with the usual ordered monoid structure induced by the natural numbers) and binary predicates $\{S\}\cup \{E_n,U_n\}_{n\in\N}$. Consider the $\lang${\hyp}structure $M$ with universe $2^\omega$ and interpretations $M\models \dd_n(\eta,\xi)\Leftrightarrow |\{i\sth \eta(i)\neq \xi(i)\}|\leq n$, $M\models E_n(\eta,\xi)\Leftrightarrow \eta(n)=\xi(n)$, $M\models U_n(\eta,\xi)\Leftrightarrow \eta(n)\neq \xi(n)$ and $M\models S(\eta,\xi)\Leftrightarrow \eta,\xi\text{ have no common zeros}$. Let $\theory=\Th_-(M)$ be the negative theory of $M$. Since $M\models \theory_\loc$, $\theory$ is locally satisfiable by \cite[Theorem 1.8]{rodriguez2024completeness}. For $\eta\in M$, write $M_\eta\coloneqq \Dtt(\eta)$ for the local component at $\eta$. Let $1\in M$ be the constant sequence $1$. 

Now, we look at $M$ as non{\hyp}local structure, i.e. as $\lang_\ast${\hyp}structure. Note that the positive logic topology in $M$ is coarser than the Tychonoff topology, so $M$ is compact with the positive logic topology. Thus, by \cite[Lemma 2.24]{segel2022positive}, we get that $M$ is universal for $\mathcal{M}_\ast(\theory)$. Furthermore, note that, in $M$, the unique element $\eta$ satisfying $M\premodels S(\eta,\eta)$ is precisely $\eta=1$. Thus, $\eta(n)=1\Leftrightarrow M\premodels \exists x\mathrel{} (S(x,x)\wedge E_n(x,\eta))$ and $\eta(n)=0\Leftrightarrow M\premodels \exists x\mathrel{} (S(x,x)\wedge U_n(x,\eta))$. Consequently, the only endomorphism of $M$ is the identity. By \cref{l:retractor iff endomorphism are automorphisms} applied to $\lang_\star${\hyp}structures, $M$ is a non{\hyp}local universal positively closed model of $\theory$. 

Suppose $M\premodels \varphi$ with $\varphi$ a positive formula in $\lang$. For $n\in \N$, let $\lang_n$ be the reduct of $\lang$ consisting of $\{S\}\cup \{\dd_k\}_{k\in\N}\cup \{E_k,U_k\}_{k\leq n}$ and take $n$ such that $\varphi$ is a formula in $\lang_n$. Consider the map $f\map M\to M_1$ given by $f(\xi)(k)=\xi(k)$ for $k\leq n$ and $f(\xi)(k)=1$ for $k>n$. Obviously, $E_k(\xi,\zeta)$ implies $E_k(f(\xi),f(\zeta))$ and $U_k(\xi,\zeta)$ implies $U_k(f(\xi),f(\zeta))$ for any $k\leq n$. Also, $\{i\sth f(\xi)(i)=f(\zeta)(i)=0\}\subseteq\{i\sth \xi(i)=\zeta(i)=0\}$, so $S(\xi,\zeta)$ implies $S(f(\xi),f(\zeta))$. On the other hand, $|\{m\sth f(\xi)(m)\neq f(\zeta)(m)\}|\leq |\{m\sth \xi(m)\neq\zeta(m)\}|$, so $\dd_k(\xi,\zeta)$ implies $\dd_k(f(\xi),f(\zeta))$ for any $k$. Thus, $f$ is an $\lang_n${\hyp}homomorphism. Since homomorphisms preserve satisfaction, we conclude that $M_1\models_\lang\varphi$. As $\varphi$ is arbitrary, we conclude that $\Th_-(M_1)=\theory$. Hence, $\theory$ is a weakly complete local theory.

Pick $\eta\in M\setminus M_1$. We claim that $M_\eta$ is not a retractor of $\theory$. In particular, we claim that $M_\eta$ is not \gls{lpc}. Consider the map $f\map M_\eta\to M_1$ given by $f(\xi)(i)=1$ if $x(i)=\eta(i)$ and $f(\xi)(i)=0$ if $\xi(i)\neq \eta(i)$. We claim that $f$ is a homomorphism from $M_\eta$ to $M_1$ which is not a positive embedding. Obviously, 
\[\begin{aligned} \{i\sth \zeta(i)=\xi(i)\}&=\{i\sth \zeta(i)=\xi(i)=\eta(i)\}\cup \{i\sth \zeta(i)=\xi(i)\neq \eta(i)\}\\ &=\{i\sth f(\zeta)(i)=f(\xi)(i)=1\}\cup\{i\sth f(\zeta)(i)=f(\xi)(i)=0\}\\ &=\{i\sth f(\zeta)(i)=f(\xi)(i)\},\end{aligned}\]
so $f$ preserves $E_n,U_n,\dd_n$ for all $n$. As $\eta\notin M_1$, for any $\zeta,\xi\in M_\eta$, we have $M_\eta\models_\lang\neg S(\zeta,\xi)$. Hence, $f$ is a homomorphism. However, $M_1\models_\lang S(1,1)$ and $M_\eta\not\models_\lang S(\eta,\eta)$, so $f$ is not a positive embedding.  
\end{ex}

On the other hand, in \cref{e:local retractors inside universal models}, there is one local component which is \gls{lpc}, so a local retractor. In fact, it turns out that \cref{e:local retractors inside universal models} has the \gls{LJCP}, so this is the only local retractor.

\begin{que} \label{q:local components are local retractors}
Suppose $\theory$ has a non{\hyp}local universal positively closed model $\mathfrak{U}$. Is there always at least one local component of $\mathfrak{U}$ which is a local retractor of $\theory$? For instance, consider the pointed case: Is the local component at the constants of $\mathfrak{U}$ a \gls{lpc} model of $\theory$?
\end{que}

\begin{ex} Without assuming $\theory$ is a local theory (i.e. closed under local implications), \cref{q:local components are local retractors} is obviously false. 

Consider the single sorted language $\lang=\{S,R\}$ where $S,R$ are binary relations, and with $\Dtt=\{=,\dd_1\}$ (with trivial ordered monoid structure). Consider $N$ with universe $\{a,b,c,d\}$ and interpretations $R=\{(a,d),(b,c),(d,a),(c,b)\}$, $S=\{(a,b),(b,a),(c,d),(d,c)\}$, and $\dd_1=\Delta_N\cup\{(a,c),(c,a),(b,d),(d,b)\}$, where $\Delta_N$ is the diagonal in $N$. %(see \cref{image:local retractors inside universal models}).

%\begin{figure}[ht]
%\begin{center}
%\begin{tikzpicture}
%\node (xL) at (-6,-3) {};  \node (yL) at (-6,3) {};  
%\node (xR) at (6,-3) {}; \node (yR) at (6,3) {};
%
%%\draw[step=0.1, gray, very thin] (xL) grid (yR);
%
%\node (a) [draw, shape=circle, fill, scale=0.15, label={above:$a$}] at (-2,2) {}; 
%\node (b) [draw, shape=circle, fill, scale=0.15, label={above:$b$}] at (2,2) {}; 
%\node (c) [draw, shape=circle, fill, scale=0.15, label={below:$c$}] at (-2,-2) {}; 
%\node (d) [draw, shape=circle, fill, scale=0.15, label={below:$d$}] at (2,-2) {}; 
%
%\draw [green, line width=8pt, opacity=0.2, line cap=round] (a) -- (c) node [pos=0.5, text=black!60!green, opacity=1, left] {$\dd_1$};
%\draw [green, line width=8pt, opacity=0.2, line cap=round] (b) -- (d) node [pos=0.5, text=black!60!green, opacity=1, right] {$\dd_1$};
%\draw[blue] (a) -- (d) node [pos=0.5, below] {$R$};
%\draw[blue] (b) -- (c);
%\draw[red] (a) -- (b) node [pos=0.5, above] {$S$};
%\draw[red] (c) -- (d) node [pos=0.5, below] {$S$};
%\end{tikzpicture}
%\caption{} \label{image:local retractors inside universal models}
%\end{center}
%\end{figure}

Then, $N$ has two maximal local substructures --- $\{a,c\}$ and $\{b,d\}$ --- but neither is a \gls{lpc} model of $\Th_-(N)$. In fact, the unique (local) retractor for $N$ is a singleton with both $R$ and $S$ empty.
\end{ex}

\subsection{Automorphisms of retractors} \label{s:automorphisms of retractor} Recall that, unless otherwise stated, we have fixed a local language $\lang$ and a locally satisfiable negative local theory $\theory$. From now on, unless otherwise stated, we also fix a retractor $\retractor$ of $\theory$.

The \emph{pointwise topology} on $\Aut(\retractor)$ is the initial topology from the \gls{lp logic topology} given by the evaluation maps $\ev_a\map \Aut(\retractor)\to \retractor^x$ for $a\in \retractor^x$ with $x$ finite. In other words, a base of open sets is given by the sets 
\[\{\sigma\in\Aut(\retractor)\sth \retractor\models\varphi(\sigma(a))\}\]
where $\varphi(x)\in\LFor_-(\lang(\retractor))$ and $a\in \retractor^x$.

\begin{rmk} \label{r:continuity evaluations} Let $x$ be a variable and $a\in\retractor^x$. Then, $\ev_a\map\Aut(\retractor)\to \retractor^x$ is continuous. For $x$ finite this is precisely the definition of the pointwise topology. For $x$ infinite it follows from the fact that the \gls{lp logic topology} of $\retractor^x$ is the initial topology given by the projections $\proj_{x'}\map \retractor^x\to \retractor^{x'}$ for $x'\subseteq x$ finite.
\end{rmk}

Let $\dd=(\dd_s)_{s\in\Sorts}$ be a choice of locality relations and $a=(a_s)_{s\in\Sorts}$ a choice of points for each sort. The \emph{$\dd${\hyp}ball at $a$ in $\Aut(\retractor)$} is the subset $\dd(\Aut(\retractor),a)\coloneqq \{f\sth f(a_s)\in \dd_s(a_s)\text{ for each }s\in\Sorts\}$. \medskip

We now prove the main result of this section. In the one{\hyp}sort case, \cref{t:automorphism retractor}(\ref{itm:automorphism retractor:compact balls}) was proved in \cite[Lemma 2.6]{hrushovski2022lascar} by Hrushovski. His proof works without modifications in the many sorts case as well. In any case, we recall it here for the sake of completeness. The trivial statements \cref{t:automorphism retractor}(\ref{itm:automorphism retractor:frechet},\ref{itm:automorphism retractor:semitopological group}) were already noted in \cite{hrushovski2022lascar} too. The easy fact \cref{t:automorphism retractor}(\ref{itm:automorphism retractor:closed subspace}) is original. Finally, \cref{t:automorphism retractor}(\ref{itm:automorphism retractor:compact-open}) corresponds to \cite[Remark 3.27]{hrushovski2022definability} in the non{\hyp}local case, but was not discussed in \cite{hrushovski2022lascar}. 

\begin{theo} \label{t:automorphism retractor} The following hold:
\begin{enumerate}[label={\rm{(\arabic*)}}, ref={\rm{\arabic*}}, wide]
\item \label{itm:automorphism retractor:closed subspace} $\Aut(\retractor)$ is a closed subspace of the set of all function on $\retractor$ with the \emph{pointwise topology}, i.e. the initial topology given by the evaluation maps on tuples. 
\item \label{itm:automorphism retractor:compact-open} Let $B$ be a ball in $\retractor^x$ and $F$ a closed subset of $B$ and $U$ an open subset. Then, $\{\sigma \sth \sigma(F)\subseteq U\}$ is open. In particular, if there is an entourage on $\retractor^x$, then $\{\sigma \sth \sigma(F)\subseteq U\}$ is open for every closed compact subset $F$ and open set $U$.
\item \label{itm:automorphism retractor:frechet} $\Aut(\retractor)$ is $\mathrm{T}_1$, i.e. Fr\'{e}chet.
\item \label{itm:automorphism retractor:compact balls} Every ball of $\Aut(\retractor)$ is compact.
\item \label{itm:automorphism retractor:semitopological group} The inversion map $g\mapsto g^{-1}$ is continuous and left and right translations are continuous, i.e. $\Aut(\retractor)$ is a semi{\hyp}topological group.
\end{enumerate}
\end{theo}
\begin{proof}
\begin{enumerate}[label={\rm{(\arabic*)}}, wide]
\item[\hspace{-1.2em}\setcounter{enumi}{1}\theenumi] Note that the set of endomorphisms of $\retractor$ is given by 
\[\mathrm{End}(\retractor)=\bigcap_a \ev_a^{-1}(\qftp_+(a)(\retractor)),\] 
so it is a closed subset in the pointwise topology. Since $\Aut(\retractor)=\mathrm{End}(\retractor)$ by \cref{c:endomorphism of retractor}, we conclude.
\item Let $\underline{F}$ be the partial \gls{lp type} defining $F$ and $\underline{U}$ the set of local negative formulas such that $\neg\underline{U}$ defines $\retractor^x\setminus U$. Assume $\underline{F}$ is closed under finite conjunctions and $\underline{U}$ is closed under finite disjuctions. By compactness of $B$ (\cref{t:compactness of retractors}), $\{\sigma\sth\sigma(F)\subseteq U\}=\bigcup_{\varphi\in \underline{F}, \psi\in\underline{U}}\{\sigma\sth\sigma(\varphi(\retractor))\subseteq \psi(\retractor)\}$. Hence, it suffices to show that $\{\sigma\sth\sigma(\varphi(\retractor))\subseteq \psi(\retractor)\}$ is open for $\psi$ local negative and $\varphi$ \gls{lp formula} with $\varphi(\retractor)\subseteq B$. 

Say $B=\dd(b)$ and $c$ are the parameters of $\varphi$ and $\psi$. Then, $\sigma(\varphi(\retractor,c))\subseteq \psi(\retractor,c)$ if and only if $\retractor\models_\lang \varphi(a,c)$ implies $\retractor\models_\lang \psi(\sigma(a),c)$. Now, since $\sigma\in\Aut(\retractor)$, it follows that $\retractor\models_\lang \psi(\sigma(a),c)$ if and only if $\retractor\models_\lang \psi(a,\sigma^{-1}(c))$. Therefore, $\sigma(\varphi(\retractor,c))\subseteq \psi(\retractor,c)$ if and only if $\retractor\models_\lang\forall x\mathrel{} (\varphi(x,c)\rightarrow \psi(x,\sigma^{-1}(c)))$. Now, since $\varphi(\retractor,c)\subseteq B$, we have that $\forall x\mathrel{} (\varphi(x,c)\rightarrow \psi(x,\sigma^{-1}(c)))$ is equivalent to $\phi(bc,\sigma^{-1}(c))\coloneqq\neg\exists x\in\dd(b)\mathrel{} (\varphi(x,c)\wedge \neg\psi(x,\sigma^{-1}(c)))$. Hence, 
\[\begin{aligned}\{\sigma\sth \sigma(\varphi(\retractor,c))\subseteq \psi(\retractor,c)\}=&\{\sigma \sth \retractor\models_\lang \phi(bc,\sigma^{-1}(c))\}\\ =&\{\sigma\sth\retractor\models_\lang\phi(\sigma(bc),c)\}=\ev^{-1}_{bc}(\phi(\retractor,\retractor,c)).\end{aligned}\]

If there is an entourage on $\retractor^x$, then every compact subset of $\retractor^x$ is contained in some ball. Indeed, if $F$ is compact and $\dd$ is an entourage on $\retractor^x$, as $F\subseteq \bigcup_{a\in F}\dd(a)$, we find finitely many $a_0,\ldots,a_{n-1}\in F$ such that $F\subseteq \bigcup_{i<n}\dd(a_i)$ by compactness of $F$ and \cref{l:entourages}. By the locality axiom \cref{itm:axiom 5}, there are $\dd_i$ for $0<i<n$ such that $a_i\in \dd_i(a_0)$. Hence, we get that $F\subseteq \dd'(a_0)$ for $\dd'\coloneqq\dd\ast \dd_1\ast\cdots\ast \dd_{n-1}$.    
\item Obvious as $\retractor$ is $\mathrm{T}_1$. Indeed, we have $\{f\}=\bigcap_a \ev^{-1}_a(\{f(a)\})$ for any $f\in \Aut(\retractor)$. 
\item Consider a ball $\dd(\Aut(\retractor),a)$ in $\Aut(\retractor)$. It suffices to show that $\dd(\Aut(\retractor),a)$ is closed under ultralimits. 

Pick a sequence $(g_i)_{i\in I}$ in $\dd(\Aut(\retractor),a)$ and $\mathfrak{u}$ an ultrafilter in $I$. Let $\retractor^\mathfrak{u}$ be the corresponding (non{\hyp}local) ultrapower. Consider the diagonal inclusion $\inc\map \retractor\to\retractor^{\mathfrak{u}}$ and $\Dtt(a)$ the local substructure of $\retractor^\mathfrak{u}$ at $a\coloneqq\inc(a)$. Note that $\inc\map \retractor\to\Dtt(a)$. By {\L}o\'{s}'s Theorem, $\retractor^\mathfrak{u}\premodels\theory$. By \cite[Lemma 1.8]{rodriguez2024completeness}, $\Dtt(a)\models_\lang \theory$. Consider the map $g_\mathfrak{u}\map x\mapsto [g_i(x)]_{\mathfrak{u}}$. By {\L}o\'{s}'s Theorem, $g_\mathfrak{u}$ is a homomorphism. As $\retractor\models_\lang \dd(g_i(a),a)$ for all $i\in I$ (and $\retractor$ is local), $g_\mathfrak{u}\map \retractor\to\Dtt(a)$. On the other hand, since $\retractor$ is a retractor, there is $r\map \Dtt(a)\to\retractor$ such that $\id=r\circ\inc$. Set $g\coloneqq r\circ g_\mathfrak{u}\map \retractor\to\retractor$. 

By \cref{c:endomorphism of retractor}, $g\in\Aut(\retractor,a)$. Since $(g_i)_{i\in I}$ is in $\dd(\Aut(\retractor),a)$, we get that $\retractor\models_\lang\dd_s(g_i(a_s),a_s)$ for all $i\in I$ and all sort $s$. Thus, by {\L}o\'{s}'s Theorem, $\retractor^{\mathfrak{u}}\premodels \dd_s(g_\mathfrak{u}(a_s),a_s)$ for all $s$, concluding that $\retractor\models_\lang \dd_s(g(a_s),a_s)$. Therefore, $g\in\dd(\Aut(\retractor),a)$.

Let $\varphi(x)\in\LFor^x_-(\lang(\retractor))$ and $b\in \retractor^x$ such that $\retractor\models_\lang \varphi(g(b))$. Then, $\Dtt(a)\models_\lang \varphi(g_\mathfrak{u}(b))$. By \cite[Lemma 1.8]{rodriguez2024completeness}, as $\varphi$ is local, we conclude that $\retractor^{\mathfrak{u}}\models \varphi(g_{\mathfrak{u}}(b))$. Hence, by {\L}o\'{s}'s Theorem, $\{i\sth \retractor\models_\lang\varphi(g_i(b))\}\in\mathfrak{u}$. Therefore, as $\varphi$ and $b$ are arbitrary, we conclude that $g$ is the $\mathfrak{u}${\hyp}limit of $(g_i)_{i\in I}$. 
\item $\ev_a^{-1}(\varphi(\retractor,b))^{-1}=\{\sigma^{-1}\sth \retractor\models_\lang\varphi(\sigma(a),b)\}=\{\sigma^{-1}\sth\retractor\models_\lang\varphi(a,\sigma^{-1}(b))\}=\ev_b^{-1}(a,\retractor)$. Thus, the inverse map is a homeomorphism. On the other hand, pick $\alpha\in\Aut(\retractor)$. Then, $\ev_a^{-1}(\varphi(\retractor))\circ\alpha=\{\sigma\circ\alpha\sth \retractor\models_\lang \varphi(\sigma(a))\}=\{\sigma\circ\alpha\sth \retractor\models_\lang\varphi(\sigma\circ\alpha(\alpha^{-1}(a)))\}=\ev_{\alpha^{-1}(a)}(\varphi(\retractor))$. Thus, right translations are homeomorphism. As the inverse map and right translations are continuous, left translations are continuous too. \qedhere
\end{enumerate}
\end{proof}
\begin{rmk} By \cref{t:automorphism retractor}(\cref{itm:automorphism retractor:closed subspace}), it seems natural to try to apply Tychonoff{'s} Theorem to prove \cref{t:automorphism retractor}(\cref{itm:automorphism retractor:compact balls}). Unfortunately, the pointwise topology in the space of function of $\retractor$ to itself is not the usual Tychonoff topology, as we are considering valuations under arbitrary finite tuples and not only under single elements. Still, it is obvious from the definition that the space of functions of $\retractor$ to itself with the pointwise topology can be canonically identified with the topological subspace
\[\left\{f\in\prod \retractor^a\sth \text{there is }g\map\retractor\to\retractor\text{ with } f_a=g(a)\text{ for all }a\right\}\]
of $\prod \retractor^a$ with the Tychonoff topology, where $\retractor^a$ is denoting a copy of $\retractor^s$ for each $a\in \retractor^s$ of finite sort $s$. Unfortunately, it seems that the space of functions from $\retractor$ to itself is not a closed subset of $\prod \retractor^a$, as in general $\prod \retractor^a$ is not Hausdorff. \end{rmk}

For the rest of the section, we adapt to the local case several results from \cite{hrushovski2022definability} which were not discussed at all in \cite{hrushovski2022lascar}. We start with \cite[Lemma 3.23]{hrushovski2022definability}.

\begin{lem} \label{l:pointed evaluation are closed} Any evaluation on a pointed tuple is a continuous closed map.
\end{lem}
\begin{proof} Continuity was already noted in \cref{r:continuity evaluations}. Consider a non{\hyp}empty basic closed subset $D$ of $\Aut(\retractor)$, so $D=\ev_b^{-1}\varphi(\retractor)$ where $\varphi(y)$ is a \gls{lp formula} with parameters in $\retractor$ and $b\in \retractor^y$. Pick $a\in\retractor^x$ pointed. We now show that $\ev_a(D)$ is closed. By \cref{c:homogeneity of retractor}, $a'\in \ev_a(D)$ if and only if there is $b'\in\retractor^y$ with $\retractor\models_\lang \varphi(b')$ such that $\ltp_+(a'b')=\ltp_+(ab)$. Write $p(x,y)\coloneqq\ltp_+(ab)$. We have 
\[\ev_a(D)=\{a'\sth \text{there is }y\text{ such that } \varphi(y,c)\wedge p(a',y)\}.\]
 
Consider $\Sigma(x)=\{\exists y\mathrel{} (\varphi(y)\wedge \psi(x,y))\sth \psi\in p\}\subseteq \LFor^x_+(\lang(\retractor))$. We conclude that $\Sigma(\retractor)=\ev_a(D)$. Indeed, if $a'\in \ev_a(D)$, then there is $b'$ such that $\retractor\models_\lang\varphi(b')\wedge p(a',b')$, so, in particular, $a'$ satisfies $\Sigma$. On the other hand, if $a'$ satisfies $\Sigma$, then $\varphi(y)\wedge p(a',y)$ is finitely locally satisfiable. As $a'$ is pointed, $p(a',y)$ implies some bound of $y$ with parameters in $a'$. Thus, $\varphi(y,c)\wedge p(a',y)$ is boundedly satisfiable in $\Th_-(\retractor/\retractor)$, so it is locally satisfiable by \cref{l:boundedly satisfiable}. Hence, by saturation of $\retractor$ (\cref{t:retractors}), we get that $\varphi(y)\wedge p(a',y)$ is realised by some $b'$ in $\retractor$, concluding that $a'\in \ev_a(D)$. In particular, $\ev_a(D)$ is closed. As $D$ is arbitrary, we conclude that the image of every basic closed subset of $\Aut(\retractor)$ is closed.

Let $F$ be an arbitrary closed subset of $\Aut(\retractor)$. Then, $F=\bigcap_{i\in I} D_i$ with $D_i$ basic closed sets. Suppose $a'$ is in the closure of $\ev_a(F)$. Then, $a'$ is in the closure of $\ev_a(\bigcap_{i\in I_0}D_i)$ for all $I_0\subseteq I$. Since the basic closed subsets are closed under finite intersections, we get that $\ev_a(\bigcap_{i\in I_0}D_i)$ is closed for $I_0\subseteq I$ finite. Hence, $\ev^{-1}_a(a')\cap \bigcap_{i\in I_0} D_i\neq\emptyset$ for any $I_0\subseteq I$ finite. Now, pick $\dd=(\dd_s)_{s\in\Sorts}$ such that $\dd(a,a')$. Then, $\ev^{-1}_a(a')\subseteq \dd(\Aut(\retractor),a)$. By \cref{t:automorphism retractor}(\ref{itm:automorphism retractor:compact balls}), we get that $\ev^{-1}_a(a')$ is compact. Therefore, $\ev^{-1}_a(a')\cap F\neq\emptyset$, concluding that $a'\in\ev_a(F)$. As $a'$ is arbitrary, we get  that $\ev_a(F)$ is closed.
\end{proof}

\begin{lem} \label{l:automorphisms are topological group} Suppose $\theory$ has full systems of entourages. Then, $\Aut(\retractor)$ is a Hausdorff topological group and evaluation maps on finite sorts are jointly continuous. 
\end{lem}
\begin{proof} Let us show first that composition in $\Aut(\retractor)$ is continuous. By \cref{t:automorphism retractor}(\ref{itm:automorphism retractor:compact-open},\ref{itm:automorphism retractor:frechet}), we have the topological base
\[\{B(F,U)\sth U\text{ open and }F\text{ closed compact in the same sort of }\retractor\},\]
where $B(F,U)\coloneqq\{\sigma\sth \sigma(F)\subseteq U\}$. Thus, it suffices to show that $\{(f,g)\sth f\circ g\in B(F,U)\}$ is open for every closed compact $F$ and open $U$.

Pick $\alpha\circ\beta\in B(F,U)$. Then, $\beta(F)\subseteq \alpha^{-1}(U)$. Since $\retractor$ is locally closed compact by \cref{c:local compactness retractor}, there are $V$ open and $K$ closed compact such that $\beta(F)\subseteq V\subseteq K\subseteq \alpha^{-1}(U)$. Hence, $(\alpha,\beta)\in B(K,U)\times B(F,V)\subseteq \{(f,g)\sth f\circ g\in B(F,U)\}$. Indeed, if $f\in B(K,U)$ and $g\in B(F,V)$, then $g(F)\subseteq V\subseteq K\subseteq f^{-1}(U)$, concluding that $f\circ g\in B(F,U)$. Since $\alpha,\beta$ and $B(F,U)$ are arbitrary, we conclude that composition is continuous. Since we already know that the inverse map is continuous by \cref{t:automorphism retractor}(\ref{itm:automorphism retractor:semitopological group}), we conclude that $\Aut(\retractor)$ is a topological group. Since $\Aut(\retractor)$ is Fr\'{e}chet by \cref{t:automorphism retractor}(\ref{itm:automorphism retractor:frechet}), it is also Hausdorff by \cite[\S 21 Theorem 4]{husain1966introduction}.

Let us now show that evaluations are jointly continuous. Pick a finite variable $x$, $a\in \retractor^x$, $\sigma\in\Aut(\retractor)$ and $U$ open subset of $\retractor^x$ such that $\sigma(a)\in U$. Then, $a\in \sigma^{-1}(U)$. By local closed compactness of $\retractor^x$, find an open set $V$ and a closed compact set $F$ such that $a\in V\subseteq F\subseteq \sigma^{-1}(U)$. Then, $(\sigma,a)\in B(F,U)\times V\subseteq \{(f,b)\sth f(b)\in U\}$. Indeed, if $f\in B(F,U)$ and $b\in V$, then $b\in V\subseteq F\subseteq f^{-1}(U)$, concluding that $f(b)\in U$. Since $\sigma$, $a$ and $U$ are arbitrary, we conclude that evaluations are jointly continuous.
\end{proof}
  
\begin{lem} \label{l:local compactness automorphisms} Suppose that $\lang$ only has finitely many single sorts and set an enumeration $s$ of all of them. Let $\dd$ be a locality relation of sort $s$. 
\begin{enumerate}[label={\rm{(\arabic*)}}, ref={\rm{\arabic*}}, wide]
\item \label{itm:local compactness automorphisms:weak local compactness} Suppose $(\psi,\varphi)$ is an entourage in $\dd$. Then, $\ev_a^{-1}\psi(\retractor,a)=\{\sigma\sth \retractor\models_\lang\psi(\sigma(a),a)\}$ is a compact neighbourhood of the identity for any $a\in \retractor^s$. In particular, $\Aut(\retractor)$ is weakly locally compact if $\retractor$ has entourages.
\item \label{itm:local compactness automorphisms:local compactness} Suppose $\{(\psi_i,\varphi_i)\}_{i\in I}$ is a full system of entourages in $\dd$. Then, 
\[\left\{\bigcap_{a\in A,\, i\in I_0} \ev_a^{-1}\psi_i(\retractor,a)\sth I_0\subseteq I,\ A\subseteq \retractor^s\text{ both finite}\right\}\] 
is a base of compact neighbourhoods of the identity. In particular, $\Aut(\retractor)$ is locally compact if $\retractor$ has full systems of entourages.
\end{enumerate}
\end{lem}
\begin{proof} 
\begin{enumerate}[wide]
\item[\hspace{-1.2em} (\ref*{itm:local compactness automorphisms:weak local compactness})] By \cref{c:weak local compactness retractor}, $\psi(\retractor,a)\subseteq \dd(a)$ is a closed neighbourhood of $a=\id(a)$. Since $\ev_a$ is continuous, we conclude that $\ev_a^{-1}\psi(\retractor,a)$ is a closed neighbourhood of $\id$. By \cref{t:automorphism retractor}(\ref{itm:automorphism retractor:compact balls}), $\dd(\Aut(\retractor),a)=\ev^{-1}_a\dd(a)$ is compact, concluding that $\ev_a^{-1}\psi(\retractor,a)$ is a closed compact neighbourhood of $\id$. Since translations are continuous, we get that $\Aut(\retractor)$ is weakly locally closed compact. 
\item[(\ref*{itm:local compactness automorphisms:local compactness})] By compactness and the fact that translations are continuous, it suffices to show that $\{\id\}=\bigcap_{a\in\retractor^s,\, i\in I}\ev_a^{-1}\psi_i(\retractor,a)$. This is obvious since $\bigcap_{i\in I}\ev_a^{-1}\psi_i(\retractor,a)=\{\sigma\sth \sigma(a)=a\}$ when $\{\psi_i\}$ a full system of approximations of $=$.\qedhere
\end{enumerate}
\end{proof}

As \cref{l:automorphisms are topological group,l:local compactness automorphisms} show, the existence of full systems of entourages is a very strong hypothesis and, even in the non{\hyp}local context, there is a plenty of examples that do not satisfy it. In its absence, one considers Hausdorff quotients of $\Aut(\retractor)$. We conclude this subsection adapting the study of these Hausdorff quotients from \cite{hrushovski2022definability}. \medskip

Let $x$ be a variable and $X$ a locally positively $\bigwedge_0${\hyp}definable subset of $\retractor^x$. Consider
\[\inftesimal_X\coloneqq \{\sigma\sth \sigma(U)\cap U\neq\emptyset \text{ for every nonempty open subset }U\subseteq X\},\]
In other words, according to \cite[Appendix C]{hrushovski2022definability}, $\inftesimal_X$ is the infinitesimal subgroup of the evaluation action of $\Aut(\retractor)$ on $X$. In the case $X=\retractor^s$, write $\inftesimal_s\coloneqq \inftesimal_{\retractor^s}$. Finally, let $\inftesimal\coloneqq\bigcap\inftesimal_s$ and $\inftesimal_!\coloneqq\bigcap_X\inftesimal_X$. So, according to \cite[Appendix C]{hrushovski2022definability}, $\inftesimal$ is the infinitesimal subgroup of the evaluation action on $\bigsqcup \retractor^s$ and $\inftesimal_!$ is the infinitesimal subgroup of the evaluation action on $\bigsqcup X$.

By \cite[Lemma C.1]{hrushovski2022definability}, we know that $\inftesimal$, $\inftesimal_X$ and $\inftesimal_!$ are closed normal subgroups of $\Aut(\retractor)$. Let $\Gscr_X\coloneqq\rfrac{\Aut(\retractor)}{\inftesimal_X}$, $\Gscr_s\coloneqq\rfrac{\Aut(\retractor)}{\inftesimal_s}$, $\Gscr\coloneqq\rfrac{\Aut(\retractor)}{\inftesimal}$ and $\Gscr_!\coloneqq\rfrac{\Aut(\retractor)}{\inftesimal_!}$. Then, by \cite[Lemma C.1]{hrushovski2022definability}, \cite[Theorem 1.5.3]{arhangelskii2008topological} and \cref{t:automorphism retractor}(\ref{itm:automorphism retractor:semitopological group}), $\Gscr_X$, $\Gscr_s$, $\Gscr$ and $\Gscr_!$ are Hausdorff semitopological groups. Furthermore, by \cite[Theorem 1.5.3]{arhangelskii2008topological}, the quotient homomorphisms from $\Aut(\retractor)$ are continuous and open.
\begin{lem} \label{l:relations between the infinitesimal subgroups} The following hold:
\begin{enumerate}[label={\rm{(\arabic*)}}, ref={\rm{\arabic*}}, wide]
\item \label{itm:relations between infinitesimal subgroups:inclusion} Let $X$ and $Y$ be locally positively $\bigwedge_0${\hyp}definable with $Y\subseteq X$. Then, $\inftesimal_Y\subseteq \inftesimal_X$. In particular, the quotient homomorphism $\pi\map\Gscr_Y\to \Gscr_X$ is a surjective continuous open map.
\item \label{itm:relations between infinitesimal subgroups:projection} Let $X\subseteq \retractor^x$ and $Y\subseteq \retractor^y$ be locally positively $\bigwedge_0${\hyp}definable with $x\subseteq y$. Suppose $\proj(Y)\subseteq X$, where $\proj$ is the restriction of the projection of $\retractor^y$ to $\retractor^x$. Then, $\inftesimal_Y\subseteq\inftesimal_X$. In particular, the quotient homomorphism $\pi\map\Gscr_Y\to \Gscr_X$ is a surjective continuous open map.
\item \label{itm:relations between infinitesimal subgroups:full} $\inftesimal_!\subseteq \inftesimal$. In particular, the quotient homomorphism $\pi\map\Gscr_!\to \Gscr$ is a surjective continuous open map.
\end{enumerate}
\end{lem}
\begin{proof}
\begin{enumerate}[wide]
\item[\hspace{-1.2em}\rm{(\ref*{itm:relations between infinitesimal subgroups:inclusion})}] By definition, the inclusion $\inc\map Y\to X$ is continuous. Also, note that the evaluation action of $\Aut(\retractor)$ commutes with $\inc$, i.e. $\inc(\sigma(a))=\sigma(\inc(a))$. Suppose $\sigma\in\inftesimal_Y$. Pick $U\subseteq X$ open. As $\sigma\in\inftesimal_X$, there is $a\in \sigma(\inc^{-1}(U))\cap\inc^{-1}(U)$. Hence, $\inc(a)\in U$ and $\sigma^{-1}(\inc(a))=\inc(\sigma^{-1}(a))\in U$, concluding $\inc(a)\in \sigma(U)\cap U$. As $U$ and $\sigma$ are arbitrary, $\inftesimal_Y\subseteq \inftesimal_X$. By \cite[Theorem 1.5.3]{arhangelskii2008topological}, the quotient homomorphism $\pi\map\Gscr_Y\to \Gscr_X$ is continuous and open.
\item[\rm{(\ref*{itm:relations between infinitesimal subgroups:projection})}] By \cref{l:product logic topologies}, $\proj$ is continuous. Also, note that the evaluation action of $\Aut(\retractor)$ commutes with $\proj$, i.e. $\proj(\sigma(a))=\sigma(\proj(a))$. Suppose $\sigma\in\inftesimal_Y$. Pick $U\subseteq X$ open. As $\sigma\in\inftesimal_Y$, there is $a\in \sigma(\proj^{-1}(U))\cap\proj^{-1}(U)$. Hence, $\proj(a)\in U$ and $\sigma^{-1}(\proj(a))=\proj(\sigma^{-1}(a))\in U$, concluding $\proj(a)\in \sigma(U)\cap U$. As $U$ and $\sigma$ are arbitrary, $\inftesimal_Y\subseteq \inftesimal_X$. By \cite[Theorem 1.5.3]{arhangelskii2008topological}, the quotient homomorphism $\pi\map\Gscr_Y\to \Gscr_X$ is continuous and open.
\item[\rm{(\ref*{itm:relations between infinitesimal subgroups:full})}] By {\rm{(\ref*{itm:relations between infinitesimal subgroups:inclusion})}}, $\inftesimal_!\subseteq\inftesimal$. By \cite[Theorem 1.5.3]{arhangelskii2008topological}, the quotient homomorphism $\pi\map\Gscr_!\to \Gscr$ is continuous and open. \qedhere
\end{enumerate}
\end{proof}

\begin{lem} Suppose that $\lang$ only has finitely many single sorts and $\theory$ has entourages. Then, $\Gscr_X$ is a locally compact Hausdorff topological group for any locally positively $\bigwedge_0${\hyp}definable set $X$. Similarly, $\Gscr$ and $\Gscr_!$ are locally compact Hausdorff topological groups too.
\end{lem}
\begin{proof} By \cref{l:local compactness automorphisms}(\ref{itm:local compactness automorphisms:weak local compactness}), $\Aut(\retractor)$ is weakly compact. Then, since $\pi\map\Aut(\retractor)\to \Gscr_X$ is continuous and open, we get that $\Gscr_X$ is weakly locally compact. Hence, by Hausdorffness, it is locally compact. As it is a semitopological group, by Ellis' Theorem \cite[Theorem 2]{ellis1957locally}, we conclude that it is a locally compact Hausdorff topological group. Similar for $\Gscr_!$ and $\Gscr$.
\end{proof}

Let $X$ be a locally positively $\bigwedge_0${\hyp}definable subset. Let $\rfrac{X}{\inftesimal_X}$ denote the quotient space of $\inftesimal_X${\hyp}orbits of $X$. In \cite{hrushovski2022definability}, it was proved that this quotient space is Hausdorff and, if $X$ is defined by a \gls{lp type}, the $\inftesimal_X${\hyp}conjugacy relation in $X$ is locally positively $\bigwedge_0${\hyp}definable --- see \cite[Propositions 3.24]{hrushovski2022definability}. We adapt here these results.

\begin{lem} \label{l:orbits space} Let $X$ be a locally positively $\bigwedge_0${\hyp}definable subset of $\retractor^x$ on a pointed variable $x$. Write $\underline{X}(x)$ for a partial \gls{lp type} without parameters defining it. Then:
\begin{enumerate}[label={\rm{(\arabic*)}}, ref={\rm{\arabic*}}, wide]
\item \label{itm:orbits space:Hausdorffness} $\rfrac{X}{\inftesimal_X}=\rfrac{X}{\inftesimal_!}$ is Hausdorff.
\item \label{itm:orbits space:orbit equivalence} Suppose $\underline{X}$ is a \gls{lp type}. Then, the $\inftesimal_X${\hyp}conjugacy equivalence relation on $X$ is locally positively $\bigwedge_0${\hyp}definable.
\end{enumerate}
\end{lem}
\begin{proof} 
\begin{enumerate}[wide]
\item[\hspace{-1.2em}{\rm{(\ref*{itm:orbits space:Hausdorffness})}}] By \cite[Lemma C.1]{hrushovski2022definability}, $\Gscr_!$ and $\Gscr_X$ are Hausdorff. By \cref{l:pointed evaluation are closed}, $\Aut(\retractor)$ acts by closed maps. By \cite[Remark C.3]{hrushovski2022definability}, it follows that $\rfrac{X}{\inftesimal_!}=\rfrac{X}{\inftesimal_X}$. By \cite[Lemma C.2(1)]{hrushovski2022definability}, $\rfrac{X}{\inftesimal_X}$ is Hausdorff. 
\item[{\rm{(\ref*{itm:orbits space:orbit equivalence})}}] By (\ref*{itm:orbits space:Hausdorffness}), $\rfrac{X}{\inftesimal_X}$ is Hausdorff, so the diagonal is closed. By \cref{l:product logic topologies}, we conclude that the $\inftesimal_X${\hyp}conjugacy equivalence relation $E$ is closed in $X\times X$. Let $\Sigma(x,y,c)$ be a partial \gls{lp type} with parameters $c$ defining $\inftesimal_X${\hyp}conjugacy in $X$. Without loss of generality, take $c$ pointed. Then, $\Sigma(a,b,c)$ if and only if $\Sigma(a,b,c')$ for any $c'$ such that $\ltp_+(c)=\ltp_+(c')$. One direction is obvious. On the other hand, by \cref{c:homogeneity of retractor}, since $\ltp_+(c)=\ltp(c')$, there is $\sigma\in\Aut(\retractor)$ such that $\sigma(c')=c$. Then, as $\Sigma(a,b,c')$, we get $\Sigma(\sigma(a),\sigma(b),c)$, so there is $g\in \inftesimal_X$ with $g\sigma(a)=\sigma(b)$, so $\sigma^{-1}g\sigma(a)=b$. By normality of $\inftesimal_X$, we get $\sigma^{-1}g\sigma\in\inftesimal_X$, so $\Sigma(a,b,c)$ by definition of $E$. 

Let $p(z)\coloneqq\ltp_+(c)$. Picking any $a\in X$, by \cref{itm:axiom 5}, there is some bound $\Bound(x,z)$ such that $\retractor\models_\lang \Bound(a,c)$. Then, by \cref{c:homogeneity of retractor}, for any $a'\in X$, there is $c'\in p(\retractor)$ with $\retractor\models \Bound(a',c')$. Hence, we conclude that $\inftesimal_X${\hyp}conjugacy is defined by 
\[\begin{aligned} \underline{E}(x,y)&\coloneqq \exists z\mathrel{} (\Sigma(x,y,z)\wedge p(z)\wedge \Bound(x,z))\\ &\coloneqq \left\{\exists z\mathrel{} (\varphi(x,y,z)\wedge \psi(z)\wedge \bigwedge B_0(x,z))\sth\varphi\in \Sigma,\ \psi\in p,\ B_0\subseteq B\text{ finite}\right\}.\end{aligned}\] 
Indeed, if $a$ and $b$ are $\inftesimal_X${\hyp}conjugates, then $\retractor\models_\lang \Sigma(a,b,c')$ for any $c'\in p(\retractor)$. Now, there is $c'\in p(\retractor)$ such that $\retractor\models_\lang \Bound(a,c')$, so we get that $\retractor\models_\lang \underline{E}(a,b)$ witnessed by $z=c'$. On the other hand, suppose $a$ and $b$ satisfy $\underline{E}$. Then, the partial \gls{lp type} $\Sigma(a,b,z)\wedge p(z)\wedge \Bound(a,z)$ is finitely locally satisfiable. Now, $\Bound(a,z)$ is a bound, so we get by \cref{l:boundedly satisfiable}, that there is $c'$ realising it. In other words, $\ltp_+(c')=\ltp_+(c)$ with $\Sigma(a,b,c')$, so $a$ and $b$ are $\inftesimal_X${\hyp}conjugates.

Therefore, we conclude that the $\inftesimal_X${\hyp}conjugacy equivalence relation is locally positively $\bigwedge_0${\hyp}definable. \qedhere 
\end{enumerate}
\end{proof}

We finish with \cite[Proposition 3.25 and Corollary 3.26]{hrushovski2022definability}. 

\begin{lem} \label{l:orbits space dense subset}  Let $p(x)$ be a \gls{lp type} without parameters of $\retractor$ on a pointed variable $x$. Slightly abusing of the notation, denote by $p$ the set of realisations $p(\retractor)$. Then, $\rfrac{p}{\inftesimal_!}$ has a dense subset of size $|\lang|^{|x|}$. In particular, $\left|\Aut\left(\rfrac{p}{\inftesimal_!}\right)\right|\leq 2^{|\lang|^{|x|}}$.
\end{lem} 
\begin{proof} Pick $c\in p$ arbitrary. By \cref{itm:axiom 5}, $p=p\cap \Dtt(c)$. As there are at most $|\lang|^{|x|}$ balls centred at $c$, it suffices to show that $\rfrac{\dd(p,c)}{\inftesimal_!}$ has a dense subset of size $|\lang|$ for any ball $\dd(c)$, where $\dd(p,c)\coloneqq p\cap \dd(c)$. 

Here, we are interested in $\rfrac{\dd(p,c)}{\inftesimal_!}$ with the subspace topology from $\rfrac{p}{\inftesimal_!}$ with the quotient topology from the \gls{lp logic topology} of $p$. By \cref{l:orbits space}(\cref{itm:orbits space:Hausdorffness}), $\rfrac{p}{\inftesimal_!}$ is Hausdorff, so $\rfrac{\dd(p,c)}{\inftesimal_!}$ is Hausdorff with the subspace topology. On the other hand, we can consider the quotient topology in $\rfrac{\dd(p,c)}{\inftesimal_!}$  from the \gls{lp logic topology} of $\dd(p,c)$. As $\dd(p,c)$ is compact in the \gls{lp logic topology} (\cref{t:compactness of retractors}), we get that $\rfrac{\dd(p,c)}{\inftesimal_!}$ is compact with the quotient topology. Now, by the initial and final properties of the quotient topology and the subspace topology respectively, we have that the identity map in $\rfrac{\dd(p,c)}{\inftesimal_!}$ is continuous from the quotient topology to the subspace topology. Consequently, both topologies agree and $\rfrac{\dd(p,c)}{\inftesimal_!}$ is a compact Hausdorff topological space. 

Then, we have that $\rfrac{\dd(p,c)}{\inftesimal_!}\times \rfrac{\dd(p,c)}{\inftesimal_!}$ with the product topology is a compact Hausdorff topological space. On the other hand, we can consider the space of orbits $\rfrac{\dd(p,c)\times\dd(p,c)}{\inftesimal_!\times\inftesimal_!}$ with the coordinate{\hyp}wise action. In this space, we get the quotient topology from the \gls{lp logic topology} in $\dd(p,c)\times\dd(p,c)$. By \cref{t:compactness of retractors} again, this is a compact topology. Now, we have the natural bijection between $\rfrac{\dd(p,c)\times\dd(p,c)}{\inftesimal_!\times\inftesimal_!}$ and $\rfrac{\dd(p,c)}{\inftesimal_!}\times \rfrac{\dd(p,c)}{\inftesimal_!}$ given by $[a,b]_{\inftesimal_!\times\inftesimal_!}\mapsto ([a]_{\inftesimal_!},[b]_{\inftesimal_!})$. By \cref{l:product logic topologies} and the final and initial properties of the quotient topology and the product topology respectively, this map is also continuous. Consequently, both topologies agree. 
 
Let $E$ be the equivalence relation of $\inftesimal_!${\hyp}conjugacy restricted to $\dd(p,c)$. By \cref{l:orbits space}(\ref{itm:orbits space:orbit equivalence}), $E$ is locally positively $\bigwedge_c${\hyp}definable. Thus, by \cite[Lemma 2.12]{rodriguez2024completeness}, as $\retractor\models^\pc_\lang\theory$, if $(a,b)\notin E$, there is a locally positively $c${\hyp}definable subset $D$ such that $(a,b)\in D$ and $D\cap E=\emptyset$. Since there are at most $|\lang|$ \glspl{lp formula} over $c$, we have a family $\{D_i\}_{i<|\lang|}$ of locally positively $c${\hyp}definable subsets disjoint to $E$ such that, for any $(a,b)\notin E$, there is $i<|\lang|$ with $(a,b)\in D_i$. 

For each $i<|\lang|$, consider
\[\widetilde{D}_i\coloneqq \{(a',b') \sth \text{there are }(a,b)\in D_i\text{ with }(a,a'),(b,b')\in E\}.\]
Let $\varphi_i$ be a \gls{lp formula} over $c$ defining $D_i$ and $\Psi$ be a partial \gls{lp type} without parameters closed under finite conjunctions defining $E$. Then, $\widetilde{D}_i$ is defined over $c$ by the partial \gls{lp type}
\[\Phi_i(x,y)\coloneqq\{\exists x' y'\mathrel{} (\varphi_i(x',y')\wedge \psi(x,x')\wedge \psi(y,y')\wedge \dd(x',c)\wedge \dd(y',c))\sth \psi\in \Psi\}.\]
Indeed, obviously $\widetilde{D}_i\subseteq \Phi_i(\retractor)$. On the other hand, if $\retractor\models_\lang \Phi_i(a,b)$, then $\{\varphi_i(x',y')\wedge \psi(a,x')\wedge \psi(b,y')\wedge \dd(x',c)\wedge \dd(y',c)\sth \psi\in \Psi\}$ is finitely satisfiable. As it implies a bound, we conclude that it is boundedly satisfiable, so there are $(a',b')\in D_i$ with $(a,a'),(b,b')\in E$ by saturation of $\retractor$. In other words, $(a,b)\in \widetilde{D}_i$.

Since $E$ is an equivalence relation and $D_i\cap E=\emptyset$, we conclude that $\widetilde{D}_i\cap E=\emptyset$. Also, note that $\widetilde{D}_i=\bar{\pi}^{-1}\bar{\pi}(D_i)$ where $\bar{\pi}=(\pi,\pi)\map \dd(p,c)\times\dd(p,c)\to \rfrac{\dd(p,c)}{\inftesimal_!}\times\rfrac{\dd(p,c)}{\inftesimal_!}$ is the quotient map --- i.e. $\pi\map \dd(p,c)\to \rfrac{\dd(p,c)}{\inftesimal_!}$ is the quotient map. Hence, $\bar{\pi}^{-1}\bar{\pi}(\widetilde{D}^c_i)=\widetilde{D}^c_i$ where $\widetilde{D}^c_i\coloneqq (\dd(p,c)\times\dd(p,c))\setminus \widetilde{D}_i$, so $\bar{\pi}(\widetilde{D}^c_i)$ is open in the quotient topology, thus also in the product topology as both agree. 

As $\bar{\pi}^{-1}\bar{\pi}(\widetilde{D}^c_i)=\widetilde{D}^c_i$, we get that $\bar{\pi}(E)=\bigcap\bar{\pi}(\widetilde{D}^c_i)$. Now, $\bar{\pi}(E)$ is the diagonal, so we have concluded that the diagonal in $\rfrac{\dd(p,c)}{\inftesimal_!}$ is the intersection of $|\lang|$ many open subsets in the product topology. As $\rfrac{\dd(p,c)}{\inftesimal_!}$ is a compact Hausdorff topological space whose diagonal is the intersection of $|\lang|$ open sets in the product topology, we conclude that there is a dense subset of $\rfrac{\dd(p,c)}{\inftesimal_!}$ of size $|\lang|$. % Indeed, say $X$ is a compact hausdorff topological space and $\Delta=\bigcap_{i<\lambda} B_i$ where $\Delta$ is the diagonal and $B_i$ open in the product topology. By compactness and Hausdorffness, we can find $K_i$ compact neighbourhood of $\Delta$ with $B_i\subseteq K_i$. Write $\Delta=\bigcap_{i<\lambda} K_i$ with $K_i$ compact neighbourhood of the diagonal. Without loss of generality, we may assume that the family $\{K_i\}_{i<\lambda}$ is closed under finite intersections. Taking $K_i\cap K_i^{-1}$, we can also assume that $K_i$ is symmetric, i.e. $(x,y)\in K_i\Leftrightarrow (y,x)\in K_i$. Now, for each $i$, for any $x\in X$, $K_i(x)=\{y\sth (x,y)\in K_i\}$ is a neighbourhood of $x$. Thus, by compactness, for each $i$, there is a finite covering $X=K_i(x_0)\cup\cdots\cup K_i(x_n)$. Take $D_i=\{x_0,\ldots,x_n\}$ and $D=\bigcup D_i$. We claim that $D$ is dense. Indeed, take a non{\hyp}empty open subset $U$. Pick $b\in U$. Then, as $\Delta=\bigcap K_i$, we have that $\{b\}=\bigcap K_i(b)$. Therefore, $\bigcap K_i(b)\subseteq U$. Now, $K_i$ is compact, so $K_i\cap \{b\}\times X$ is compact, so $K_i(b)$ is compact. By Hausdorffness, $K_i(b)$ is closed. Then, by compactness, there is a finite subset $I$ such that $\bigcap_{i\in I}K_i(b)\subseteq U$. As we assume that $\{K_i\}_{i<\lambda}$ is closed under finite intersections, there is $i$ such that $K_i(b)\subseteq U$. Now, by choice of $D_i$, there is $x\in D_i\subseteq D$ such that $b\in K_i(x)$, so $x\in K_i(b)\subseteq U$ (as we assume that $K_i$ is symmetric). In conclusion, $U\cap D\neq \emptyset$. As $U$ is arbitrary, $D$ is dense.  

As $\rfrac{p}{\inftesimal_!}$ is Hausdorff, any automorphism fixing a dense subset is the identity. Hence, as $\rfrac{p}{\inftesimal_!}$ has a dense subset of size $|\lang|^{|x|}$, we conclude $\left|\Aut\left(\rfrac{p}{\inftesimal_!}\right)\right|\leq 2^{|\lang|^{|x|}}$.
\end{proof}

\begin{coro} \label{c:size of Hausdorff quotient} Let $p(x)$ be a \gls{lp type} without parameters of $\retractor$ on a pointed variable. Then, $|\Gscr_p|\leq 2^{|\lang|^{|x|}}$. %In particular, $|\Gscr_!|\leq 2^{|\lang|^{\lambda}}$ where $\lambda$ is the number of sorts of $\lang$. 
\end{coro}
\begin{proof} Slightly abusing notation, let us keep writing $p$ for the set of realisations $p(\retractor)$. Consider the group homomorphism $f\map \Aut(\retractor)\to \Aut\left(\rfrac{p}{\inftesimal_p}\right)$ given by $f(\sigma)([a])=[\sigma(a)]$, where $[\tbullet]$ denotes the $\inftesimal_p${\hyp}conjugacy class. By \cref{l:orbits space}(\ref{itm:orbits space:orbit equivalence}), $\inftesimal_p${\hyp}conjugacy in $p$ is locally positively $\bigwedge_0${\hyp}definable, so $[\sigma(a)]=[\sigma(b)]$ when $[a]=[b]$. In other words, $f$ is well{\hyp}defined. 

By \cref{l:orbits space dense subset}, $\left|\Aut\left(\rfrac{p}{\inftesimal_p}\right)\right|\leq 2^{|\lang|^{|x|}}$. Thus, by the isomorphism theorem, it suffices to show that $\ker(f)=\inftesimal_p$. Trivially, $\inftesimal_p\subseteq \ker(f)$. 

Pick $\sigma\in \ker(f)$ and an arbitrary open subset $U\subseteq p$ in the \gls{lp logic topology}. By \cite[Appendix C, eq. 2]{hrushovski2022definability}, we need to show that $\sigma(U)\subseteq \cl(U)$, where $\cl(U)$ is the closure of $U$ in the \gls{lp logic topology}. Let $a\in U$ arbitrary. Since $\sigma\in \ker(f)$, $\sigma(a)$ and $a$ are $\inftesimal_p${\hyp}conjugates. Thus, there is $\sigma'\in \inftesimal_p$ such that $\sigma(a)=\sigma'(a)$. Hence, $\sigma(a)=\sigma'(a)\in \sigma'(U)\subseteq \cl(U)$. As $a\in U$ is arbitrary, $\sigma(U)\subseteq \cl(U)$.
\end{proof}

\printglossaries

\bibliographystyle{alphaurl}
\bibliography{Retractors_in_local_positive_logic}
%\printbibliography[heading=bibintoc]
%\begin{thebibliography}{0}
%\bibitem{arhangelskii2008topological}
%%Alexander Arhangel'skii and Mikhail Tkachenko
% Arhangel'skii, A. \& Tkachenko, M. (2008).\newblock \textsl{Topological groups and related structures},\newblock Atlantis Studies in Mathematics 1.\newblock Atlantis Press, Paris; World Scientific Publishing Co. Pte. Ltd., Hackensack, NJ}.\newblock ISBN:9789078677062.
%
%\bibitem{ellis1957locally}
%%Robert Ellis
%Ellis, R. (1957).\newblock Locally compact transformation groups.\newblock \textsl{Duke Mathematical Journal}, 24(2), 119--125.\newblock \href{https://doi.org/10.1215/S0012-7094-57-02417-1}{DOI:10.1215/S0012-7094-57-02417-1}.
%
%\bibitem{hrushovski2022lascar}
%%Ehud Hrushovski
%Hrushovski, E. (2022).\newblock \textsl{Beyond the Lascar Group}.\newblock \href{https://arxiv.org/abs/2011.12009v3}{arXiv:2011.12009v3}.
%
%\bibitem{hrushovski2022definability}
%%Ehud Hrushovski
%Hrushovski, E. (2022).\newblock \textsl{Definability patterns and their symmetries}.\newblock \href{https://arxiv.org/abs/1911.01129v4}{arXiv:1911.01129v4}.
%
%\bibitem{husain1966introduction}
%%Taqdir Husain
%Husain, T. (1966).\newblock \textsl{Introduction to topological groups}.\newblock W. B. Saunders Co., Philadelphia, Pa.-London.\newblock ISBN:9780721648552.
%
%\bibitem{rodriguez2024completeness}
%%Arturo Rodriguez Fanlo \& Ori Segel
%Rodriguez Fanlo, A. \& Segel, O. (2024). \newblock \textsl{Completeness in local positive logic}. \newblock \href{https://arxiv.org/abs/2401.03260v2}{arXiv:2401.03260v2}.
%
%\bibitem{segel2022positive},
%%Ori Segel
%Segel, O. (2022).\newblock \textsl{Positive Definability Patterns}.\newblock \href{https://arxiv.org/abs/2207.12449v2}{arXiv:2207.12449v2}.
%\end{thebibliography}
\end{document}